\documentclass[letterpaper]{article}
\usepackage[margin=1in]{geometry}
\usepackage{float}
\usepackage{multirow} 
\usepackage{bm}
\usepackage{amsmath,amsthm,amsfonts,amssymb}
\usepackage{enumerate,xcolor}
\usepackage{graphicx}
\usepackage{subcaption}
\usepackage[authoryear,round]{natbib}
\usepackage[ruled,vlined]{algorithm2e}
\usepackage{booktabs,tabularx}
\usepackage{placeins}
\usepackage{needspace}

\PassOptionsToPackage{hyphens}{url}
\usepackage{url} 
\usepackage{hyperref}
\hypersetup{
  pdftitle  = {Rough Heston Model as the Scaling Limit of Bivariate Cumulative Heavy-Tailed INAR Processes: Weak-Error Bounds and Option Pricing},
  pdfauthor = {Yingli Wang and Zhenyu Cui and Lingjiong Zhu},
  colorlinks = true,
  linkcolor  = [rgb]{0.10,0.20,0.55},
  citecolor  = [rgb]{0.10,0.20,0.55},
  urlcolor   = [rgb]{0.10,0.20,0.55}
}

\usepackage{booktabs}
\usepackage{array}

\allowdisplaybreaks

\numberwithin{equation}{section}
\newtheorem{theorem}{Theorem}[section]
\newtheorem{lemma}[theorem]{Lemma}
\newtheorem{proposition}[theorem]{Proposition}
\newtheorem{corollary}[theorem]{Corollary}

\newtheorem{assumption}[theorem]{Assumption}

\theoremstyle{definition}
\newtheorem{definition}[theorem]{Definition}

\theoremstyle{remark}
\newtheorem{remark}[theorem]{Remark}

\newcommand{\R}{\mathbb{R}}
\newcommand{\E}{\mathbb{E}}
\newcommand{\C}{\mathbb{C}}

\newcommand{\ii}{\mathrm{i}}
\newcommand{\Var}{\mathrm{Var}}

\title{Rough Heston Model as the Scaling Limit of Bivariate Cumulative Heavy-Tailed INAR Processes: Weak-Error Bounds and Option Pricing\thanks{All code (including FFT acceleration and option pricers) is available at \protect\url{https://github.com/gagawjbytw/INAR-rough-Heston}.}}

\author{Yingli Wang\thanks{School of Mathematical Sciences, Fudan University, Shanghai, People's Republic of China (\href{mailto:yingliwang@fudan.edu.cn}{\nolinkurl{yingliwang@fudan.edu.cn}}).}
\and Zhenyu Cui\thanks{School of Business, Stevens Institute of Technology,
Hoboken, NJ 07030, USA (\href{mailto:zcui6@stevens.edu}{\nolinkurl{zcui6@stevens.edu}}).}
\and Lingjiong Zhu\thanks{Corresponding author. Department of Mathematics, Florida State University, Tallahassee, FL 32306, USA (\href{mailto:zhu@math.fsu.edu}{\nolinkurl{zhu@math.fsu.edu}}).}}

\date{\today}

\begin{document}

\maketitle

\begin{abstract}
We study nearly unstable bivariate cumulative heavy-tailed INAR($\infty$) processes and show that, under a one-factor parameterization and a suitable scaling, they converge to the rough Heston model. This yields a discrete-time microstructural route to the joint price-variance dynamics and gives explicit formulas linking the INAR asymmetry parameters to the leverage correlation and diffusion scale of the limiting volatility process.
For pricing, we reduce the exact finite-$\tau$ Fourier--Laplace transform recursion to a quadratic discrete Volterra equation and compare it with the rough Heston fractional Riccati equation. Under a transform-strip condition for the continuous model and bounded-Riccati regularity, every fixed bounded frequency window admits a uniform transform-error bound of order $\tau^{-\alpha}+\tau^{-(1-\alpha)}$. The finite-$\tau$ transform is entire, and a uniform exponential-moment estimate then removes the Fourier cutoff and yields convergence of the exact, untruncated call prices.
We also develop an FFT-accelerated CDQ simulator with $\mathcal O(\tau\log^2\tau)$ complexity per path and use it to price European and path-dependent options, benchmark the Markovian case $\alpha=1$, and analyze the model-implied volatility surface.
\end{abstract}





\section{Introduction}

The rough Heston model of \cite{el2019characteristic} is a stochastic-volatility model with discounted asset price and variance dynamics
\begin{align*}
&dS_t = S_t\sqrt{V_t}\,dW_t,
\\
&V_t = V_0 + \frac{1}{\Gamma(\alpha)}\int_0^t (t-s)^{\alpha-1}\gamma(\theta - V_s)\,ds
      + \frac{1}{\Gamma(\alpha)}\int_0^t (t-s)^{\alpha-1}\gamma\nu \sqrt{V_s}\,dB_s,
\end{align*}
where $\gamma,\theta,\nu,V_0$ are positive constants, $(W_{t})_{t\geq 0}$ and $(B_{t})_{t\geq 0}$ are standard Brownian motions with correlation $\rho$, and $\alpha\in(1/2,1)$ governs the roughness of the volatility path. It may be viewed as a rough-volatility extension of the classical Heston model \cite{heston1993closed}, motivated by empirical evidence that volatility is rough \cite{gatheral2018volatility}. Early pricing evidence under rough volatility and its implications for short-maturity skew appear in \cite{BayerFrizGatheral2015PricingRough}, while \cite{KellerResselMajid2020Comparison} clarify how rough and classical Heston specifications affect the volatility surface.

A central theoretical breakthrough is due to \cite{el2019characteristic}, who connect nearly unstable Hawkes processes to fractional volatility models and derive the characteristic function of the rough Heston log-price through a fractional Riccati equation. Their work builds on the scaling-limit theory for nearly unstable Hawkes processes developed in \cite{jaisson2015limit,jaisson2016rough}. This Hawkes-based route provides a continuous-time microstructural foundation for rough Heston. A natural question is whether one can obtain a parallel foundation in discrete time.

A convenient discrete framework is provided by INAR($\infty$) processes with Poisson thinning. If $(X_n)$ denotes the count process, then
\[
\left.X_n\right|\mathcal F_{n-1}\sim \mathrm{Poisson}(\lambda_n),\qquad
\lambda_n=\mu+\sum_{s=1}^{n-1}\alpha_{n-s}X_s,
\]
and the cumulative process $N_n=\sum_{s=1}^n X_s$ may be viewed as a discrete Hawkes process.

Hawkes processes were introduced by \cite{hawkes1971spectra} and have since become a standard model for self-exciting event data. Functional limit theorems for Hawkes processes in various regimes are developed, for example, in \cite{bacry2013some,horst2026functional}. Their discrete-time analogs, cumulative INAR($\infty$) processes, provide an equally natural framework for count data observed on a time grid. The structural relation between INAR($\infty$) and Hawkes processes is well known: \cite{kirchner2016hawkes} develops the INAR($\infty$) framework and proves a discrete-to-continuous approximation theorem toward Hawkes processes. For broader discussions of Poisson autoregressive models and their links with INAR and Hawkes processes, see \cite{fokianos2021multivariate,huang2023nonlinear}. Discrete Hawkes models have also appeared in applications; see, for example, \cite{xu2022self}.

For rough-volatility scaling limits in discrete time, the key precursor is \cite{wang2026scaling}. They show that a heavy-tailed nearly unstable cumulative INAR($\infty$) process converges to a fractional CIR process by adapting the program of \cite{jaisson2016rough}. In particular, they consider kernels with heavy tail
\[
  \alpha_n\sim \frac{K}{n^{1+\alpha}},
  \qquad n\to\infty,
\]
for $\alpha\in(1/2,1)$, and identify the rough variance dynamics in a univariate setting.

This paper extends that discrete scaling picture from the univariate volatility setting to a bivariate price--variance framework. We study nearly unstable heavy-tailed bivariate cumulative INAR($\infty$) processes, prove that they converge to the full rough Heston model, and derive the associated microscopic-to-macroscopic parameter mapping. In this way, we obtain a discrete-time microstructural foundation for both volatility and price, complementing the continuous-time Hawkes-based viewpoint in \cite{el2018microstructural,horst2022microstructure}.

Beyond the scaling limit, we derive a second-order discrete Volterra reduction of the exact INAR transform recursion. Under the continuous transform-strip and bounded-Riccati assumptions stated in Section~\ref{sec:weak-euro}, bounded-window Riccati comparison gives the transform-error scale $\tau^{-\alpha}+\tau^{-(1-\alpha)}$, while a uniform Carr--Madan tail estimate yields convergence of the full call prices.

On the numerical side, simulation of rough Heston has been studied extensively. Existing approaches include Euler-type and strong approximations \cite{liang2017strong,RICHARD2021109,richard2023discrete}, multifactor approximations \cite{abi2019multifactor}, fast Monte Carlo and Fourier-based methods \cite{ma2022fast,callegaro2021fast,boyarchenko2025fast}, CTMC approximations \cite{yang2024}, weak simulation schemes tailored to rough Heston \cite{BayerBreneis2024WeakRoughHeston}, and related semimartingale approximations for control problems \cite{ma2023optimal}. Recent weak-error analyses also cover rough-volatility models driven explicitly by fractional Brownian motion; see \cite{Gassiat2023WeakRoughVol,BayerHallTempone2021WeakLinearRoughVol,FrizSalkeldWagenhofer2024WeakRoughVol}. Here we develop a discrete microstructural approximation and a single simulation framework for European and path-dependent contracts.

The main contributions of the paper are as follows.
\begin{enumerate}
  \item \textit{A bivariate INAR microstructural limit for the rough Heston model.} We prove the joint convergence of a nearly unstable heavy-tailed bivariate cumulative INAR($\infty$) model to the rough Heston dynamics; see Theorem~\ref{theorem1} and Corollary~\ref{maintheorem}. This yields a discrete-time microstructural foundation for both price and variance, together with an explicit mapping from the INAR parameters to the rough Heston coefficients and leverage correlation.

  \item \textit{A transform-based weak-error analysis and full call-price convergence.} Under the continuous transform-strip and bounded-Riccati assumptions of Section~\ref{sec:weak-euro}, the discrete Volterra reduction gives the bounded-window transform-error scale $\tau^{-\alpha}+\tau^{-(1-\alpha)}$, while a uniform $\mathcal O(R^{-1})$ tail estimate yields full-call convergence, locally uniformly in log-moneyness.

  \item \textit{FFT-accelerated simulation.} We build a CDQ--FFT implementation of the INAR($\infty$) recursion with $\mathcal O(\tau\log^2\tau)$ cost per path and use it for European and path-dependent contracts. The lag-one kernel from Remark~\ref{rem:alpha-one-kernel} also provides a classical Heston benchmark.

  \item \textit{A unified numerical study across payoff classes.} We benchmark the scheme on European options, illustrate its use for Asian, lookback, and barrier contracts, compare implied-volatility slices with several recent rough Heston schemes, and examine the model-implied volatility surface.
\end{enumerate}

\section{From INAR\texorpdfstring{$(\infty)$}{(infty)} Processes to the Rough Heston Model}\label{sec:INAR_to_rH}


\subsection{A discrete-time microstructural route to the rough Heston model}
For each admissible integer $\tau$ specified below, let
$X_n^{\tau,\pm}:=N_n^{\tau,\pm}-N_{n-1}^{\tau,\pm}$ denote the increments of
a bivariate cumulative INAR($\infty$) process. We use the zero-history
convention $N_0^{\tau,\pm}=0$ and $X_n^{\tau,\pm}=0$ for $n\le0$. Conditionally on
$\mathcal F_{n-1}$, the two increments are independent Poisson variables with
respective means $\lambda_n^{\tau,+}$ and $\lambda_n^{\tau,-}$, where
\[
 \boldsymbol\lambda_n^\tau
 =
 \hat\mu_\tau(n)\begin{pmatrix}1\\1\end{pmatrix}
 +\sum_{s=1}^{n-1}K_{n-s}^\tau
   \begin{pmatrix}X_s^{\tau,+}\\X_s^{\tau,-}\end{pmatrix}.
\]
We work throughout with the following one-factor specification.  Its equal
matrix rows balance the two conditional intensities, $\beta>1$ produces
negative leverage in the limit, and the scalar kernel supplies the heavy tail.

\begin{definition}[Model Parameterization]\label{def:params}
Let $\beta>1$, $1/2<\alpha<1$, $\gamma>0$, $\xi_0>0$ and $\mu>0$ be
constants, and set
\[
  \tau_{\min}:=\left\lfloor\gamma^{1/\alpha}\right\rfloor+1.
\]
The parameters of the INAR($\infty$) process sequence are specified as
follows for each integer $\tau\ge\tau_{\min}$:
\begin{itemize}
    \item \textit{Decay and Intensity Parameters:} The decay parameter $a_\tau$ and the exogenous order intensity $\mu_\tau$ are set to:
    \begin{equation}\label{a:tau:mu:tau}
        a_\tau:=1-\gamma \tau^{-\alpha}, \quad \text{and} \quad \mu_{\tau}:=\mu\tau^{\alpha-1}.
    \end{equation}
    In particular, $a_\tau\in(0,1)$ for every admissible $\tau$.
    
    \item \textit{Influence Kernels:} The influence matrix kernel $K_n^\tau$ is given by the product of a scalar kernel $\varphi_n^\tau$ and a constant matrix $\chi$:
    \[
        K_{n}^{\tau}:=\varphi_n^{\tau} \chi, \quad \text{where} \quad \chi:=\frac1{\beta+1}\left(
        \begin{matrix}
          1 & \beta\\
          1 & \beta
        \end{matrix}
        \right).
    \]
    The scalar kernel $\varphi_n^\tau$ is a rescaled version of a base kernel $\varphi_n$:
    \[
        \varphi_n^{\tau}:=a_{\tau}\varphi_n,
    \]
    where the base kernel $(\varphi_n)_{n\ge 1}$ is defined as:
    \[
      \varphi_n:=
      \begin{cases}
        1-\frac1{\Gamma(1-\alpha)}, &n=1,\\
        \frac1{\Gamma(1-\alpha)}\left( \frac1{(n-1)^\alpha}-\frac1{n^\alpha}\right), &n\ge2.
      \end{cases}
    \]
    
    \item \textit{Baseline Intensity:} The time-dependent baseline intensity $\hat\mu_{\tau}(n)$ is defined as:
    \[
      \hat \mu_{\tau}(n):=\mu_{\tau}+\xi_0 \mu_{\tau}\left( \frac1{1-a_{\tau}}\left( 1-\sum_{s=1}^{n-1}\varphi_s^{\tau} \right)-\sum_{s=1}^{n-1}\varphi_s^{\tau} \right).
    \]
\end{itemize}
\end{definition}

The preceding parameterization is admissible.  Indeed,
$0<\Gamma(1-\alpha)^{-1}<1$, so $\varphi_n\ge0$ for every $n$, and the
telescoping definition gives $\sum_{n\ge1}\varphi_n=1$.  To check the
baseline, set $A_{n-1}:=\sum_{s=1}^{n-1}\varphi_s\in[0,1]$.  Then
\[
  \frac{1}{1-a_\tau}\left(1-a_\tau A_{n-1}\right)
  -a_\tau A_{n-1}
  =\frac{1-a_\tau(2-a_\tau)A_{n-1}}{1-a_\tau}
  \ge 1-a_\tau>0,
\]
and hence
$\hat\mu_\tau(n)\ge\mu_\tau\left(1+\xi_0(1-a_\tau)\right)>0$.
Here and below, for a square matrix $A$, we write
\[
  \operatorname{spr}(A):=\max\{\lvert\lambda\rvert:\lambda\text{ is an eigenvalue of }A\}
\]
for its spectral radius.  Finally, $\chi$ is nonnegative with eigenvalues $1$ and $0$, and therefore
\[
  \operatorname{spr}\!\left(\sum\nolimits_{n\ge1}K_n^\tau\right)
  =\operatorname{spr}(a_\tau\chi)=a_\tau<1.
\]

\begin{remark}[Classical case $\alpha=1$]\label{rem:alpha-one-kernel}
Definition~\ref{def:params} applies to the rough regime $\alpha\in(1/2,1)$.
For the classical Heston benchmark in Section~\ref{sec:numerical}, we use the
following endpoint specification.  For every integer
$\tau\ge\lfloor\gamma\rfloor+1$, set
\[
  a_\tau:=1-\gamma\tau^{-1},
  \qquad
  \mu_\tau:=\mu,
  \qquad
  \varphi_1:=1,
  \qquad
  \varphi_n:=0\quad(n\ge2),
\]
and retain $\varphi_n^\tau=a_\tau\varphi_n$, the matrix $\chi$, and the
baseline formula in Definition~\ref{def:params}.  Equivalently,
\[
  \hat\mu_\tau(1)=\mu+\frac{\xi_0\mu}{1-a_\tau},
  \qquad
  \hat\mu_\tau(n)=\mu+\xi_0\mu(1-a_\tau),\quad n\ge2.
\]
The price coefficients $c_\tau$ and $d_\tau$ are still given by
\eqref{eq:ctau-dtau-weak}, with $\alpha=1$.  Thus only the lag-one term remains
in the recursion. This endpoint specification provides the classical-Heston
numerical benchmark in Section~\ref{sec:numerical}. The convergence theorems
are stated for $\alpha\in(1/2,1)$.
\end{remark}

The baseline follows the discrete analogue of
\cite[Definition~2.1]{el2019characteristic}.  The kernel is normalized by
$\sum_{n\ge1}\varphi_n=1$ and satisfies
$n^\alpha\sum_{k>n}\varphi_k\to\Gamma(1-\alpha)^{-1}$.

Because the rows of $\chi$ coincide,
$\lambda_n^{\tau,+}=\lambda_n^{\tau,-}$.  Let
$\mathbf M_n^\tau=\mathbf N_n^\tau-\sum_{s=1}^n\boldsymbol\lambda_s^\tau$.
Resolving the resulting scalar renewal equation with
$\psi^\tau:=\sum_{k\ge1}(\varphi^\tau)^{*k}$ gives
\begin{equation}\label{eq:lambda-decomp}
  \lambda_n^{\tau,+} = \hat{\mu}_\tau(n) + \sum_{s=1}^{n-1} \psi_{n-s}^\tau \hat{\mu}_\tau(s) + \frac{1}{1 + \beta} \sum_{s=1}^{n-1} \psi_{n-s}^\tau \left( M_s^{\tau,+} - M_{s-1}^{\tau,+} + \beta M_s^{\tau,-} - \beta M_{s-1}^{\tau,-} \right).
\end{equation}
This is the vector specialization of the renewal identity in
\cite[Lemma~9]{wang2026scaling}.

\subsection{The rough limits of cumulative INAR\texorpdfstring{$(\infty)$}{(infty)} processes}
The scaling $\mu_\tau=\mu\tau^{\alpha-1}$ and
$1-a_\tau=\gamma\tau^{-\alpha}$ is chosen to produce a nondegenerate
intensity limit.  We use the following vector notation.
For $n\ge 1$, let 
\[
\mathbf N_n^{\tau}:=\begin{pmatrix}N_n^{\tau,+}\\ N_n^{\tau,-}\end{pmatrix}
\quad\text{and}\quad
\boldsymbol\lambda_n^{\tau}:=\mathbb E\left[\left.\mathbf N_n^{\tau}-\mathbf N_{n-1}^{\tau}\right|\mathcal F_{n-1}\right]
=\begin{pmatrix}\lambda_n^{\tau,+}\\ \lambda_n^{\tau,-}\end{pmatrix},
\]
where $(\mathcal F_n)_{n\ge0}$ is the natural filtration, 
$\mathbf N_n^{\tau}$ is the bivariate counting process up to step $n$, and 
$\boldsymbol\lambda_n^{\tau}$ is its (predictable) intensity vector. 
With this notation, fixing $T>0$, for $t\in[0,T]$, we define the rescaled cumulative counts, cumulative intensities, 
and the associated normalized martingale as
\[
    \mathbf Y_t^{\tau}
    :=\frac{1-a_{\tau}}{\tau^\alpha \mu}\mathbf N_{\lfloor t\tau \rfloor}^{\tau},\qquad
    \boldsymbol\Lambda_t^{\tau}
    :=\frac{1-a_{\tau}}{\tau^\alpha \mu}\sum_{s=1}^{\lfloor t\tau \rfloor}\boldsymbol\lambda_s^{\tau},\qquad 
    \mathbf Z_t^{\tau}
    :=\sqrt{\frac{\tau^\alpha\mu}{1-a_{\tau}}}\left(\mathbf Y_t^{\tau}-\boldsymbol\Lambda_t^{\tau}\right).
\]

A key technical fact is that the rescaled renewal kernel associated with $\psi^\tau$ converges weakly to the Mittag--Leffler resolvent density
\[
  f_{\alpha,\gamma}(t)=\gamma t^{\alpha-1}E_{\alpha,\alpha}(-\gamma t^\alpha);
\]
see Lemma~\ref{lem:kernel_convergence}.

\begin{proposition}\label{proposition1}
  Fix $T>0$. Under the parameterization in Definition~\ref{def:params}, the sequence
  $(\boldsymbol\Lambda^{\tau},\mathbf Y^{\tau},\mathbf Z^{\tau})$ is $C$-tight on $[0,T]$.
  In addition,
  \begin{equation}\label{eq:compensator-count-close}
  \sup_{t \in [0,T]}
  \left\|\boldsymbol{\Lambda}_t^\tau-\mathbf Y_t^\tau\right\|_1
  \xrightarrow{\mathbb P}0 .
  \end{equation}
  Moreover, if $(\mathbf Y,\mathbf Z)$ is a limit point of
  $(\mathbf Y^{\tau},\mathbf Z^{\tau})$, then $\mathbf Z$ is a continuous martingale with
  \[
    [\mathbf Z,\mathbf Z]=\mathrm{diag}(\mathbf Y),
  \]
  that is,
$[Z^+,Z^+]_t=Y_t^+$,
$[Z^-,Z^-]_t=Y_t^-$
and $[Z^+,Z^-]_t=0$, for any $t\in[0,T]$.
\end{proposition}

\begin{proof}
See Appendix~\ref{app:proposition1}.
\end{proof}

Since $\Lambda_t^{\tau,+}=\Lambda_t^{\tau,-}$, the approximation
\eqref{eq:compensator-count-close}
also gives
\[
  \sup_{t \in [0,T]} \left|Y_t^{\tau,+} - Y_t^{\tau,-}\right|
  \xrightarrow{\mathbb P}0,
\]
as $\tau \rightarrow \infty$, so both components have the same limit.

\begin{theorem}[Identification and weak convergence to the rough Heston variance]\label{theorem1}
  Fix $T>0$. Under the model parameterization in Definition~\ref{def:params}, as $\tau\rightarrow \infty$, the sequence of processes $(\boldsymbol\Lambda_t^{\tau},\mathbf Y_t^{\tau},\mathbf Z_t^{\tau})_{t\in[0,T]}$ converges in law under the Skorokhod $J_1$ topology on $D([0,T],\mathbb{R}^6)$ to $(\boldsymbol\Lambda,\mathbf Y,\mathbf Z)$, where the limit process is described by a rate process $v = (v_t)_{t\in[0,T]}$ as follows:
  \[
    \boldsymbol\Lambda_t=\mathbf Y_t=\int_0^t v_s ds\left(  \begin{matrix}
     1\\1
    \end{matrix}\right), \quad \text{and} \quad \mathbf Z_t=\int_0^t\sqrt{v_s}\left(\begin{matrix}
    dB_s^1\\dB_s^2
    \end{matrix}\right),
  \]
  and $v$ is a nonnegative weak solution of the stochastic Volterra equation
  \begin{equation}\label{eq:vt}
    v_t=\xi_0+\frac1{\Gamma(\alpha)}\int_0^t(t-s)^{\alpha-1}\gamma(1-v_s)ds+\gamma \sqrt{\frac{1+\beta^2}{\gamma \mu(1+\beta)^2}}\frac1{\Gamma(\alpha)}\int_0^t(t-s)^{\alpha-1}\sqrt{v_s}dW_s^V,
  \end{equation}
  where $B^1$ and $B^2$ are independent standard Brownian motions and
  $W^V:=\frac{B^1+\beta B^2}{\sqrt{1+\beta^2}}$.
  The law of this nonnegative weak solution and the joint law of
  $(\boldsymbol\Lambda,\mathbf Y,\mathbf Z)$ are unique. Furthermore, for
  every $\epsilon\in(0,\alpha-1/2)$, the process $v$ has H\"older regularity
  $\alpha-1/2-\epsilon$.
\end{theorem}

\begin{proof}
See Appendix~\ref{app:proof-theorem1}.
\end{proof}

\subsection{Parameter mapping and correlation for the macroscopic limit}

Throughout the limit, we define the Brownian motions driving price and variance via the following normalized linear combinations
\begin{align}
W_t^X := \frac{1}{\sqrt{2}}\left(B_t^1-B_t^2\right),
\qquad
W_t^V := \frac{B_t^1+\beta B_t^2}{\sqrt{1+\beta^2}},\label{eq:wtxv}
\end{align}
with $(B^1,B^2)$ a two-dimensional standard Brownian motion with independent components. A direct calculation gives the correlation
\begin{equation}\label{eq:rho-def}
  d\left\langle W^X,W^V\right\rangle_t
  =\frac{1-\beta}{\sqrt{2(1+\beta^2)}}dt
  =:\rho dt,
\end{equation}
where $\rho\in\left(-\frac1{\sqrt2},0\right)$ is consistent with the leverage effect induced by the liquidity asymmetry $\beta>1$.

Moreover, moving from the rate process $v$ (for intensities) to the variance process
\[
  V_t := \theta v_t,
  \qquad
  V_0:=\theta\xi_0,
\]
rescales the drift to $\gamma(\theta-V_t)$ and the diffusion coefficient
accordingly; in particular, the initial variance is $V_0=\theta\xi_0$.
Matching the diffusion scale obtained from the INAR martingale CLT yields the identity
\begin{equation}\label{eq:mu-nu-map}
  \nu = \sqrt{\frac{\theta(1+\beta^2)}{\gamma\mu(1+\beta)^2}}
  \qquad\Longleftrightarrow\qquad
  \mu = \frac{\theta(1+\beta^2)}{\gamma\nu^2(1+\beta)^2}.
\end{equation}
We will henceforth \emph{define} $\nu$ via \eqref{eq:mu-nu-map}, so that, with $V=\theta v$ and the correlation $\rho$ in \eqref{eq:rho-def}, the limiting variance dynamics takes the rough Heston form in \eqref{eq:V}.

\noindent\textit{Structural correlation range.}
The present one-factor equal-row specification of $\chi$ yields
$\rho\in(-1/\sqrt2,0)$. A non-equal-row or multifactor influence matrix extends
the attainable correlation range. As $\beta\to\infty$, one has
$\rho\to-1/\sqrt2$, while
\[
  \mu
  =\frac{\theta}{\gamma\nu^2}
    \frac{1+\beta^2}{(1+\beta)^2}
  \longrightarrow \frac{\theta}{\gamma\nu^2}>0,
\]
so the intensity-scale mapping converges to a positive limit.
Moreover, the factor $(1+\beta^2)/(1+\beta)^2$ lies in $(1/2,1)$, and the
classical-Heston Feller inequality associated with the same parameters,
$2\gamma\theta\ge(\gamma\nu)^2$, is equivalent to
$\mu\ge\tfrac12(1+\beta^2)/(1+\beta)^2$.

Then we can also build microscopic processes converging to the log-price. 
For $\theta>0$, set
\begin{equation}\label{eq:ctau-dtau-weak}
  c_\tau:=\sqrt{\frac{\theta(1-a_\tau)}{2\mu\tau^\alpha}},
  \qquad
  d_\tau:=c_\tau-\log\!\left(2-e^{-c_\tau}\right).
\end{equation}
Since $1-a_\tau=\gamma\tau^{-\alpha}$, we have
\[
  c_\tau=\sqrt{\frac{\theta\gamma}{2\mu}}\,\tau^{-\alpha}
  \longrightarrow 0.
\]
The choice of $d_\tau$ is dictated by the exponential martingale
condition.  Because the two conditional Poisson intensities are equal,
we require
\[
  e^{c_\tau-d_\tau}+e^{-c_\tau}=2,
\]
which gives precisely
\[
  d_\tau=c_\tau-\log\!\left(2-e^{-c_\tau}\right).
\]
The compensation is well defined because $c_\tau>0$, so
$2-e^{-c_\tau}>0$.  To see the coefficients in its small-$c_\tau$
expansion, write $q_\tau:=1-e^{-c_\tau}$.  Then
\begin{equation}\label{eq:dtau-expansion}
  \begin{aligned}
    d_\tau
    &=-\log\!\left(1-\left(1-e^{-c_\tau}\right)^2\right)
      =-\log(1-q_\tau^2),
    &\qquad q_\tau=c_\tau-\frac12c_\tau^2+\mathcal O(c_\tau^3).
  \end{aligned}
\end{equation}
and hence
\[
  q_\tau^2=c_\tau^2-c_\tau^3+\mathcal O(c_\tau^4),
  \qquad
  -\log(1-q_\tau^2)=q_\tau^2+\mathcal O(q_\tau^4).
\]
Consequently,
\begin{equation*}
  d_\tau=c_\tau^2-c_\tau^3+\mathcal O(c_\tau^4),
  \qquad
  \frac{d_\tau}{c_\tau^2}\longrightarrow1.
\end{equation*}
Fixing $T>0$, for $t\in[0,T]$ define the microscopic log-price process
$P^\tau=(P_t^\tau)_{t\in[0,T]}$ by
\begin{equation}\label{eq:Ptau-microscopic}
\begin{aligned}
  P_t^{\tau}
  :={c_\tau}
    \left( N_{\lfloor t \tau \rfloor}^{\tau,+}-N_{\lfloor t \tau \rfloor}^{\tau,-} \right)
   -d_\tau N_{\lfloor t \tau \rfloor}^{\tau,+}=\sqrt{\frac{\theta}{2}}\left(Z_{t}^{\tau,+}-Z_{t}^{\tau,-}\right)
   -\frac{\theta}{2}\frac{d_\tau}{c_\tau^2}Y_{t}^{\tau,+}.
\end{aligned}
\end{equation}
This choice gives an exact finite-$\tau$ exponential-martingale
compensation: $S^\tau=S_0e^{P^\tau}$ is a martingale on the grid.
Indeed, for $n\ge1$,
\[
  \Delta P_{n/\tau}^\tau
  =(c_\tau-d_\tau)X_n^{\tau,+}-c_\tau X_n^{\tau,-}.
\]
Conditionally on $\mathcal F_{n-1}$, the two increments are independent
Poisson variables with the same mean $\lambda_n^{\tau,+}$, and therefore
\begin{align*}
  \E\!\left[\left.e^{\Delta P_{n/\tau}^\tau}\right|\mathcal F_{n-1}\right]
  &=\exp\!\left\{\lambda_n^{\tau,+}
      \left(e^{c_\tau-d_\tau}+e^{-c_\tau}-2\right)\right\}
   =1.
\end{align*}
Consequently, for $S_0>0$, $S_t^\tau:=S_0e^{P_t^\tau}$ is a positive
martingale on the discrete grid.  Defining the piecewise-constant
filtration by $\mathcal F_t^\tau:=\mathcal F_{\lfloor\tau t\rfloor}$,
the same process, extended piecewise constantly between grid points, is a
positive martingale on $[0,T]$.

\begin{corollary}[Price-process consequence: rough Heston limit]\label{maintheorem}
  Fix $T>0$. Under the same conditions as Theorem \ref{theorem1}, as $\tau\to\infty$, the rescaled price processes $(P_t^{\tau})_{t\in[0,T]}$ converge in law (Skorokhod $J_1$ topology on $D([0,T])$) to
  \begin{equation}\label{eq:P}
    P_t=\int_0^t\sqrt{V_s}dW_s^X-\frac12\int_0^tV_sds,
  \end{equation}
  where $V_t=\theta v_t$ is a nonnegative weak solution, unique in law, of the
  rough Volterra equation
  \begin{equation}\label{eq:V}
    V_t=V_0+\frac{\gamma}{\Gamma(\alpha)}\int_0^t (t-s)^{\alpha-1}\left(\theta-V_s\right)ds
    +\frac{\gamma\nu}{\Gamma(\alpha)}\int_0^t (t-s)^{\alpha-1}\sqrt{V_s}dW_s^V,
  \end{equation}
  where $\nu$ is defined by \eqref{eq:mu-nu-map} and
  $(W^X,W^V)$ is the Brownian motion from \eqref{eq:wtxv}, with correlation
  $\rho$ given by \eqref{eq:rho-def}.
\end{corollary}

\begin{proof}
See Appendix~\ref{sec:proofoftheorem2}.
\end{proof}


\section{Transform Weak-Error Bounds and Call-Price Convergence}\label{sec:weak-euro}

We first reduce the exact exponential transform recursion to a quadratic discrete Volterra equation and compare its renewal resolvent with the continuous fractional Riccati resolvent. Under the continuous transform-strip and bounded-Riccati assumptions stated below, this gives the quantitative bounded-window transform estimate. Uniform damping-line moments then control the Fourier tails and yield convergence of the full call prices under the same assumptions.

The slower branch $\tau^{-(1-\alpha)}$ comes from endpoint defects in the
weakly singular quadratures for the weighted baseline and the cumulative
resolvent. We first record the baseline estimate.

\begin{proposition}[Weighted baseline-sum approximation]\label{prop:baseline-riemann}
Fix $T>0$ and let $m_\tau:=\lfloor\tau T\rfloor$. With
\[
  \varepsilon_\tau:=1-a_\tau=\gamma\tau^{-\alpha},
\]
for every $f\in C^\alpha([0,T])$, there exists a constant $C_T>0$, depending
only on the fixed model parameters and on $T$, but independent of $\tau$ and
$f$, such that, for all sufficiently large $\tau$,
\begin{equation}\label{eq:baseline-riemann}
\begin{aligned}
  &\left|
    \frac{\varepsilon_\tau}{\mu}
    \sum_{m=1}^{m_\tau}\hat\mu_\tau(m)f\!\left(T-\frac{m}{\tau}\right)
    -
    \gamma\int_0^T f(s)\,ds
    -
    \xi_0 I^{1-\alpha}f(T)
  \right|
  \\
  &\le
  C_T\|f\|_{C^\alpha([0,T])}\tau^{-\alpha}
  +
  C_{\mathrm{base}}(\alpha)\|f\|_{C^\alpha([0,T])}\tau^{-(1-\alpha)},
\end{aligned}
\end{equation}
where
\[
  I^{1-\alpha}f(T)
  :=
  \frac{1}{\Gamma(1-\alpha)}
  \int_0^T(T-s)^{-\alpha}f(s)\,ds,
\]
and one may take the explicit coefficient
\[
  C_{\mathrm{base}}(\alpha)
  :=|\xi_0|
  \left|1+\frac{\zeta(\alpha)}{\Gamma(1-\alpha)}\right|.
\]
\end{proposition}

\begin{proof}
See Appendix~\ref{app:baseline-riemann}.
\end{proof}

The corresponding cumulative-resolvent defect is analyzed directly in the
time domain in Appendix~\ref{app:proof-cum-resolvent-uniform-comparison}.

Throughout this section, the maturity, frequency window, and pricing parameters
are fixed, with $S_0>0$ and $K>0$. Constants are uniform in $\tau$; dependence on these fixed
quantities is suppressed, while selected dependence on $\alpha$ is made explicit,
as in $C_{\mathrm{base}}(\alpha)$. Set
$m_\tau:=\lfloor \tau T\rfloor$, $t_n:=n/\tau$, and
$k:=\log(K/S_0)$. Recalling $c_\tau$ and $d_\tau$ from
\eqref{eq:ctau-dtau-weak}, the microscopic log-price at maturity is
\begin{equation}\label{eq:Ptau-weak}
  P_T^\tau
  =c_\tau\left(N_{m_\tau}^{\tau,+}-N_{m_\tau}^{\tau,-}\right)-d_\tau N_{m_\tau}^{\tau,+}.
\end{equation}


\subsection{Continuous fractional Riccati equation}

Let $P_T=\log(S_T/S_0)$ denote the rough Heston log-price at maturity $T$.
By \cite{el2019characteristic} and the complex-parameter extension in
\cite[Theorem~7]{gerhold2019moment}, its log-price transform admits the representation
\begin{equation}\label{eq:phi-cont-weak}
\phi(z,T):=\E\!\left[e^{zP_T}\right]
  =\exp\!\left(\gamma\theta\, I^1 h(\cdot,z)(T)+V_0\, I^{1-\alpha}h(\cdot,z)(T)\right),
\end{equation}
for complex $z$ in the continuous-model analyticity strip introduced below, where, for $r\in(0,1]$,
\[
  I^r f(t):=\frac{1}{\Gamma(r)}\int_0^t (t-s)^{r-1}f(s)\,ds,
\]
and $h(\cdot,z)$ solves the fractional Riccati equation
\begin{equation}\label{eq:riccati-cont-weak}
  D^\alpha h(t,z)
  =\frac12(z^2-z)+\gamma(\rho\nu z-1)h(t,z)
   +\frac{(\gamma\nu)^2}{2}\left(h(t,z)\right)^2,
  \qquad \left(I^{1-\alpha}h(\cdot,z)\right)(0+)=0.
\end{equation}
For real $u$, the characteristic function is recovered as $\phi(\ii u,T)$.
Here, the fractional derivative $D^\alpha$ is defined by
\[
  D^\alpha f(t)
  :=\frac{1}{\Gamma(1-\alpha)}
    \frac{d}{dt}\int_0^t (t-s)^{-\alpha}f(s)\,ds.
\]
Equivalently, applying $I^\alpha$ to \eqref{eq:riccati-cont-weak}, one obtains the Volterra form
\begin{equation}\label{eq:riccati-volterra-weak}
  h(t,z)
  =\frac{1}{\Gamma(\alpha)}\int_0^t
    (t-s)^{\alpha-1}
    \left[
      \frac12(z^2-z)
      +\gamma(\rho\nu z-1)h(s,z)
      +\frac{(\gamma\nu)^2}{2}\left(h(s,z)\right)^2
    \right]ds.
\end{equation}


\subsection{Exact discrete transform recursion}\label{subsec:exact-discrete-transform}

Poisson branching--immigration yields an \emph{exact} discrete Volterra
recursion for the finite-$\tau$ log-price transform.

Throughout this subsection, we use the following convention for the bivariate influence matrix: the \emph{rows} correspond to the child type (future arrival type), while the \emph{columns} correspond to the parent type (source event type). Under the parameterization of Definition~\ref{def:params},
\[
  \chi=\frac1{1+\beta}
  \begin{pmatrix}
    1 & \beta\\
    1 & \beta
  \end{pmatrix},
\]
so that a ``$+$''-parent produces ``$+$''- and ``$-$''-children with the same lag-$k$ mean $\varphi_k^\tau/(1+\beta)$, whereas a ``$-$''-parent produces ``$+$''- and ``$-$''-children with the same lag-$k$ mean $\beta\varphi_k^\tau/(1+\beta)$.

For $z\in\C$, define the one-step exponential weights
\begin{equation}\label{eq:Theta-pm-exact}
  \Theta_+^\tau(z):=z(c_\tau-d_\tau),
  \qquad
  \Theta_-^\tau(z):=-zc_\tau,
\end{equation}
where $c_\tau$ and $d_\tau$ are defined in \eqref{eq:ctau-dtau-weak}.
These are the direct contributions of one ``$+$''-arrival and one ``$-$''-arrival, respectively, to the exponent $zP_T^\tau$.

Next, introduce the \emph{parent-type aggregated lag coefficients}
\begin{equation}\label{eq:q-pm-exact}
  q_k^{\tau,+}:=\frac{\varphi_k^\tau}{1+\beta},
  \qquad
  q_k^{\tau,-}:=\frac{\beta\varphi_k^\tau}{1+\beta},
  \qquad k\ge 1.
\end{equation}

We now introduce the cluster transforms. Consider a single immigrant born at the current time and let its descendants evolve over the next $n$ time steps according to the bivariate Poisson branching mechanism of the INAR($\infty$) model. Let
\[
  \mathcal M_n^{\tau,+}(z)
  \quad\text{and}\quad
  \mathcal M_n^{\tau,-}(z)
\]
denote the transforms of the total contribution of such a cluster to the exponent $zP_T^\tau$, when the initial immigrant is of type $+$ and $-$, respectively. Since the ancestor itself contributes exactly one ``$+$''-jump or one ``$-$''-jump at its birth time, we have the initial conditions
\begin{equation}\label{eq:M0-exact}
  \mathcal M_0^{\tau,+}(z)=e^{\Theta_+^\tau(z)},
  \qquad
  \mathcal M_0^{\tau,-}(z)=e^{\Theta_-^\tau(z)}.
\end{equation}

Independence of the two Poisson child counts gives the aggregate continuation
factor
$\exp\{q_k^{\tau,\pm}(\mathcal M_{n-k}^{\tau,+}
+\mathcal M_{n-k}^{\tau,-}-2)\}$ used below.

The next proposition gives the exact cluster recursion in this aggregated representation and the exact transform of the finite-$\tau$ INAR approximation.

\begin{proposition}[Exact exponential discrete Volterra recursion]\label{prop:exact-transform-weak}
For each $\tau\ge\tau_{\min}$ and each $z\in\C$, let $q_k^{\tau,+}$ and $q_k^{\tau,-}$ be the parent-type coefficients from \eqref{eq:q-pm-exact}, corresponding respectively to a ``$+$''-parent and a ``$-$''-parent. Then the cluster transforms $\mathcal M_n^{\tau,+}(z)$ and $\mathcal M_n^{\tau,-}(z)$ satisfy, for every $n\ge1$,
\begin{align}
  \mathcal M_n^{\tau,+}(z)
  &=\exp\!\left(
      \Theta_+^\tau(z)
      +\sum_{k=1}^n q_k^{\tau,+}
       \left(\mathcal M_{n-k}^{\tau,+}(z)+\mathcal M_{n-k}^{\tau,-}(z)-2\right)
    \right), \label{eq:Mplus-exact}
  \\
  \mathcal M_n^{\tau,-}(z)
  &=\exp\!\left(
      \Theta_-^\tau(z)
      +\sum_{k=1}^n q_k^{\tau,-}
       \left(\mathcal M_{n-k}^{\tau,+}(z)+\mathcal M_{n-k}^{\tau,-}(z)-2\right)
    \right). \label{eq:Mminus-exact}
\end{align}
If we define
\begin{equation}\label{eq:Gn-exact}
  G_n^\tau(z):=\mathcal M_n^{\tau,+}(z)+\mathcal M_n^{\tau,-}(z)-2,
\end{equation}
then $G_n^\tau(z)$ satisfies the closed recursion
\begin{equation}\label{eq:G-rec-exact}
  G_n^\tau(z)
  =\exp\!\left(
      \Theta_+^\tau(z)+\sum_{k=1}^n q_k^{\tau,+}G_{n-k}^\tau(z)
    \right)
   +\exp\!\left(
      \Theta_-^\tau(z)+\sum_{k=1}^n q_k^{\tau,-}G_{n-k}^\tau(z)
    \right)-2.
\end{equation}
Moreover, the exact transform of the INAR approximation is
\begin{equation}\label{eq:Phi-exact}
  \Phi^\tau(z,T):=\E\!\left[e^{zP_T^\tau}\right]
  =\exp\!\left(\sum_{m=1}^{m_\tau}\hat\mu_\tau(m)\,G_{m_\tau-m}^\tau(z)\right).
\end{equation}
In particular, the martingale identity
$e^{c_\tau-d_\tau}+e^{-c_\tau}=2$ gives
$G_n^\tau(1)=0$ for every $n$ by induction, and hence
$\Phi^\tau(1,T)=1$ exactly. Moreover, $\Phi^\tau(\cdot,T)$ is entire.
\end{proposition}

\begin{proof}
See Appendix~\ref{app:weak-euro-exact}.
\end{proof}

The continuous transform $\phi$, the discrete transform $\Phi^\tau$, and the
fractional Riccati solution $h$ have now all been defined. We impose the
continuous-model analytic conditions used in the comparison. For $\eta\in\R$
and $R>0$,
write
\[
  \Gamma_{\eta,R}:=\{\eta+\ii\xi:\ |\xi|\le R\}.
\]

\begin{assumption}[Admissible continuous transform strip]\label{ass:strip}
Fix the maturity $T>0$ considered in this section. We assume that there exists
a nonempty open vertical strip
\[
  \mathcal S_T
  :=\{z\in\C:\underline\eta_T<\operatorname{Re}z<\overline\eta_T\}
\]
such that $\phi(z,T)$ is finite and analytic on $\mathcal S_T$, and the affine
representation \eqref{eq:phi-cont-weak} holds there with the nonexploding
solution $h(\cdot,z)$ of \eqref{eq:riccati-cont-weak}. All weak-error statements
below are formulated for vertical lines contained in this strip.
The finite-$\tau$ transform $\Phi^\tau(\cdot,T)$ is entire by
Proposition~\ref{prop:exact-transform-weak}; hence the strip condition concerns
only the continuous transform.
\end{assumption}

\begin{assumption}[Bounded-strip Riccati regularity]\label{ass:riccati-bound}
Fix $\eta\in\R$ and $R>0$ with
$\Gamma_{\eta,R}\subset\mathcal S_T$. We assume
\[
  M_{\eta,R,T}
  :=
  \sup_{z\in\Gamma_{\eta,R}}\sup_{0\le t\le T}|h(t,z)|
  <\infty.
\]
\end{assumption}

\Needspace{8\baselineskip}
\begin{remark}[Numerical verification]\label{rem:numerical-strip-check}
For real $\eta>1$ with $\phi(\eta,T)<\infty$, convexity and the martingale
property of $e^{P_t}$ \cite[Lemma~7.3 and Example~7.2]{abi2019affine} give
$\phi(\eta,t)\le\phi(\eta,T)$ for $t\le T$, so the terminal real moment controls
all earlier times. For every rough-Heston parameter set used in
Section~\ref{sec:numerical} at $T=1$, we numerically checked the conditions in
Assumptions~\ref{ass:strip} and~\ref{ass:riccati-bound} on the grid
$\operatorname{Re}z\in\{1.5,2,2.5\}$ and
$\operatorname{Im}z\in\{-20,-19,\ldots,20\}$.  The continuous fractional
Riccati solutions remained finite and stable under refinement from $1000$ to
$2000$ time steps, and the exact finite-$\tau$ transform recursion remained
finite for every discretization level used in the experiments. These results
support the assumptions over the numerical domain used in
Section~\ref{sec:numerical} and are reproducible with the accompanying
open-source code.
\end{remark}


\subsection{Second-order reduction on bounded frequency windows}
\label{subsec:second-order-reduction}

We now reduce the exact exponential recursion
\eqref{eq:G-rec-exact} to a second-order discrete recursion on each bounded vertical strip. Fix $\eta\in\mathbb R$ and $R>0$.
For the remainder of this subsection, all constants may depend on $(\eta,R,T)$, but not on $(n,\tau,z)$.
All moment-generating-function calculations are restricted to
$\Gamma_{\eta,R}\subset\mathcal S_T$.

Recalling the exponents appearing in \eqref{eq:G-rec-exact}, define
\begin{equation}\label{eq:AB-def-new33}
  A_n^\tau(z)
  :=\Theta_+^\tau(z)+\sum_{k=1}^n q_k^{\tau,+}G_{n-k}^\tau(z),
  \qquad
  B_n^\tau(z)
  :=\Theta_-^\tau(z)+\sum_{k=1}^n q_k^{\tau,-}G_{n-k}^\tau(z),
\end{equation}
where the coefficients are defined in \eqref{eq:Theta-pm-exact} and
\eqref{eq:q-pm-exact}. By definition,
\begin{equation}\label{eq:G-from-AB-new33}
  G_n^\tau(z)=e^{A_n^\tau(z)}+e^{B_n^\tau(z)}-2.
\end{equation}

With $\varepsilon_\tau:=1-a_\tau=\gamma\tau^{-\alpha}$, set
$L:=M_{\eta,R,T}+2$ and introduce the first exit index
\[
  \sigma_\tau
  :=\inf\left\{0\le n\le m_\tau:
       \frac{\mu}{\theta\varepsilon_\tau}
       \sup_{z\in\Gamma_{\eta,R}}|G_n^\tau(z)|>L\right\},
\]
with the convention $\inf\varnothing=m_\tau+1$.
Its dependence on $(\eta,R)$ is suppressed in the notation.
Appendix~\ref{app:weak-euro-second-order} establishes a uniform third-order
Taylor expansion of \eqref{eq:G-from-AB-new33} through a stopped bounded-strip
continuation argument, yielding the following recursion.

\begin{proposition}[Second-order reduced discrete recursion]
\label{prop:reduced-recursion-new33}
Under Assumptions~\ref{ass:strip} and~\ref{ass:riccati-bound}, for every fixed
$\eta\in\mathbb R$ and $R>0$ with
$\Gamma_{\eta,R}\subset\mathcal S_T$, uniformly for
$0\le n\le m_\tau\wedge\sigma_\tau$ and
$z\in\Gamma_{\eta,R}$,
\begin{equation}\label{eq:G-reduced-new33}
\begin{aligned}
  G_n^\tau(z)
  &=
    \sum_{k=1}^n \varphi_k^\tau\,G_{n-k}^\tau(z)
    +c_\tau^2(z^2-z) \\
  &\quad
    + \frac{1-\beta}{1+\beta}\,c_\tau z
      \sum_{k=1}^n \varphi_k^\tau\,G_{n-k}^\tau(z)
    + \frac{1+\beta^2}{2(1+\beta)^2}
      \left(\sum_{k=1}^n \varphi_k^\tau\,G_{n-k}^\tau(z)\right)^2
    + \mathcal R_n^\tau(z),
\end{aligned}
\end{equation}
where
\begin{equation}\label{eq:R-reduced-new33}
  \sup_{0\le n\le m_\tau\wedge\sigma_\tau}\sup_{z\in\Gamma_{\eta,R}}
  |\mathcal R_n^\tau(z)|
  \le C_{\eta,R,T}\,\varepsilon_\tau^3.
\end{equation}
\end{proposition}

\begin{proof}
See Appendix~\ref{app:proof-reduced-recursion-new33}.
\end{proof}


\subsection{Rescaling and match with the continuous fractional Riccati equation}
\label{subsec:rescaling-match}

After rescaling, Proposition~\ref{prop:reduced-recursion-new33} has the same
source, linear, and quadratic coefficients as the fractional Riccati equation,
up to a local truncation error.

Recall that
\[
  \varepsilon_\tau:=1-a_\tau=\gamma\tau^{-\alpha},
\]
and define the rescaled discrete variable
\begin{equation}\label{eq:H-rescaled-def}
  H_n^\tau(z):=\frac{\mu}{\theta\,\varepsilon_\tau}\,G_n^\tau(z),
  \qquad 0\le n\le m_\tau,\quad z\in\Gamma_{\eta,R}.
\end{equation}
Equivalently,
\[
  G_n^\tau(z)=\frac{\theta}{\mu}\,\varepsilon_\tau H_n^\tau(z).
\]

Substitution in \eqref{eq:G-reduced-new33}, followed by
\eqref{eq:rho-def} and \eqref{eq:mu-nu-map}, matches the source coefficient
exactly.  The exact martingale compensation contributes
$z(c_\tau^2-d_\tau)=\mathcal O(\varepsilon_\tau^3)$, which is included in
$\mathcal R_n^\tau$ in Proposition~\ref{prop:reduced-recursion-new33}.
Before collecting the remaining lower-order terms, the cross and quadratic
coefficients are, respectively,
$a_\tau\gamma\rho\nu\,\tau^{-\alpha}$
and
$a_\tau^2\frac{(\gamma\nu)^2}{2}\tau^{-\alpha}$.
Since $1-a_\tau=\gamma\tau^{-\alpha}$, replacing $a_\tau$ and $a_\tau^2$
by one contributes $\mathcal O(\tau^{-2\alpha})$ on the stopped bounded
strip. Together with the rescaled Taylor remainder, these terms define
$E_n^\tau$. On the stopped grid, the resulting recursion is
\begin{equation}\label{eq:Htau-rec-main}
\begin{aligned}
  H_n^\tau(z)
  &= a_\tau\sum_{k=1}^n \varphi_k H_{n-k}^\tau(z)
     +\frac{1}{2\tau^\alpha}(z^2-z)
\\
  &\quad
     +\gamma\rho\nu\,z\,\tau^{-\alpha}
      \sum_{k=1}^n \varphi_k H_{n-k}^\tau(z)
     +\frac{(\gamma\nu)^2}{2}\tau^{-\alpha}
      \left(\sum_{k=1}^n \varphi_k H_{n-k}^\tau(z)\right)^2
     +E_n^\tau(z),
\end{aligned}
\end{equation}
where Proposition~\ref{prop:reduced-recursion-new33} and
$\varepsilon_\tau=\gamma\tau^{-\alpha}$ imply
\begin{equation}\label{eq:En-bound}
  \sup_{0\le n\le m_\tau\wedge\sigma_\tau}\sup_{z\in\Gamma_{\eta,R}}
  |E_n^\tau(z)|
  \le C_{\eta,R,T}\tau^{-2\alpha}.
\end{equation}

To identify the limiting Volterra operator, let
$\psi^\tau:=\sum_{m\ge1}(\varphi^\tau)^{*m}$
be the discrete renewal kernel associated with $\varphi^\tau=a_\tau\varphi$, and define
\[
  \mathsf{S}_k^\tau:=\frac{\tau(1-a_\tau)}{a_\tau}\psi_k^\tau,
  \qquad k\ge1.
\]
A key technical fact is that the measures associated with $\mathsf{S}^\tau$ converge weakly to the Mittag--Leffler resolvent density
\[
  f_{\alpha,\gamma}(t)=\gamma t^{\alpha-1}E_{\alpha,\alpha}(-\gamma t^\alpha);
\]
see Lemma~\ref{lem:kernel_convergence}. This is the resolvent kernel in the
continuous rough Heston Riccati equation.


\subsection{Renewal inversion and cumulative-resolvent comparison}
\label{sec:renewal-comparison}

We rewrite the reduced recursion in resolvent form and compare the discrete
and continuous cumulative resolvents.

Introduce the quadratic polynomial
\begin{equation}\label{eq:Fz-def}
  F_z(x):=\frac12(z^2-z)+\gamma\rho\nu z\,x+\frac{(\gamma\nu)^2}{2}x^2,
  \qquad x\in\C.
\end{equation}
Then \eqref{eq:Htau-rec-main} can be rewritten, for
$0\le n\le m_\tau\wedge\sigma_\tau$, as
\begin{equation}\label{eq:Htau-rec-main-again}
  H_n^\tau(z)
  =
  a_\tau\sum_{k=1}^n\varphi_k H_{n-k}^\tau(z)
  +\tau^{-\alpha}F_z\!\left(\sum_{k=1}^n\varphi_k H_{n-k}^\tau(z)\right)
  +E_n^\tau(z),
\end{equation}
where $E_n^\tau$ satisfies \eqref{eq:En-bound} on the same stopped grid.

Appendix~\ref{app:proof-barH} gives the algebraic reformulation that replaces
the averaged state by $H_n^\tau$ and bounds the associated lower-order error.

\begin{definition}[Discrete and continuous cumulative resolvents]
\label{def:cum-resolvents}
The discrete cumulative resolvent is defined by
\[
  \mathcal G_n^\tau
  :=
  \frac{1-a_\tau}{\gamma a_\tau}
  \sum_{k=1}^n\psi_k^\tau,
  \qquad n\ge0,
\]
with $\mathcal G_0^\tau=0$.
The continuous cumulative resolvent associated with $f_{\alpha,\gamma}$ is
defined by
\[
  \mathcal G(t)
  :=
  \frac1\gamma\left(1-E_\alpha(-\gamma t^\alpha)\right),
  \qquad t\ge0.
\]
Equivalently,
$\mathcal G(t)=\frac1\gamma\int_0^t f_{\alpha,\gamma}(s)\,ds$.
\end{definition}

Appendix~\ref{app:weak-error-proofs} proves the discrete renewal equation and
compares it with the continuous cumulative resolvent through a time-domain
quadrature-defect argument.


\subsection{Riccati comparison on bounded strips}
\label{subsec:transform-bound}

For fixed $(\eta,R)$, constants may depend on $(\eta,R,T)$ but not on
$(n,\tau,z)$. Appendix~\ref{app:weak-error-proofs} establishes the Riccati
time regularity, weakly singular quadrature decomposition, and renewal bounds
used in the following uniform cumulative-resolvent comparison.

\begin{proposition}[Uniform cumulative-resolvent comparison]
\label{prop:cum-resolvent-uniform-comparison}
Assume $\alpha\in(1/2,1)$, and let $\mathcal G_n^\tau$ and $\mathcal G(t)$ be
the cumulative resolvents from Definition~\ref{def:cum-resolvents}.
Then there exists a coefficient $C_{\mathrm{res}}(\alpha)\ge0$ and a constant $C_T>0$, independent of $\tau$, such that
\begin{equation}\label{eq:cum-resolvent-uniform-comparison}
  \sup_{0\le n\le m_\tau}
  |\mathcal G_n^\tau-\mathcal G(t_n)|
  \le
  C_T\tau^{-\alpha}+C_{\mathrm{res}}(\alpha)\,\tau^{-(1-\alpha)},
\end{equation}
where $C_{\mathrm{res}}(\alpha)<\infty$ for every fixed
$\alpha\in(1/2,1)$. 
\end{proposition}

\begin{proof}
See Appendix~\ref{app:proof-cum-resolvent-uniform-comparison}.
\end{proof}

We now introduce the discrete and continuous resolvent operators acting on a test
function $g$:
\[
  (\mathcal Q_\tau g)_n
  :=
  \frac{1}{\gamma\tau}\sum_{j=1}^{n-1}\mathsf{S}_{n-j}^\tau g(t_j),
  \qquad 1\le n\le m_\tau,
\]
and
\[
  (\mathcal Q g)(t_n)
  :=
  \frac1\gamma\int_0^{t_n}f_{\alpha,\gamma}(t_n-s)g(s)\,ds.
\]

\begin{proposition}[Transfer from cumulative-resolvent comparison to operator comparison]
\label{prop:resolvent-transfer-from-cumulative}
Let $g:[0,T]\to\C$ be absolutely continuous and suppose that
\begin{equation}\label{eq:g-assump-2}
  |\partial_t g(t)| \le C_g t^{\alpha-1},
  \qquad 0<t\le T.
\end{equation}
Then there exists a coefficient $C_{\mathrm{op}}(\alpha)\ge0$ and a constant
$C_T'>0$, depending only on $(\alpha,\gamma,T)$, such that
\begin{equation}\label{eq:resolvent-operator-compare-main}
  \sup_{1\le n\le m_\tau}
  |(\mathcal Q_\tau g)_n-(\mathcal Q g)(t_n)|
  \le
  C_T'(|g(0)|+C_g)\tau^{-\alpha}
  +
  C_{\mathrm{op}}(\alpha)(|g(0)|+C_g)\tau^{-(1-\alpha)},
\end{equation}
where $C_{\mathrm{op}}(\alpha)<\infty$ for every fixed
$\alpha\in(1/2,1)$. 
\end{proposition}

\begin{proof}
See Appendix~\ref{app:proof-resolvent-transfer-from-cumulative}.
\end{proof}

Applying Proposition~\ref{prop:resolvent-transfer-from-cumulative} to the
continuous Riccati nonlinearity yields the nonlinear resolvent estimate needed
below. The final comparison between $H^\tau$ and $h$ is then closed with a
discrete weakly singular Gr\"onwall estimate; see Appendix~\ref{app:weak-error-proofs}.

\begin{proposition}[Bounded-strip Riccati comparison]
\label{prop:Volterra-rate-fixed}
Fix $\eta\in\R$ and $R>0$. Under Assumptions~\ref{ass:strip}
and~\ref{ass:riccati-bound}, there exist a coefficient
$C_{\mathrm{Ric}}(\alpha)\ge0$ and a constant $C_{\eta,R,T}>0$ such that,
for all sufficiently large $\tau$,
\begin{equation}\label{eq:Volterra-rate-fixed}
  \sup_{0\le n\le m_\tau}\sup_{z\in\Gamma_{\eta,R}}
  \left|H_n^\tau(z)-h(t_n,z)\right|
  \le
  C_{\eta,R,T}\tau^{-\alpha}+C_{\mathrm{Ric}}(\alpha)\tau^{-(1-\alpha)},
\end{equation}
where $C_{\mathrm{Ric}}(\alpha)<\infty$ for every fixed
$\alpha\in(1/2,1)$. 
\end{proposition}

\begin{proof}
See Appendix~\ref{app:proof-Volterra-rate-fixed}.
\end{proof}

The continuation in the proof yields $\sigma_\tau=m_\tau+1$ for all
sufficiently large $\tau$ and extends the stopped expansions to the full grid.

We now transfer the bounded-strip Riccati comparison from
Proposition~\ref{prop:Volterra-rate-fixed} to the exponent and transform levels.
Define the canonical cumulant exponents in the two affine representations
directly by
\[
 \Psi^\tau(z,T):=\sum_{m=1}^{m_\tau}\hat\mu_\tau(m)
 G_{m_\tau-m}^\tau(z),
 \qquad
 \Psi(z,T):=\gamma\theta I^1h(\cdot,z)(T)
              +V_0 I^{1-\alpha}h(\cdot,z)(T).
\]
Thus \eqref{eq:Phi-exact} and \eqref{eq:phi-cont-weak} state directly that
$\Phi^\tau=e^{\Psi^\tau}$ and $\phi=e^\Psi$. The next proposition combines
the cumulant-exponent and transform comparisons.

\begin{proposition}[Bounded-strip cumulant and transform comparison]
\label{prop:transform-bound-fixed}
For every fixed $\eta\in\R$ and $R>0$ with
$\Gamma_{\eta,R}\subset\mathcal S_T$, under
Assumptions~\ref{ass:strip} and~\ref{ass:riccati-bound}, there are finite
coefficients $C_{\mathrm{exp}}(\alpha),C_{\Phi}(\alpha)\ge0$ and a constant
$C_{\eta,R,T}>0$ such that, for all sufficiently large $\tau$,
\begin{equation}\label{eq:log-transform-bound}
  \sup_{z\in\Gamma_{\eta,R}}
  |\Psi^\tau(z,T)-\Psi(z,T)|
  \le
  C_{\eta,R,T}\tau^{-\alpha}
  +
  C_{\mathrm{exp}}(\alpha)\tau^{-(1-\alpha)},
\end{equation}
and
\begin{equation}\label{eq:transform-bound-fixed}
  \sup_{z\in\Gamma_{\eta,R}}
  \left|
    \Phi^\tau(z,T)-\phi(z,T)
  \right|
  \le
  C_{\eta,R,T}\tau^{-\alpha}+C_{\Phi}(\alpha)\tau^{-(1-\alpha)}.
\end{equation}
Equivalently, \eqref{eq:transform-bound-fixed} is uniform for $|\xi|\le R$
under the parametrization $z=\eta+\ii\xi$.
\end{proposition}

\begin{proof}
See Appendix~\ref{app:proof-transform-bound-fixed}.
\end{proof}

\noindent\textit{Explicit slow-rate factor.}
Set $c_1(\alpha):=1+\zeta(\alpha)/\Gamma(1-\alpha)$. The explicit choices in
Appendix~\ref{app:weak-error-proofs} give
$C_\bullet=|c_1|\Lambda_\bullet$ for
$\bullet\in\{\mathrm{base},\mathrm{res},\mathrm{op},\mathrm{Ric},
\mathrm{exp},\Phi\}$, where $\Lambda_\bullet(\alpha)<\infty$ for each fixed
$\alpha$. Thus, for each fixed $\alpha$, these constants are finite and
independent of $\tau$, as required for the rates above. The explicit factor
$c_1(\alpha)$ satisfies the following endpoint asymptotic: since
$\zeta(\alpha)+\Gamma(1-\alpha)=\mathcal O(1-\alpha)$ and
$\Gamma(1-\alpha)\sim(1-\alpha)^{-1}$ as $\alpha\uparrow1$, one has
$c_1(\alpha)=\mathcal O((1-\alpha)^2)$. At
$\alpha=0.55,0.62,0.80,0.95$, its values are, respectively,
$0.1471,0.1093,0.0334,0.00224$.

\subsection{Full call-price convergence}
\label{subsec:full-call-price}

For $S_0>0$ and $K>0$, define the full discounted call prices
\[
  C^\tau(T,K):=\E\!\left[\left(S_0e^{P_T^\tau}-K\right)^+\right],
  \qquad
  C(T,K):=\E\!\left[\left(S_0e^{P_T}-K\right)^+\right].
\]
\begin{theorem}[Full call-price convergence]
\label{thm:full-call-convergence}
Under Assumption~\ref{ass:strip}, fix $\eta>1$ with
$\underline\eta_T<\eta<\overline\eta_T$, and suppose
Assumption~\ref{ass:riccati-bound} holds on $\Gamma_{\eta,R}$ for every fixed
$R>0$.  Then
\begin{equation}\label{eq:full-call-convergence}
  C^\tau(T,K)\longrightarrow C(T,K)
  \qquad\text{as }\tau\to\infty.
\end{equation}
Moreover, the convergence is uniform when $k=\log(K/S_0)$ ranges over any
fixed compact subset of $\R$.
\end{theorem}

The proof combines Proposition~\ref{prop:transform-bound-fixed} on each fixed
Fourier window with a uniform exponential-moment bound that removes the
window.

\begin{proof}
See Appendix~\ref{app:proof-full-call-price}.
\end{proof}


\section{Numerical Results}\label{sec:numerical}

The European experiments compare the microstructural approximation with
available Fourier benchmarks; the Asian, lookback, and barrier experiments
apply the same simulator to path-dependent payoffs.

\subsection{Payoff conventions and implementation workflow}
\label{subsec:numerical-workflow}

\paragraph{Pricing and payoff conventions}
Under the risk-neutral measure, we write $S$ for the discounted price, so
\[
dS_t = S_t \sqrt{V_t} dW_t^X,
\]
with the rough Heston variance and correlation specified above.  We set
$r=0$ in all reported prices.  Besides the European payoffs
$(S_T-K)^+$ and $(K-S_T)^+$, we use the grid-monitored quantities
\begin{equation}\label{Asian:defn}
 \bar S_T=\frac1{M+1}\sum_{i=0}^M S_{iT/M},\qquad
 S_T^{\max}=\max_{0\le i\le M}S_{iT/M},\qquad
 S_T^{\min}=\min_{0\le i\le M}S_{iT/M}.
\end{equation}
Arithmetic Asian calls and puts replace $S_T$ by $\bar S_T$; lookback calls
and puts have payoffs $(S_T^{\max}-K)^+$ and $(K-S_T^{\min})^+$.  The
up-and-in call is $(S_T-K)^+\mathbf1_{\{S_T^{\max}\ge B\}}$, and the
down-and-out put is $(K-S_T)^+\mathbf1_{\{S_T^{\min}>B\}}$.  Thus averages,
extrema, and barrier hits
are all monitored on the simulation grid, including the initial point.

We simulate the \emph{INAR($\infty$)} recursion on $[0,T]$ with zero
pre-history ($X_n\equiv0$ for $n\le0$) via a non-circular linear convolution
(CDQ--FFT). At time $n\le\lfloor\tau T\rfloor$, the update uses the lags
$1,\dots,n-1$. Hence, for any $p\ge\lfloor\tau T\rfloor$, the INAR($p$)
recursion \cite{alzaid1990integer,jin1991integer} coincides \emph{pathwise}
with our INAR($\infty$) update on $[0,T]$. Thus the finite-horizon
implementation uses every available lag. The naive per-path cost is $\mathcal{O}(\tau^2)$.  CDQ--FFT reduces it to
$\mathcal{O}(\tau\log^2\tau)$; the effect of finite $\tau$ is examined below.
The implementation uses the base kernel $\varphi$ from
Definition~\ref{def:params} and records all terminal and path-dependent
summaries on the fly.

\subsubsection{Parallel Monte Carlo implementation}

At each integer time $n$, the two counts are sampled independently from a
Poisson distribution with common intensity
\begin{equation}\label{eq:lambda-update}
  \lambda_n = \hat{\mu}_\tau(n) + \sum_{k=1}^{n-1} w_k Q_{n-k},
  \qquad
  Q_n := X_n^+ + \beta X_n^-.
\end{equation}
Collapsing the bivariate system gives
\[
  w_k=\frac{a_\tau}{1+\beta}\varphi_k,
  \qquad
  \E[\left.Q_n\right|\mathcal F_{n-1}]=(1+\beta)\lambda_n.
\]
Thus the factor $1/(1+\beta)$ in $w_k$ cancels the conditional-mean factor
$1+\beta$, recovering the intended scalar excitation coefficient
$a_\tau\varphi_k$. All path
functionals are computed on the simulated grid over $[0,T]$.

\begin{algorithm}[ht]
\caption{Parallel INAR($\infty$) simulation with CDQ--FFT convolution}\label{algo}
\SetAlgoLined
\small
\KwIn{$(\alpha,\gamma,\rho,\nu,\theta,V_0,S_0)$, $\tau$, and number of paths}
Derive the INAR parameters and precompute the kernel, baseline, and CDQ--FFT blocks\;
\For{each thread in parallel}{
    \For{each simulation assigned to the thread}{
        \For{each pre-computed CDQ operation in execution order}{
            \eIf{the operation is a leaf at time $t$}{
                $\lambda^{\mathrm{hist}}_t\leftarrow\sum_{k=1}^{t-1}w_kQ_{t-k}$ and
                $\lambda \leftarrow \hat{\mu}_\tau(t)+\lambda^{\mathrm{hist}}_t$\\
                Sample $X_t^+, X_t^- \sim \mathrm{Poisson}(\lambda)$ independently\\
                Set $Q_t \leftarrow X_t^+ + \beta X_t^-$ and update $N_t^\pm$\\
                $S_t \leftarrow S_0\exp\!\left(c_\tau(N_t^+ - N_t^-) - d_\tau N_t^+\right)$\\
                Update the running payoff summaries
            }{
                Propagate the convolution directly for small blocks and by FFT otherwise
            }
        }
        Accumulate payoff first and second moments
    }
}
Aggregate thread-local moments and form estimates and 95\% confidence intervals\;
\end{algorithm}


\subsection{Rough-regime numerical experiments}

We use the simulator from Subsection~\ref{subsec:numerical-workflow} with the
rough-Heston parameters
\[
(\alpha,\gamma,\rho,\nu,\theta,V_0,S_0)
=(0.62,0.1,-0.681,0.331,0.3156,0.0392,100).
\]
The mapping in \eqref{eq:rho-def}--\eqref{eq:mu-nu-map} gives
$(\beta,\mu,\xi_0)=(27.5583,26.8592,0.124208)$.

\subsubsection{Accuracy--cost comparison and discretization choice}
\label{sec:convergence_analysis}

For a European call with $K=100$, Table~\ref{tab:tau_convergence} reports the
accuracy--cost trade-off at $T=1$ as $\tau$ increases from $40$ to $320$.

\begin{table}[ht]
    \centering
    \caption{European call estimates ($K=100$) across the tested discretization levels. All simulations were run with 500,000 paths; the intervals are 95\% Monte Carlo confidence intervals.}
    \label{tab:tau_convergence}
    \small
    \setlength{\tabcolsep}{5pt}
    \begin{tabular}{ccccc}
        \toprule
        \textit{$\tau$} & \textit{\shortstack{Estimated\\ price}} & \textit{95\% CI} & \textit{\shortstack{Deviation from\\ benchmark}} & \textit{\shortstack{Time\\ (seconds)}}\\
        \midrule
        40    & 9.5198 & [9.4767, 9.5629] & $+0.0461$ & 0.39\\
        80    & 9.4497 & [9.4069, 9.4925] & $-0.0240$ & 0.83\\
        160   & 9.4556 & [9.4128, 9.4984] & $-0.0181$ & 2.05\\
        320   & 9.4908 & [9.4480, 9.5336] & $+0.0171$ & 4.83\\
        \midrule
        \multicolumn{2}{c}{\textit{Benchmark (\cite{callegaro2021fast}): 9.4737}} & -- & -- & -- \\
        \bottomrule
\end{tabular}
\end{table}

The benchmark lies inside the reported Monte Carlo confidence intervals for
$\tau=80,160$, and $320$ and is $0.0030$ below the lower endpoint for
$\tau=40$. The deviations at $\tau=160$ and $\tau=320$ are indistinguishable
at the reported Monte Carlo precision. We use the finest tested grid,
$\tau=320$, in the subsequent classical-benchmark
and surface experiments.

\subsubsection{Comparison with existing \texorpdfstring{$\alpha<1$}{alpha<1} schemes on implied-volatility slices}
\label{sec:iv-slice-comparison}

We next consider a fixed-maturity implied-volatility slice and compare our INAR($\infty$)-based simulator with three representative alternatives from the recent literature in the regime $\alpha<1$. More precisely, we use the rough Heston parameter set
\[
\gamma=0.3,\ \theta=\frac{0.02}{0.3},\ \nu=1,\ \rho=-0.7,\ V_0=0.02,\ S_0=100,\ T=1,
\]
and the log-moneyness grid $k=\log(K/S_0)\in\{-0.5,-0.4,\ldots,0.5\}$. We compare the FFT-accelerated INAR($\infty$) scheme with a hybrid quadratic-exponential (HQE) weak scheme from the affine-forward-variance literature \cite{Gatheral2022,BayerBreneis2024WeakRoughHeston}, the integrated-variance Euler scheme of \cite{richard2023discrete} combined with a three-factor completely monotone kernel approximation in the spirit of \cite{AlfonsiKebaier2024}, and the inverse-Gaussian iVi scheme of \cite{AbiJaberAttal2025}. For all methods, we use $2^{24}$ Monte Carlo paths, $10$ worker threads, and a common time discretization of $160$ steps. The reference slice is computed by Fourier inversion of the rough Heston characteristic function.

The resulting smile overlays are shown in Figure~\ref{fig:iv_slice_panel_alpha}.
Slice-level accuracy varies across methods and roughness levels. The
corresponding maximum absolute slice errors and wall-clock times are reported in
Table~\ref{tab:iv_slice_panel_alpha}.

\begin{table}[ht]
\centering
\caption{Slice-level comparison for the multi-$\alpha$ experiment in Figure~\ref{fig:iv_slice_panel_alpha}. Each entry reports the maximum absolute IV-slice error on the 11-point log-moneyness grid, followed by the wall-clock time in seconds on the same 10-thread run.}
\label{tab:iv_slice_panel_alpha}
\small
\setlength{\tabcolsep}{5pt}
\begin{tabular}{ccccc}
\toprule
$\alpha$ & Integrated (3 factors) & INAR-CDQ-FFT & iVi & HQE \\
\midrule
0.55 & $0.0080$ ($11.9$s) & $0.0058$ ($41.3$s) & $0.0014$ ($29.9$s) & $0.0006$ ($28.2$s) \\
0.62 & $0.0058$ ($13.0$s) & $0.0040$ ($46.8$s) & $0.0014$ ($28.5$s) & $0.0006$ ($32.6$s) \\
0.80 & $0.0080$ ($12.9$s) & $0.0019$ ($58.0$s) & $0.0013$ ($28.4$s) & $0.0006$ ($34.7$s) \\
0.95 & $0.0577$ ($12.7$s) & $0.0011$ ($64.6$s) & $0.0021$ ($28.8$s) & $0.0010$ ($33.1$s) \\
\bottomrule
\end{tabular}
\end{table}

At the common step budget, HQE and iVi have the two smallest slice errors for
$\alpha=0.55$, $0.62$, and $0.80$. The INAR error decreases from $0.0058$ to
$0.0011$ as $\alpha$ increases; at $\alpha=0.95$ it is second only to HQE and
below iVi. The three-factor integrated approximation has errors $0.0080$,
$0.0058$, and $0.0080$ for the first three slices, while at $\alpha=0.95$ its
error is $0.0577$ and its smile is visibly flatter than the reference. The
path-dependent experiments below use the same INAR microstructural simulator.

\begin{figure}[ht]
    \centering
    \includegraphics[width=0.92\textwidth]{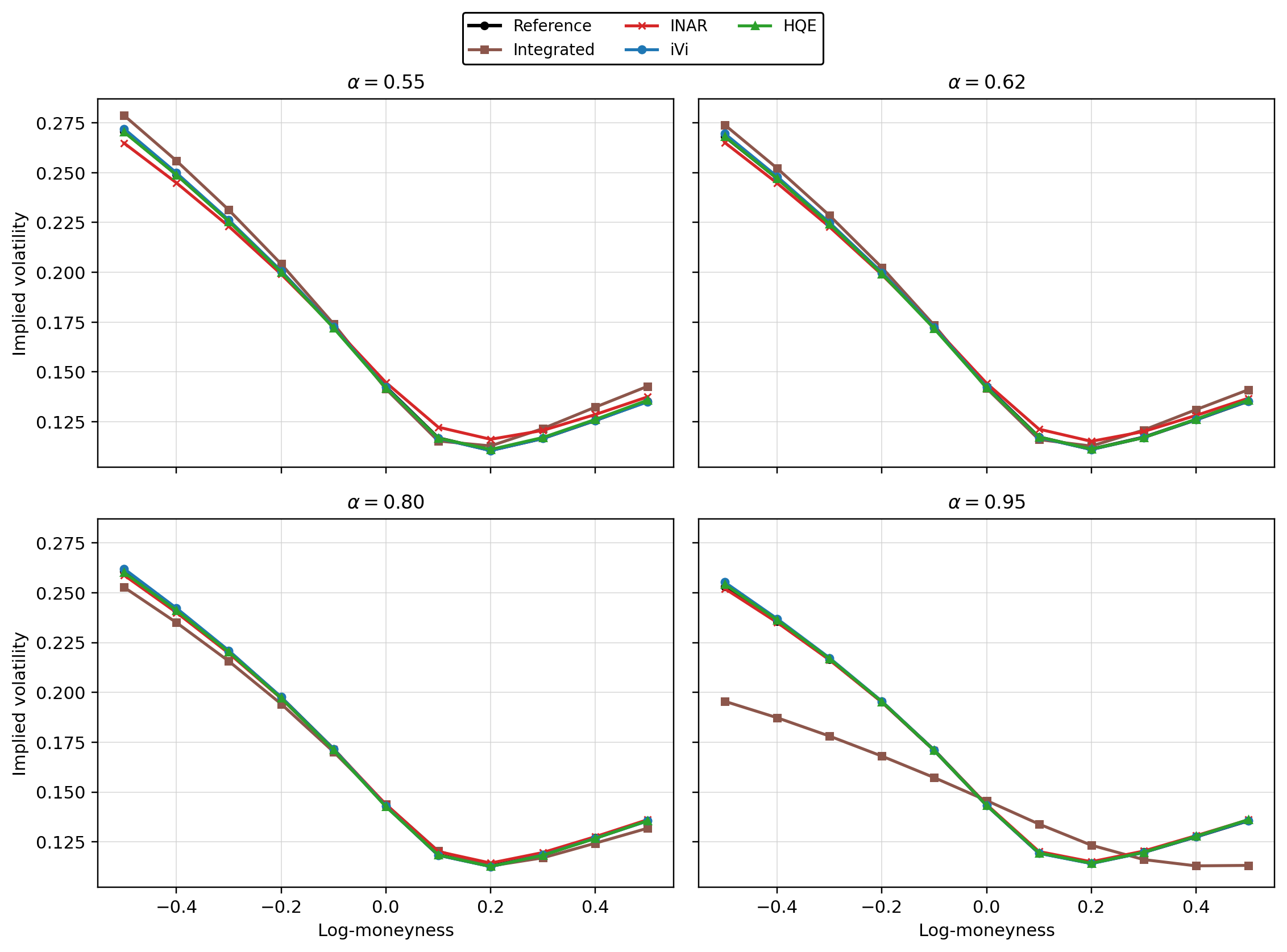}
    \caption{Implied-volatility slice comparison for $\alpha\in\{0.55,0.62,0.80,0.95\}$ at maturity $T=1$. In each panel, the black curve is the Fourier reference smile, the red curve is the INAR($\infty$)-CDQ-FFT slice, the blue and green curves are the iVi and HQE weak-scheme benchmarks, and the brown curve is the $3$-factor integrated multifactor Euler approximation. All four simulation methods use $2^{24}$ Monte Carlo paths.}
    \label{fig:iv_slice_panel_alpha}
\end{figure}


\subsection{Classical Heston benchmark \texorpdfstring{$(\alpha=1)$}{(alpha=1)}}

At $\alpha=1$, we use the Markovian kernel from
Remark~\ref{rem:alpha-one-kernel} and compare directly with closed-form Heston
prices, keeping the remaining macroscopic parameters from the rough-regime
experiment fixed.  Table~\ref{tab:european_comparison} uses $T=1$, $\tau=320$,
$S_0=100$, $K=80,90,100,110,120$, and $500{,}000$ paths.

\begin{table}[ht]
  \centering
  \caption{Comparison of European option prices between the closed-form solution and the INAR($\infty$) implementation}
  \small
  \setlength{\tabcolsep}{5pt}
  \begin{tabular*}{\textwidth}{@{\extracolsep{\fill}}lcccccc@{}}
  \toprule
  \multirow{2}{*}{Strike} & \multicolumn{2}{c}{Closed-form solution} & \multicolumn{4}{c}{%
  INAR($\infty$)} \\
  \cmidrule(lr){2-3}\cmidrule(lr){4-7}
   & Call & Put & Call & 95\% CI & Put & 95\% CI \\
  \midrule
  80  & 21.8822 & 1.8822 & 21.8835 & [21.8267, 21.9404] & 1.8752 & [1.8611, 1.8892] \\
  90  & 14.6187 & 4.6187 & 14.6203 & [14.5705, 14.6702] & 4.6120 & [4.5888, 4.6352] \\
  100 & 9.0983  & 9.0983  & 9.0943  & [9.0532, 9.1355]  & 9.0860  & [9.0529, 9.1191] \\
  110 & 5.2883  & 15.2883 & 5.2811  & [5.2491, 5.3132]  & 15.2728 & [15.2305, 15.3151] \\
  120 & 2.8849  & 22.8849 & 2.8807  & [2.8569, 2.9045]  & 22.8724 & [22.8226, 22.9222] \\
  \bottomrule
  \end{tabular*}
  \label{tab:european_comparison}
\end{table}

\Needspace{6\baselineskip}
\paragraph{Path-dependent benchmark comparison}
We compare the path simulator at $K=100$ with a nonnegative-projection Euler
Heston benchmark on the same $1/320$ grid and under the monitoring conventions
in \eqref{Asian:defn}. The Euler estimates use $500{,}000$ antithetic path pairs
and the INAR estimates use $500{,}000$ exact-Poisson paths; the up-and-in and
down-and-out barriers are $110$ and $90$, respectively.
\begin{table}[ht]
\centering
\caption{Classical-Heston path-dependent benchmark comparison at $K=100$.}
\label{tab:pathdependent-classical-summary}
\small
\begin{tabular}{lccc}
\toprule
Payoff & Euler benchmark & INAR estimate & 95\% CI \\
\midrule
Asian call       & 4.9135  & 4.9281  & [4.9066, 4.9495] \\
Asian put        & 4.9151  & 4.9202  & [4.9015, 4.9388] \\
Lookback call    & 18.4117 & 18.4428 & [18.3981, 18.4875] \\
Lookback put     & 16.5131 & 16.5221 & [16.4899, 16.5543] \\
Up-and-in call   & 8.9723  & 8.9830  & [8.9417, 9.0242] \\
Down-and-out put & 0.1495  & 0.1510  & [0.1485, 0.1535] \\
\bottomrule
\end{tabular}
\end{table}

The INAR estimates track the European and path-dependent benchmarks. For the
path-dependent option set, the Euler and INAR runtimes are $2.4$ and $4.8$
seconds, respectively.


\subsection{Analysis of the implied volatility surface}

We examine the model-implied volatility (IV) surface generated by the chosen
rough Heston specification, including its short-maturity ATM skew and its
comparison with a Markovian benchmark. Related empirical evidence on
power-law skew and rough-versus-Markovian performance is discussed in
\cite{guyon2022does,abi2025volatility}.

Following the simulation setup in Section~\ref{sec:numerical}, we price European calls on a grid of maturities and log-moneyness
$k=\log(K/S_0)\in\{-0.20,-0.18,\ldots,0.20\}$. For each $(T,k)$ we invert the Black--Scholes formula to obtain IVs. Figure~\ref{fig:iv_surface} compares the resulting surface for the rough specification ($\alpha=0.62$) with the classical-Heston benchmark ($\alpha=1$). The classical benchmark uses the lag-one kernel from Remark~\ref{rem:alpha-one-kernel}, with the remaining macroscopic parameters and the numerical grid held fixed. We consider maturities in $[1/12,1]$, sampled at monthly grid points.

\begin{figure}[ht]
    \centering
    \includegraphics[width=0.7\textwidth]{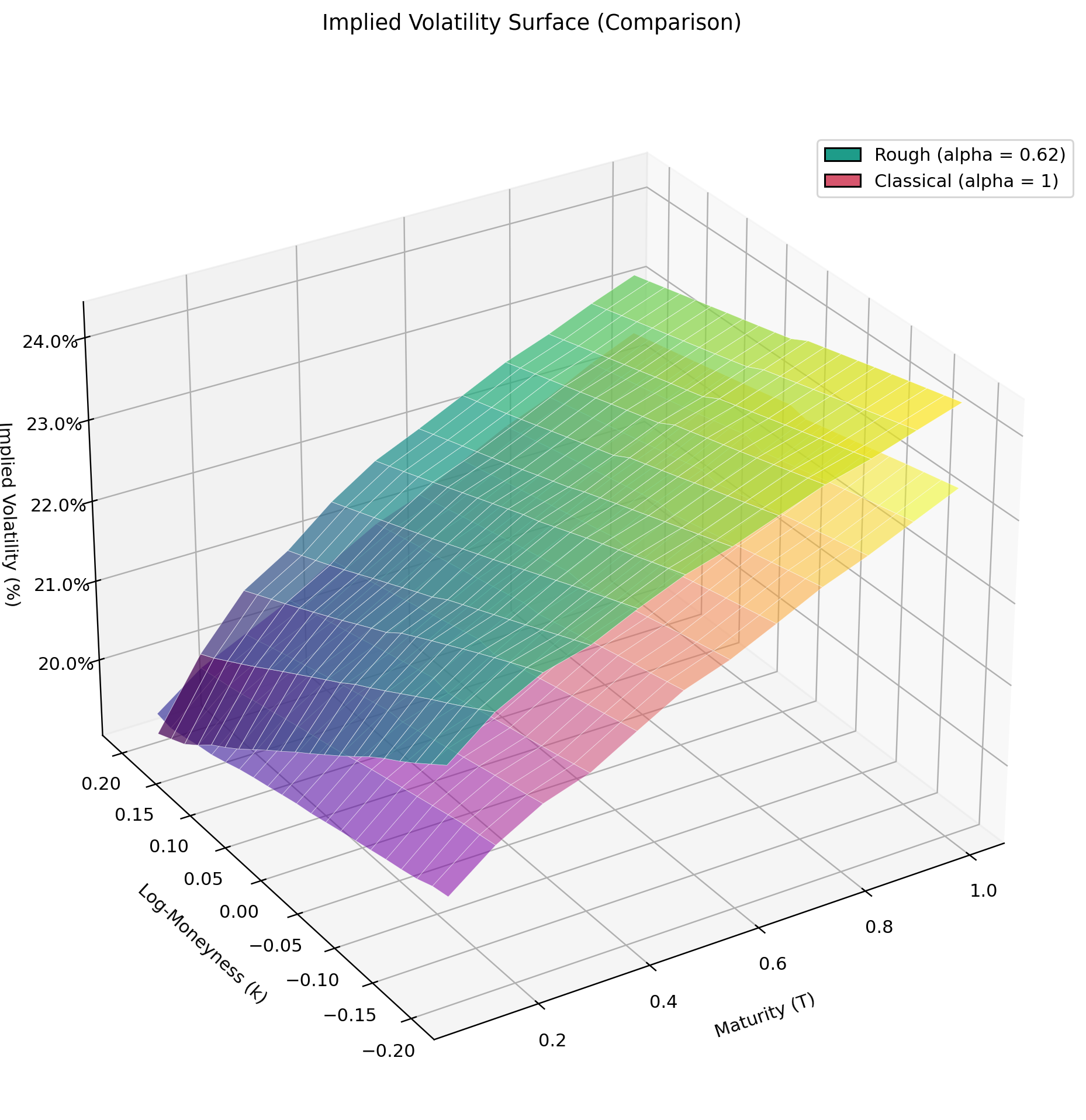}
    \caption{Comparison of model-implied volatility surfaces for the rough specification ($\alpha=0.62$) and the classical-Heston benchmark ($\alpha=1$), both produced with the INAR($\infty$)-FFT simulator. The classical benchmark uses the lag-one kernel from Remark~\ref{rem:alpha-one-kernel}, with the remaining macroscopic parameters fixed. Grid: $T\in\{1/12,2/12,3/12,\ldots,1\}$ and 21 equally spaced log-moneyness points in $[-0.2,0.2]$. All runs use grid frequency $\tau=320$ per year (so $\Delta t=1/320$) and $10^6$ paths.}
    \label{fig:iv_surface}
\end{figure}

\subsubsection{ATM skew and implied-volatility term structure}

Figure~\ref{fig:atm_skew_panel} (a) traces the finite-difference ATM skew across maturities, while Figure~\ref{fig:atm_skew_panel} (b) shows the corresponding ATM implied volatilities. The skew is computed by a centered difference using $k=\pm0.02$. In our rough Heston specification, the skew steepens as maturity shrinks, whereas the ATM level follows a smoother upward term structure. A log--log least-squares fit over the twelve simulated maturities (Figure~\ref{fig:atm_skew_panel}, panel (c)) gives
$|\mathrm{skew}(T)|\approx0.0373T^{-0.266}$ with $R^2\approx0.772$.
This fit describes the finite-horizon relation on $T\in[1/12,1]$, whereas the
rough-Heston exponent $\alpha-1=-0.38$ describes the asymptotic regime
$T\downarrow0$.

\begin{figure}[ht]
    \centering
    \includegraphics[width=0.88\textwidth]{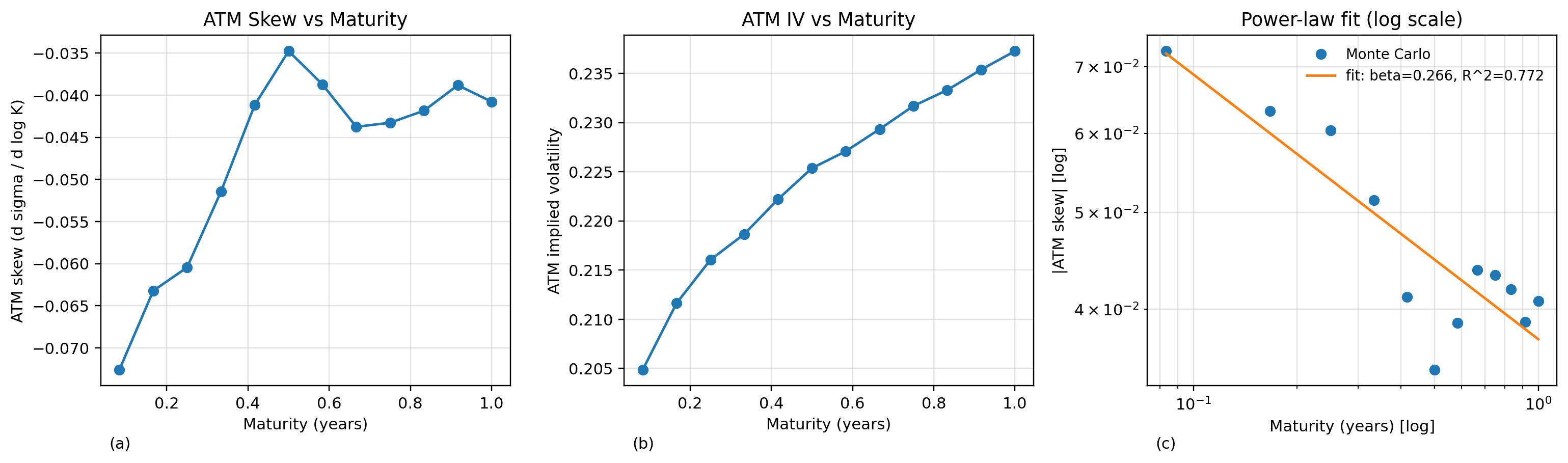}
    \caption{ATM skew and implied volatility under the rough specification ($\alpha=0.62$). Panel (a) shows the steepening skew toward short maturities; panel (b) reports the smoother ATM level term structure; panel (c) shows a finite-horizon least-squares power-law fit over the twelve simulated maturities.}
    \label{fig:atm_skew_panel}
\end{figure}


\FloatBarrier
\section{Conclusion}\label{sec:conclusion}

We introduce a bivariate cumulative heavy-tailed INAR($\infty$) framework whose nearly unstable scaling limit is the rough Heston model of \cite{el2019characteristic}, extending to the bivariate price--variance setting the discrete cumulative INAR scaling picture developed in \cite{wang2026scaling}. This provides a discrete-time microstructural interpretation of the joint variance-price dynamics and makes the link between the asymmetry of the microscopic order-flow mechanism and the macroscopic leverage effect explicit through a concrete parameter mapping.

On the pricing side, under the continuous transform-strip and bounded-Riccati assumptions of Section~\ref{sec:weak-euro}, the quadratic discrete Volterra reduction gives the bounded-window transform-error scale $\tau^{-\alpha}+\tau^{-(1-\alpha)}$, and uniform $\mathcal O(R^{-1})$ Fourier tails yield convergence of the full call prices, locally uniformly in log-moneyness.

From a computational perspective, the INAR($\infty$) representation provides
a unified simulation architecture for European and path-dependent options.
The CDQ--FFT implementation reduces the long-memory convolution cost to
$\mathcal O(\tau\log^2\tau)$ per path. Across the European and path-dependent
payoffs, the $\alpha=1$ estimates agree with the classical Heston benchmarks;
the rough-regime experiments exhibit the implied-volatility patterns of the
chosen specification. On the European IV slices, HQE and iVi have the two
smallest errors for $\alpha=0.55$, $0.62$, and $0.80$, while INAR is second
only to HQE at $\alpha=0.95$; the same INAR simulator also covers the
path-dependent payoffs.

Natural theoretical extensions are a direct characterization of the
continuous-model transform strip and quantitative high-frequency
transform bounds yielding a full-price rate. Numerical extensions include
calibration, variance reduction, statistical inference for the
microstructural INAR model, and broader comparisons with rough-volatility
schemes.

\section*{Acknowledgments}

The authors would like to thank the AE and two anonymous referees for helpful comments. Lingjiong Zhu is partially supported by the grants NSF DMS-2053454, NSF DMS-2208303. 

\bibliographystyle{plainnat}
\bibliography{bibtex}

@article{rebolledo1980central,
  title   = {Central limit theorems for local martingales},
  author  = {Rebolledo, Rolando},
  journal = {Zeitschrift f{\"u}r Wahrscheinlichkeitstheorie und Verwandte Gebiete},
  volume  = {51},
  number  = {3},
  pages   = {269--286},
  year    = {1980}
}

@article{carrmadan1999fft,
  author  = {Carr, Peter and Madan, Dilip B.},
  title   = {Option Valuation Using the Fast {F}ourier Transform},
  journal = {Journal of Computational Finance},
  year    = {1999},
  volume  = {2},
  number  = {4},
  pages   = {61--73}
}

@article{abi2025volatility,
  title={Volatility models in practice: Rough, path-dependent, or {M}arkovian?},
  author={Abi Jaber, Eduardo and Li, Shaun},
  journal={Mathematical Finance},
  volume={35},
  number={4},
  pages={796--817},
  year={2025},
  publisher={Wiley Online Library}
}

@article{boyarchenko2025fast,
  title={Fast reliable pricing and calibration of the rough {H}eston model},
  author={Boyarchenko, S. and de Innocentis, M. and Levendorski{\u\i}, S.},
  journal={arXiv preprint arXiv:2508.15080},
  year={2025}
}

@Article{Gassiat2023WeakRoughVol,
  author = {Gassiat, Paul},
  title = {Weak Error Rates of Numerical Schemes for Rough Volatility},
  journal = {SIAM Journal on Financial Mathematics},
  year = {2023},
  volume = {14},
  number = {2},
  pages = {475--496}
}

@article{BayerHallTempone2021WeakLinearRoughVol,
  title={Weak error rates for option pricing under linear rough volatility},
  author={Bayer, Christian and Hall, Eric Joseph and Tempone, Ra{\'u}l},
  journal={International Journal of Theoretical and Applied Finance},
  volume={25},
  number={07n08},
  pages={2250029},
  year={2022},
  publisher={World Scientific}
}

@article{FrizSalkeldWagenhofer2024WeakRoughVol,
  title={Weak error estimates for rough volatility models},
  author={Friz, Peter K and Salkeld, William and Wagenhofer, Thomas},
  journal={Annals of Applied Probability},
  volume={35},
  number={1},
  pages={64--98},
  year={2025},
  publisher={Institute of Mathematical Statistics}
}

@Article{abi2019affine,
  title = {Affine {V}olterra processes},
  author = {Abi Jaber, E. and Larsson, M. and Pulido, S.},
  year = {2019},
  journal = {Annals of Applied Probability},
  volume = {29},
  number = {5},
  pages = {3155--3200},
}

@Article{ma2023optimal,
  title = {Optimal reinsurance-investment with loss aversion under rough {H}eston model},
  author = {Ma, J. and Lu, Z. and Chen, D.},
  journal = {Quantitative Finance},
  volume = {23},
  number = {1},
  pages = {95--109},
  year = {2023}
}

@Article{BayerBreneis2024WeakRoughHeston,
  author = {Bayer, C. and Breneis, S.},
  title = {Efficient option pricing in the rough {H}eston model using weak simulation schemes},
  journal = {Quantitative Finance},
  year = {2024},
  volume = {24},
  number = {9},
  pages = {1247--1261}
}

@Article{KellerResselMajid2020Comparison,
  title = {A comparison principle between rough and non-rough {H}eston models—with applications to the volatility surface},
  author = {Keller-Ressel, M. and Majid, A.},
  journal = {Quantitative Finance},
  volume = {20},
  number = {6},
  pages = {919--933},
  year = {2020}
}

@Article{BayerFrizGatheral2015PricingRough,
  title = {Pricing under rough volatility},
  author = {Bayer, C. and Friz, P. and Gatheral, J.},
  journal = {Quantitative Finance},
  volume = {16},
  number = {6},
  pages = {887--904},
  year = {2016}
}

@Article{jaisson2015limit,
  title = {{Limit Theorems for nearly unstable Hawkes processes}},
  author = {Jaisson, T. and Rosenbaum, M.},
  journal = {Annals of Applied Probability},
  volume = {25},
  number = {2},
  pages = {600--631},
  year = {2015}
}

@Article{callegaro2021fast,
  title = {{Fast hybrid schemes for fractional Riccati equations (rough is not so tough)}},
  author = {Callegaro, G. and Grasselli, M. and Pages, G.},
  journal = {Mathematics of Operations Research},
  volume = {46},
  number = {1},
  pages = {221--254},
  year = {2021},
  publisher = {INFORMS}
}

@Article{fokianos2021multivariate,
  title = {Multivariate Count Time Series Modelling},
  journal = {Econometrics and Statistics},
  volume = {31},
  pages = {100-116},
  year = {2024},
  author = {Fokianos, K.}
}

@Article{huang2023nonlinear,
  title = {{Nonlinear Poisson autoregression and nonlinear Hawkes processes}},
  author = {Huang, L. and Khabou, M.},
  journal = {Stochastic Processes and their Applications},
  volume = {161},
  pages = {201--241},
  year = {2023},
  publisher = {Elsevier}
}

@Article{liang2017strong,
  title = {{Strong superconvergence of the Euler-Maruyama method for linear stochastic Volterra integral equations}},
  author = {Liang, H. and Yang, Z. and Gao, J.},
  journal = {Journal of Computational and Applied Mathematics},
  volume = {317},
  pages = {447--457},
  year = {2017},
  publisher = {Elsevier}
}

@Article{abi2019multifactor,
  title = {Multifactor approximation of rough volatility models},
  author = {Abi Jaber, E. and El Euch, O.},
  journal = {SIAM Journal on Financial Mathematics},
  volume = {10},
  number = {2},
  pages = {309--349},
  year = {2019},
  publisher = {SIAM}
}

@Article{ma2022fast,
  title = {A fast algorithm for simulation of rough volatility models},
  author = {Ma, J. and Wu, H.},
  journal = {Quantitative Finance},
  volume = {22},
  number = {3},
  pages = {447--462},
  year = {2022},
  publisher = {Taylor \& Francis}
}

@Article{yang2024,
  title = {A general valuation framework for rough stochastic local volatility models and applications},
  journal = {European Journal of Operational Research},
  volume = {322},
  number = {1},
  pages = {307-324},
  year = {2025},
  author = {Yang, W. and Ma, J. and Cui, Z.},
}

@Article{bacry2013some,
  title = {Some limit theorems for {H}awkes processes and application to financial statistics},
  author = {Bacry, E. and Delattre, S. and Hoffmann, M. and Muzy, J.-F.},
  journal = {Stochastic Processes and their Applications},
  volume = {123},
  number = {7},
  pages = {2475--2499},
  year = {2013},
  publisher = {Elsevier}
}

@Article{jaisson2016rough,
  title = {{Rough fractional diffusions as scaling limits of nearly unstable heavy tailed Hawkes processes}},
  author = {Jaisson, T. and Rosenbaum, M.},
  journal = {Annals of Applied Probability},
  volume = {26},
  number = {5},
  pages = {2860--2882},
  year = {2016}
}

@Article{gerhold2019moment,
  title = {Moment Explosions in the Rough {H}eston Model},
  author = {Gerhold, S. and Gerstenecker, C. and Pinter, A.},
  journal = {Decisions in Economics and Finance},
  volume = {42},
  number = {2},
  pages = {575--608},
  year = {2019}
}

@Article{gatheral2018volatility,
  title = {Volatility Is Rough},
  author = {Gatheral, J. and Jaisson, T. and Rosenbaum, M.},
  journal = {Quantitative Finance},
  volume = {18},
  number = {6},
  pages = {933--949},
  year = {2018}
}

@Book{jacod2013limit,
  title = {Limit Theorems for Stochastic Processes},
  author = {Jacod, J. and Shiryaev, A.},
  volume = {288},
  year = {2013},
  publisher = {Springer Science \& Business Media}
}

@Article{kirchner2016hawkes,
  title = {{H}awkes and {INAR} ($\infty$) processes},
  author = {Kirchner, M.},
  journal = {Stochastic Processes and their Applications},
  volume = {126},
  number = {8},
  pages = {2494--2525},
  year = {2016},
  publisher = {Elsevier}
}

@Book{revuz2013continuous,
  title = {Continuous Martingales and Brownian Motion},
  author = {Revuz, D. and Yor, M.},
  volume = {293},
  year = {2013},
  publisher = {Springer Science \& Business Media}
}

@Article{xu2022self,
  title = {A self-and mutual-exciting model for discrete-time data: Case study on online money market fund},
  author = {Xu, Y. and Zhu, L. and Wang, H.},
  journal = {Available at SSRN 3665436},
  year = {2022}
}

@Article{hawkes1971spectra,
  title = {Spectra of some self-exciting and mutually exciting point processes},
  author = {Hawkes, A.},
  journal = {Biometrika},
  volume = {58},
  number = {1},
  pages = {83--90},
  year = {1971},
  publisher = {Oxford University Press}
}

@Article{alzaid1990integer,
  title = {An integer-valued $p$th-order autoregressive structure ({INAR}($p$)) process},
  author = {Alzaid, A. and Al-Osh, M.},
  journal = {Journal of Applied Probability},
  volume = {27},
  number = {2},
  pages = {314--324},
  year = {1990},
  publisher = {Cambridge University Press}
}

@Article{jin1991integer,
  title = {The integer-valued autoregressive ({INAR}(p)) model},
  author = {Du, J.-G. and Li, Y.},
  journal = {Journal of Time Series Analysis},
  volume = {12},
  number = {2},
  pages = {129--142},
  year = {1991},
  publisher = {Wiley Online Library}
}

@Article{horst2026functional,
  title  = {Functional limit theorems for {H}awkes processes},
  author = {Horst, U. and Xu, W.},
  journal= {Probability Theory and Related Fields},
  volume = {194},
  pages  = {917--996},
  year   = {2026}
}

@Article{horst2022microstructure,
  title = {The microstructure of stochastic volatility models with self-exciting jump dynamics},
  author = {Horst, U. and Xu, W.},
  journal = {Annals of Applied Probability},
  volume = {32},
  number = {6},
  pages = {4568--4610},
  year = {2022},
  publisher = {Institute of Mathematical Statistics}
}

@Article{el2019characteristic,
  title = {{The characteristic function of rough Heston models}},
  author = {El Euch, O. and Rosenbaum, M.},
  journal = {Mathematical Finance},
  volume = {29},
  number = {1},
  pages = {3--38},
  year = {2019},
  publisher = {Wiley Online Library}
}

@Article{wang2026scaling,
  title={Scaling limit of heavy-tailed nearly unstable cumulative {INAR} ($\infty$) processes and rough fractional diffusions},
  author={Cai, Chunhao and He, Ping and Wang, Qinghua and Wang, Yingli},
  journal={Methodology and Computing in Applied Probability},
  volume = {28},
  year = {2026},
  eid = {13},
  publisher = {Springer}
}

@Article{heston1993closed,
  title = {A closed-form solution for options with stochastic volatility with applications to bond and currency options},
  author = {Heston, S. L.},
  journal = {Review of Financial Studies},
  volume = {6},
  number = {2},
  pages = {327--343},
  year = {1993},
  publisher = {Oxford University Press}
}

@Article{el2018microstructural,
  title = {The microstructural foundations of leverage effect and rough volatility},
  author = {El Euch, O. and Fukasawa, M. and Rosenbaum, M.},
  journal = {Finance and Stochastics},
  volume = {22},
  pages = {241--280},
  year = {2018},
  publisher = {Springer}
}

@Article{richard2023discrete,
  title = {{On the discrete-time simulation of the rough Heston model}},
  author = {Richard, A. and Tan, X. and Yang, F.},
  journal = {SIAM Journal of Financial Mathematics},
  volume = {14},
  number = {1},
  pages = {223--249},
  year = {2023}
}

@Article{RICHARD2021109,
  title = {Discrete-time simulation of stochastic {V}olterra equations},
  author = {Richard, A. and Tan, X. and Yang, F.},
  journal = {Stochastic Processes and their Applications},
  volume = {141},
  pages = {109-138},
  year = {2021}
}

@article{guyon2022does,
  author  = {Guyon, Julien and El Amrani, Mehdi},
  title   = {Does the term-structure of equity at-the-money skew really follow a power law?},
  journal = {Risk},
  year    = {2022}
}

@article{Gatheral2022,
  author  = {Gatheral, Jim},
  title   = {Efficient simulation of affine forward variance models},
  journal = {Risk},
  year    = {2022}
}

@article{AlfonsiKebaier2024,
  author  = {Alfonsi, Aur{\'e}lien and Kebaier, Ahmed},
  title   = {Approximation of stochastic Volterra equations with kernels of completely monotone type},
  journal = {Mathematics of Computation},
  volume  = {93},
  number  = {346},
  pages   = {643--677},
  year    = {2024}
}

@article{AbiJaberAttal2025,
  author  = {Abi Jaber, Eduardo and Attal, Elie},
  title   = {{S}imulating integrated {V}olterra square-root processes and {V}olterra {H}eston models via {I}nverse {G}aussian},
  journal = {arXiv preprint arXiv:2504.19885},
  year    = {2025}
}

@article{AbiJaber2021WeakExistence,
  author  = {Abi Jaber, Eduardo},
  title   = {Weak existence and uniqueness for affine stochastic {V}olterra equations with {$L^1$}-kernels},
  journal = {Bernoulli},
  volume  = {27},
  number  = {3},
  pages   = {1583--1615},
  year    = {2021}
}

@article{garsia1962discrete,
  author  = {Garsia, Adriano M. and Lamperti, John},
  title   = {A discrete renewal theorem with infinite mean},
  journal = {Commentarii Mathematici Helvetici},
  volume  = {37},
  pages   = {221--234},
  year    = {1962/1963}
}

@book{protter2005stochastic,
  author    = {Protter, Philip E.},
  title     = {Stochastic Integration and Differential Equations},
  edition   = {Second},
  series    = {Stochastic Modelling and Applied Probability},
  volume    = {21},
  publisher = {Springer},
  address   = {Berlin},
  year      = {2005}
}

\newpage
\appendix

\section{Proofs for the Scaling-Limit Section}
\label{app:scaling-proofs}

Throughout Appendices~\ref{app:scaling-proofs} and \ref{app:weak-error-proofs}, we use the following asymptotic notation, with the relevant limiting regime specified by the context. We write $f \sim g$ to denote asymptotic equivalence, i.e., $f/g\to 1$ in the stated limit. We use $f \lesssim g$ to indicate that $f \le C g$ for some positive constant $C$ independent of the asymptotic parameter in the stated regime. We write $x \asymp y$ to denote that there exist constants $0<c<C<\infty$
such that $cy \le x \le Cy$ uniformly in the stated regime.

\subsection{Renewal-kernel convergence and pointwise estimates}

\begin{lemma}[Weak convergence of the renewal kernel]
\label{lem:kernel_convergence}
Let $(\psi_n^\tau)_{n\ge1}$ be the sequence of discrete renewal kernels associated
with $\varphi^\tau=a_\tau\varphi$, and define
\[
  \mathsf S^\tau(t)
  :=
  \frac{1-a_\tau}{a_\tau}\tau\,\psi^\tau_{\lfloor t\tau\rfloor},
  \qquad t\ge0.
\]
Here and below we use the convention $\psi_0^\tau:=0$.
Then, as $\tau\to\infty$, the measure $\mathsf S^\tau(t)\,dt$ converges weakly on
$[0,\infty)$ to the probability measure with density
\[
  f_{\alpha,\gamma}(t)
  :=
  \gamma t^{\alpha-1}E_{\alpha,\alpha}(-\gamma t^\alpha),
  \qquad t>0.
\]
\end{lemma}

\begin{proof}
This is Lemma~2.4 of \cite{wang2026scaling}, specialized to
$1-a_\tau=\gamma\tau^{-\alpha}$.  Indeed, the tail normalization in
Definition~\ref{def:params} gives
$1-\widehat\varphi(s)\sim s^\alpha$, and therefore
\[
  \widehat{\mathsf S^\tau}(z)
  =(1-a_\tau)\frac{\tau(1-e^{-z/\tau})}{z}
    \frac{\widehat\varphi(z/\tau)}
         {1-a_\tau\widehat\varphi(z/\tau)}
  \longrightarrow \frac{\gamma}{\gamma+z^\alpha}.
\]
The limit is the Laplace transform of $f_{\alpha,\gamma}$, and the
continuity theorem for Laplace transforms yields the claim.
\end{proof}

For later use, define the discrete cumulative resolvent by $\mathcal G_0^\tau:=0$
and
\[
  \mathcal G_n^\tau
  :=
  \frac{1-a_\tau}{\gamma a_\tau}
  \sum_{k=1}^n\psi_k^\tau,
  \qquad n\ge 1.
\]

\begin{lemma}[Pointwise renewal-kernel bound]
\label{lem:psi-pointwise}
For every $T>0$, there exists $C_T>0$, independent of $\tau$, such
that
\begin{equation}\label{eq:psi-pointwise}
  \psi_n^\tau\le C_T n^{\alpha-1},
  \qquad 1\le n\le m_\tau.
\end{equation}
Equivalently,
\begin{equation}\label{eq:G-increment-pointwise}
  \Delta\mathcal G_n^\tau
  :=
  \mathcal G_n^\tau-\mathcal G_{n-1}^\tau
  =
  \frac{1-a_\tau}{\gamma a_\tau}\psi_n^\tau
  \le C_T\tau^{-\alpha}n^{\alpha-1},
  \qquad 1\le n\le m_\tau.
\end{equation}
\end{lemma}

\begin{proof}
Let
$u_n:=\sum_{m\ge1}(\varphi^{*m})_n$,
$n\ge1$,
be the renewal mass sequence associated with the probability distribution
$\varphi$ on $\mathbb N$.  By the telescoping definition of $\varphi$,
\[
  \overline\varphi_n
  :=\sum_{j=n+1}^{\infty}\varphi_j
  =\frac{n^{-\alpha}}{\Gamma(1-\alpha)},
  \qquad n\ge1.
\]
Since $\alpha\in(1/2,1)$ and $\varphi$ has lattice span one, the arithmetic
strong renewal theorem \cite{garsia1962discrete}---that is, the local
renewal theorem for an integer-valued interarrival distribution with
infinite mean---gives
$u_n\sim\frac{1}{\Gamma(\alpha)}n^{\alpha-1}$,
$n\to\infty$.
Consequently, after enlarging the constant to cover finitely many small
values of $n$,
\[
  u_n\le C n^{\alpha-1},
  \qquad n\ge1.
\]
On the other hand, since $\varphi^\tau=a_\tau\varphi$ and $0<a_\tau<1$,
\[
  \psi_n^\tau
  =\sum_{m\ge1}a_\tau^m(\varphi^{*m})_n
  \le u_n,
\]
which proves \eqref{eq:psi-pointwise}.  This bound is also recorded in
\cite[Lemma~2.7]{wang2026scaling}.  Finally,
$a_\tau\ge a_{\tau_{\min}}>0$, and hence multiplication by
$\frac{1-a_\tau}{\gamma a_\tau}
  =\frac{\tau^{-\alpha}}{a_\tau}$
yields \eqref{eq:G-increment-pointwise}.
\end{proof}

\subsection{Proof of Proposition~\ref{proposition1}}\label{app:proposition1}
Fix $T>0$.

\medskip
\noindent
\textit{Step 1: uniform moment bounds for the intensities.}
We first establish the moment estimates for the intensities that will be used throughout the proof.

For $1\le s\le \lfloor \tau T\rfloor$, using Definition~\ref{def:params} and
\[
  1-\sum_{u=1}^{s-1}\varphi_u^\tau
  =(1-a_\tau)+a_\tau\sum_{u=s}^{\infty}\varphi_u,
\]
we obtain
\begin{align*}
  \hat\mu_\tau(s)
  &=\mu_\tau+\xi_0\mu_\tau\left(
      \frac{1}{1-a_\tau}\left((1-a_\tau)+a_\tau\sum_{u=s}^{\infty}\varphi_u\right)
      -a_\tau\sum_{u=1}^{s-1}\varphi_u
    \right) \\
  &\le C\mu_\tau\left(
      1+\frac{1}{1-a_\tau}\sum_{u=s}^{\infty}\varphi_u
    \right).
\end{align*}
For the kernel from Definition~\ref{def:params},
\[
  \sum_{u=s}^{\infty}\varphi_u
  =\frac{1}{\Gamma(1-\alpha)}(s-1)^{-\alpha},
  \qquad s\ge2,
\]
while $\sum_{u=1}^{\infty}\varphi_u=1$. Since
$\mu_\tau=\mu\tau^{\alpha-1}$ and $1-a_\tau=\gamma\tau^{-\alpha}$, it follows that
\begin{equation}\label{eq:muhat-uniform-bound-app-prop1}
  \sup_{1\le s\le \lfloor \tau T\rfloor}\hat\mu_\tau(s)
  \lesssim \tau^{2\alpha-1}.
\end{equation}

Next fix $n\le \lfloor \tau T\rfloor$. By \eqref{eq:lambda-decomp},
\[
  \lambda_n^{\tau,+}
  =
  \hat\mu_\tau(n)
  +
  \sum_{s=1}^{n-1}\psi_{n-s}^\tau \hat\mu_\tau(s)
  +
  \kappa_n^{\tau,+},
\]
where
\[
  \kappa_n^{\tau,+}
  :=
  \frac1{1+\beta}\sum_{s=1}^{n-1}\psi_{n-s}^{\tau}
  \left(
    M_s^{\tau,+}-M_{s-1}^{\tau,+}
    +\beta\left(M_s^{\tau,-}-M_{s-1}^{\tau,-}\right)
  \right).
\]

We first bound the deterministic convolution term $\sum_{s=1}^{n-1}\psi_{n-s}^\tau \hat\mu_\tau(s)$. From the explicit estimate obtained above,
$\hat\mu_\tau(s)\lesssim \tau^{\alpha-1}+\tau^{2\alpha-1}s^{-\alpha}$,
uniformly for $1\le s\le \lfloor \tau T\rfloor$. Writing $j=n-s$, we obtain
\begin{align}\label{for:two:terms}
  \sum_{s=1}^{n-1}\psi_{n-s}^\tau \hat\mu_\tau(s)
  \lesssim
  \tau^{\alpha-1}\sum_{j=1}^{n-1}\psi_j^\tau
  +\tau^{2\alpha-1}\sum_{j=1}^{n-1}\psi_j^\tau (n-j)^{-\alpha}.
\end{align}

For the first term on the right side of \eqref{for:two:terms},
\[
  \tau^{\alpha-1}\sum_{j=1}^{n-1}\psi_j^\tau
  \le
  \tau^{\alpha-1}\sum_{j\ge1}\psi_j^\tau
  =
  \tau^{\alpha-1}\frac{a_\tau}{1-a_\tau}
  \lesssim \tau^{2\alpha-1}.
\]

For the second term on the right side of \eqref{for:two:terms}, Lemma~\ref{lem:psi-pointwise} gives
$\psi_j^\tau\le C_T j^{\alpha-1}$,
$1\le j\le \lfloor \tau T\rfloor$.
Hence,
\[
  \sum_{j=1}^{n-1}\psi_j^\tau (n-j)^{-\alpha}
  \lesssim
  \sum_{j=1}^{n-1} j^{\alpha-1}(n-j)^{-\alpha}.
\]
We split this sum at $j=\lfloor n/2\rfloor$. For $1\le j\le \lfloor n/2\rfloor$,
we have $n-j\asymp n$, and thus
\[
  \sum_{j=1}^{\lfloor n/2\rfloor} j^{\alpha-1}(n-j)^{-\alpha}
  \lesssim
  n^{-\alpha}\sum_{j=1}^{\lfloor n/2\rfloor}j^{\alpha-1}
  \lesssim
  n^{-\alpha}n^\alpha
  \lesssim 1.
\]
For $\lfloor n/2\rfloor+1\le j\le n-1$, write $r=n-j$. Then $j\asymp n$, so that
\[
  \sum_{j=\lfloor n/2\rfloor+1}^{n-1} j^{\alpha-1}(n-j)^{-\alpha}
  \lesssim
  n^{\alpha-1}\sum_{r=1}^{\lceil n/2\rceil} r^{-\alpha}
  \lesssim
  n^{\alpha-1}n^{1-\alpha}
  \lesssim 1.
\]
Therefore,
$\sum_{j=1}^{n-1}\psi_j^\tau (n-j)^{-\alpha}\lesssim 1$,
uniformly in $n\le \lfloor \tau T\rfloor$. Consequently,
\begin{equation}\label{eq:conv-muhat-bound-app-prop1}
  \sup_{1\le n\le \lfloor \tau T\rfloor}
  \sum_{s=1}^{n-1}\psi_{n-s}^\tau \hat\mu_\tau(s)
  \lesssim \tau^{2\alpha-1}.
\end{equation}

Taking expectations in \eqref{eq:lambda-decomp} and using that $\kappa_n^{\tau,+}$ has mean zero,
we deduce from \eqref{eq:muhat-uniform-bound-app-prop1} and
\eqref{eq:conv-muhat-bound-app-prop1} that
\begin{equation}\label{eq:Elambda-prop1}
  \sup_{1\le n\le \lfloor \tau T\rfloor}\E[\lambda_n^{\tau,+}]
  \lesssim \tau^{2\alpha-1}.
\end{equation}

We next bound the second moment. Since martingale increments at different times are orthogonal,
\begin{align*}
\E\!\left[\left(\kappa_n^{\tau,+}\right)^2\right]
  =
  \frac1{(1+\beta)^2}
  \sum_{s=1}^{n-1}(\psi_{n-s}^\tau)^2
  \E\!\left[
    \lambda_s^{\tau,+}+\beta^2\lambda_s^{\tau,-}
  \right]=
  \frac{1+\beta^2}{(1+\beta)^2}
  \sum_{s=1}^{n-1}(\psi_{n-s}^\tau)^2\E[\lambda_s^{\tau,+}],
\end{align*}
where we used $\lambda_s^{\tau,+}=\lambda_s^{\tau,-}$. Using \eqref{eq:Elambda-prop1} and
Lemma~\ref{lem:psi-pointwise},
\[
  \E\!\left[(\kappa_n^{\tau,+})^2\right]
  \lesssim
  \tau^{2\alpha-1}\sum_{j=1}^{n-1}(\psi_j^\tau)^2
  \lesssim
  \tau^{2\alpha-1}\sum_{j=1}^{\lfloor \tau T\rfloor}j^{2\alpha-2}
  \lesssim \tau^{4\alpha-2},
\]
since $2\alpha-2>-1$.

Returning to \eqref{eq:lambda-decomp}, the deterministic part
$\hat\mu_\tau(n)+\sum_{s=1}^{n-1}\psi_{n-s}^\tau\hat\mu_\tau(s)$
is $\mathcal O(\tau^{2\alpha-1})$ uniformly in $n\le \lfloor \tau T\rfloor$, and hence
\begin{equation}\label{eq:Elambda2-prop1}
  \sup_{1\le n\le \lfloor \tau T\rfloor}\E[(\lambda_n^{\tau,+})^2]
  \lesssim \tau^{4\alpha-2}.
\end{equation}
The same bounds hold for the ``$-$'' component.

\medskip
\noindent
\textit{Step 2: $C$-tightness of $\mathbf Y^\tau$.}
Set
$r_\tau:=\frac{1-a_\tau}{\tau^\alpha\mu}$
and let $\widetilde Y^{\tau,+}$ be the continuous, piecewise-linear
interpolation of $Y^{\tau,+}$ on the grid $\{n/\tau:n\ge0\}$:
\[
  \widetilde Y_t^{\tau,+}
  :=
  r_\tau\left(
    N_{\lfloor\tau t\rfloor}^{\tau,+}
    +(\tau t-\lfloor\tau t\rfloor)
      X_{\lfloor\tau t\rfloor+1}^{\tau,+}
  \right),\qquad 0\le t\le T.
\]
The moment bounds of Step~1, applied on the slightly larger fixed horizon
$T+1$, give
\begin{equation}\label{eq:X-second-moment-interpolation}
  \sup_{1\le i\le\lfloor\tau(T+1)\rfloor}
  \E\left[\left(X_i^{\tau,+}\right)^2\right]\lesssim\tau^{4\alpha-2}.
\end{equation}
For $0\le s<t\le T$, define
\[
  I_{s,t}^\tau
  :=
  \left\{
    \lfloor\tau s\rfloor+1,\ldots,\lceil\tau t\rceil
  \right\},
\]
and, for $i\in I_{s,t}^\tau$, set
\[
  w_i^\tau(s,t)
  :=
  \tau\,\operatorname{Leb}\!\left(
    (s,t]\cap\left(\frac{i-1}{\tau},\frac{i}{\tau}\right]
  \right),
\]
where $\operatorname{Leb}$ denotes Lebesgue measure on $\mathbb R$.
Then $w_i^\tau(s,t)\in[0,1]$ and
\[
  \sum_{i\in I_{s,t}^\tau}w_i^\tau(s,t)=\tau(t-s).
\]
Moreover,
\[
  \widetilde Y_t^{\tau,+}-\widetilde Y_s^{\tau,+}
  =
  r_\tau\sum_{i\in I_{s,t}^\tau}
  w_i^\tau(s,t)X_i^{\tau,+}.
\]
Consequently, the weighted Cauchy--Schwarz inequality and
\eqref{eq:X-second-moment-interpolation} yield
\begin{align*}
  \E\!\left[
    \left|\widetilde Y_t^{\tau,+}-\widetilde Y_s^{\tau,+}\right|^2
  \right]
  &\le
  r_\tau^2
  \left(\sum_{i\in I_{s,t}^\tau}w_i^\tau(s,t)\right)
  \sum_{i\in I_{s,t}^\tau}
  w_i^\tau(s,t)\,
  \E\left[\left(X_i^{\tau,+}\right)^2\right] \\
  &\lesssim
  \tau^{-4\alpha}\,
  \tau^2(t-s)^2\,
  \tau^{4\alpha-2}
  \lesssim |t-s|^2.
\end{align*}
Hence, the Kolmogorov criterion gives tightness of
$(\widetilde Y^{\tau,+})_{\tau\ge\tau_{\min}}$ in $C([0,T])$.

It remains to compare the interpolation with the original step process.  We
have
\[
  \sup_{0\le t\le T}
  \left|\widetilde Y_t^{\tau,+}-Y_t^{\tau,+}\right|
  \le
  r_\tau
  \max_{1\le i\le\lfloor\tau T\rfloor+1}X_i^{\tau,+}.
\]
For every $\varepsilon>0$, the union bound, Markov's inequality, and
\eqref{eq:X-second-moment-interpolation} give
\[
  \mathbb P\!\left(
    r_\tau
    \max_{1\le i\le\lfloor\tau T\rfloor+1}X_i^{\tau,+}
    >\varepsilon
  \right)
  \lesssim
  \tau\cdot\tau^{4\alpha-2}\cdot\tau^{-4\alpha}
  \lesssim\tau^{-1}\longrightarrow0.
\]
Thus $Y^{\tau,+}$ is $C$-tight. The same argument applies to the minus
component, proving that $\mathbf Y^\tau$ is $C$-tight.

\medskip
\noindent
\textit{Step 3: $C$-tightness of $\boldsymbol\Lambda^\tau$.}
By \eqref{eq:compensator-count-close}, proved in
Appendix~\ref{app:tightness},
\[
  \sup_{t\in[0,T]}\|\boldsymbol\Lambda_t^\tau-\mathbf Y_t^\tau\|_1\to0
  \qquad\text{in probability.}
\]
Since $\mathbf Y^\tau$ is $C$-tight, it follows that $\boldsymbol\Lambda^\tau$ is also
$C$-tight.

\medskip
\noindent
\textit{Step 4: $C$-tightness of $\mathbf Z^\tau$.}
We verify the $C$-tightness of $Z^{\tau,+}$ and $Z^{\tau,-}$ via a
martingale tightness criterion (e.g.\ \cite[Proposition~12]{rebolledo1980central}):
it suffices to show that their predictable quadratic variation processes are
$C$-tight and that their maximal jumps vanish in probability.

By definition,
\[
  Z_t^{\tau,+}
  =
  \sqrt{\frac{\tau^\alpha\mu}{1-a_\tau}}
  \left(Y_t^{\tau,+}-\Lambda_t^{\tau,+}\right)
  =
  \sqrt{\frac{1-a_\tau}{\tau^\alpha\mu}}\,M_{\lfloor \tau t\rfloor}^{\tau,+},
\]
and similarly
\[
  Z_t^{\tau,-}
  =
  \sqrt{\frac{1-a_\tau}{\tau^\alpha\mu}}\,M_{\lfloor \tau t\rfloor}^{\tau,-}.
\]
Hence $Z^{\tau,+}$ and $Z^{\tau,-}$ are square-integrable c\`adl\`ag martingales. Moreover,
\[
  \left\langle Z^{\tau,+},Z^{\tau,+}\right\rangle_t
  =
  \frac{1-a_\tau}{\tau^\alpha\mu}
  \sum_{s=1}^{\lfloor \tau t\rfloor}\lambda_s^{\tau,+}
  =
  \Lambda_t^{\tau,+},
\]
and similarly
\[
  \left\langle Z^{\tau,-},Z^{\tau,-}\right\rangle_t
  =
  \Lambda_t^{\tau,-}.
\]
Since the family $(\boldsymbol\Lambda^\tau)_\tau$ is $C$-tight, these two
families of predictable quadratic variation processes are $C$-tight.
It remains to show that
\[
  \sup_{t\in[0,T]}\left|\Delta Z_t^{\tau,+}\right|\xrightarrow{\mathbb P}0,
  \qquad
  \sup_{t\in[0,T]}\left|\Delta Z_t^{\tau,-}\right|\xrightarrow{\mathbb P}0.
\]
The only jump times are grid points.  More precisely, for
$t=s/\tau$, $1\le s\le\lfloor\tau T\rfloor$,
\[
  \Delta Z_{s/\tau}^{\tau,+}
  =
  \sqrt{\frac{1-a_\tau}{\tau^\alpha\mu}}
  \left(M_s^{\tau,+}-M_{s-1}^{\tau,+}\right)
  =
  \sqrt{\frac{1-a_\tau}{\tau^\alpha\mu}}
  \left(X_s^{\tau,+}-\lambda_s^{\tau,+}\right).
\]
Consequently,
\[
  \sup_{t\in[0,T]}\left|\Delta Z_t^{\tau,+}\right|
  \le
  \sqrt{\frac{1-a_\tau}{\tau^\alpha\mu}}
  \sup_{1\le s\le \lfloor \tau T\rfloor}
  \left|X_s^{\tau,+}-\lambda_s^{\tau,+}\right|.
\]
We use the fourth-moment Markov inequality. For every $\varepsilon>0$,
\begin{align*}
  \mathbb P\!\left(
    \sup_{t\in[0,T]}\left|\Delta Z_t^{\tau,+}\right|>\varepsilon
  \right)
  &\le
  \sum_{s=1}^{\lfloor \tau T\rfloor}
  \mathbb P\!\left(
    \left|X_s^{\tau,+}-\lambda_s^{\tau,+}\right|
    >
    \varepsilon\sqrt{\frac{\tau^\alpha\mu}{1-a_\tau}}
  \right) \\
  &\le
  \sum_{s=1}^{\lfloor \tau T\rfloor}
  \frac{
    \E\!\left[\left|X_s^{\tau,+}-\lambda_s^{\tau,+}\right|^4\right]
  }{
    \varepsilon^4\left(\frac{\tau^\alpha\mu}{1-a_\tau}\right)^2
  }.
\end{align*}

Now conditionally on $\mathcal F_{s-1}$,
$\left.X_s^{\tau,+}\right| \mathcal F_{s-1}\sim \mathrm{Poisson}(\lambda_s^{\tau,+})$,
and the centered fourth moment of a Poisson random variable is
\[
  \E\!\left[
    \left.\left(X_s^{\tau,+}-\lambda_s^{\tau,+}\right)^4 \right| \mathcal F_{s-1}
  \right]
  =
  \lambda_s^{\tau,+}+3(\lambda_s^{\tau,+})^2.
\]
Taking expectations and using \eqref{eq:Elambda-prop1} and \eqref{eq:Elambda2-prop1},
we obtain
\[
  \sup_{1\le s\le \lfloor \tau T\rfloor}
  \E\!\left[\left|X_s^{\tau,+}-\lambda_s^{\tau,+}\right|^4\right]
  \lesssim
  \tau^{2\alpha-1}+\tau^{4\alpha-2}
  \lesssim \tau^{4\alpha-2}.
\]
Since
$\left(\frac{\tau^\alpha\mu}{1-a_\tau}\right)^2\asymp \tau^{4\alpha}$,
it follows that as $\tau\rightarrow\infty$,
\begin{align*}
  \mathbb P\!\left(
    \sup_{t\in[0,T]}\left|\Delta Z_t^{\tau,+}\right|>\varepsilon
  \right)
  &\lesssim
  \lfloor \tau T\rfloor \cdot \tau^{4\alpha-2}\cdot \tau^{-4\alpha}
  \lesssim \tau^{-1}\to0.
\end{align*}
Thus, as $\tau\to\infty$,
\[
\sup_{t\in[0,T]}\left|\Delta Z_t^{\tau,+}\right|\to0
  \qquad\text{in probability.}
\]
The same proof applies to $Z^{\tau,-}$.
Consequently,
\[
  \sup_{t\in[0,T]}\left\|\Delta \mathbf Z_t^\tau\right\|_1
  \le
  \sup_{t\in[0,T]}\left|\Delta Z_t^{\tau,+}\right|
  +
  \sup_{t\in[0,T]}\left|\Delta Z_t^{\tau,-}\right|
  \to 0
\]
in probability.

\medskip
\noindent
\textit{Step 5: optional quadratic variations, $C$-tightness, and identification
of any limit point.}
For the ``$+$'' component,
\[
  [Z^{\tau,+},Z^{\tau,+}]_{n/\tau}
  =
  \frac{1-a_\tau}{\tau^\alpha\mu}
  \sum_{s=1}^{n}(X_s^{\tau,+}-\lambda_s^{\tau,+})^2,
\]
so that
\[
  U_n^{\tau,+}
  :=
  \left[Z^{\tau,+},Z^{\tau,+}\right]_{n/\tau}
  -
  \left\langle Z^{\tau,+},Z^{\tau,+}\right\rangle_{n/\tau}
  =
  \frac{1-a_\tau}{\tau^\alpha\mu}
  \sum_{s=1}^{n}
  \left(\left(X_s^{\tau,+}-\lambda_s^{\tau,+}\right)^2-\lambda_s^{\tau,+}\right).
\]
Then $(U_n^{\tau,+})_{n\ge0}$ is a martingale, and by orthogonality of martingale
increments,
\begin{align*}
  \E\!\left[\left(U_{\lfloor t\tau\rfloor}^{\tau,+}\right)^2\right]
  &=
  \left(\frac{1-a_\tau}{\tau^\alpha\mu}\right)^2
  \sum_{s=1}^{\lfloor t\tau\rfloor}
  \E\!\left[
    \Var\!\left(\left.
      \left(X_s^{\tau,+}-\lambda_s^{\tau,+}\right)^2\right|\mathcal F_{s-1}
    \right)
  \right].
\end{align*}
Since $\left.X_s^{\tau,+}\right|\mathcal F_{s-1}\sim\mathrm{Poisson}(\lambda_s^{\tau,+})$,
\[
  \Var\!\left(\left.(X_s^{\tau,+}-\lambda_s^{\tau,+})^2\right|\mathcal F_{s-1}\right)
  =
  2(\lambda_s^{\tau,+})^2+\lambda_s^{\tau,+}.
\]
Therefore,
\[
  \E\!\left[\left(U_{\lfloor t\tau\rfloor}^{\tau,+}\right)^2\right]
  \lesssim
  \left(\frac{1-a_\tau}{\tau^\alpha\mu}\right)^2
  \sum_{s=1}^{\lfloor t\tau\rfloor}
  \left(\E[(\lambda_s^{\tau,+})^2]+\E[\lambda_s^{\tau,+}]\right).
\]
Using \eqref{eq:Elambda-prop1} and \eqref{eq:Elambda2-prop1}, we get
$\E\!\left[\left(U_{\lfloor t\tau\rfloor}^{\tau,+}\right)^2\right]
  \lesssim \tau^{-1}$,
uniformly for $t\in[0,T]$. Hence, by Doob's inequality, as $\tau\to\infty$,
\[
  \sup_{t\in[0,T]}
  \left|
    [Z^{\tau,+},Z^{\tau,+}]_t-\Lambda_t^{\tau,+}
  \right|
  \to0
  \qquad\text{in probability.}
\]
The same argument applies to the ``$-$'' component. Since
$(\boldsymbol\Lambda^\tau)_\tau$ is $C$-tight, the preceding estimates show
that the optional quadratic variations of $Z^{\tau,+}$ and $Z^{\tau,-}$ are
$C$-tight. Together with the vanishing-jump estimate from Step~4,
\cite[Proposition~12]{rebolledo1980central} therefore implies that
$(Z^{\tau,+})_{\tau\ge\tau_{\min}}$ and
$(Z^{\tau,-})_{\tau\ge\tau_{\min}}$ are $C$-tight. Hence
$(\mathbf Z^\tau)_{\tau\ge\tau_{\min}}$ is $C$-tight on $[0,T]$.

Combining this with \eqref{eq:compensator-count-close}, as $\tau\to\infty$,
\begin{equation}\label{eq:optional-bracket-count-close}
  \sup_{t\in[0,T]}
  \left|
    [Z^{\tau,+},Z^{\tau,+}]_t-Y_t^{\tau,+}
  \right|
  \to0
  \qquad\text{in probability.}
\end{equation}
The estimates above and \eqref{eq:Elambda-prop1} also give
\begin{equation}\label{eq:uniform-bracket-bound}
  C_T:=
  \sup_{\tau\ge\tau_{\min}}
  \E\!\left[
    \left\langle Z^{\tau,+},Z^{\tau,+}\right\rangle_T
    +
    \left\langle Z^{\tau,-},Z^{\tau,-}\right\rangle_T
  \right]
  <\infty.
\end{equation}
To verify the predictably uniform tightness (P-UT) property, let
$K^\tau=(K^{\tau,+},K^{\tau,-})$ be an elementary predictable process with
$\lVert K^\tau\rVert\le1$, and set
$M^{\tau,K}:=K^\tau\mathbin{\boldsymbol{\cdot}}\mathbf Z^\tau$.
The conditional independence of the two Poisson innovations gives
$\langle Z^{\tau,+},Z^{\tau,-}\rangle=0$. Hence,
\[
  \left\langle M^{\tau,K}\right\rangle_T
  \le
  \left\langle Z^{\tau,+},Z^{\tau,+}\right\rangle_T
  +
  \left\langle Z^{\tau,-},Z^{\tau,-}\right\rangle_T.
\]
Thus \eqref{eq:uniform-bracket-bound} implies
\[
  \sup_{\tau\ge\tau_{\min},\,K^\tau}
  \E\!\left[\left\langle M^{\tau,K}\right\rangle_T\right]
  \le C_T.
\]
Applying Lenglart's inequality \cite[Lemma~I.3.30]{jacod2013limit} to
$|M^{\tau,K}|^2$, which is $L$-dominated by
$\langle M^{\tau,K}\rangle$, and then using Markov's inequality, we obtain,
for all $K,\eta>0$,
\[
  \mathbb P\!\left(\sup_{t\le T}\left|M_t^{\tau,K}\right|>K\right)
  \le
  \frac{\eta}{K^2}
  +
  \mathbb P\!\left(\left\langle M^{\tau,K}\right\rangle_T\ge\eta\right)
  \le
  \frac{\eta}{K^2}+\frac{C_T}{\eta}.
\]
Choosing $\eta=K$ and then letting $K\to\infty$ proves that
$\mathbf Z^\tau$ satisfies P-UT.

For each $i\in\{+,-\}$, the martingale $Z^{\tau,i}$ is square-integrable
and starts at zero. The martingale isometry and
\eqref{eq:uniform-bracket-bound} therefore yield
\begin{equation}\label{eq:Z-uniform-L2}
  \sup_{\tau\ge\tau_{\min}}\sup_{t\in[0,T]}
  \E\!\left[|Z_t^{\tau,i}|^2\right]
  =
  \sup_{\tau\ge\tau_{\min}}\sup_{t\in[0,T]}
  \E\!\left[
    \left\langle Z^{\tau,i},Z^{\tau,i}\right\rangle_t
  \right]
  \le C_T.
\end{equation}
Consequently, for every $R>0$, \eqref{eq:Z-uniform-L2} gives
\[
  \sup_{\tau\ge\tau_{\min}}\sup_{t\in[0,T]}
  \E\!\left[
    |Z_t^{\tau,i}|\mathbf 1_{\{|Z_t^{\tau,i}|>R\}}
  \right]
  \le \frac{C_T}{R}.
\]
Thus the family
$\{Z_t^{\tau,i}:\tau\ge\tau_{\min},\ t\in[0,T]\}$ is uniformly integrable.

Let $(\mathbf Y,\mathbf Z)$ be a limit point of
$(\mathbf Y^\tau,\mathbf Z^\tau)$ along a convergent subsequence, still denoted
by $\tau$. By the $C$-tightness already established, this limit has continuous
paths. The uniform integrability above and
\cite[Proposition~IX.1.12]{jacod2013limit}, applied componentwise, imply
that $\mathbf Z$ is a martingale with respect to the filtration generated by
$(\mathbf Y,\mathbf Z)$. Moreover, convergence of the one-dimensional
marginals, the Portmanteau theorem, and \eqref{eq:Z-uniform-L2} give, for
each $t\in[0,T]$ and $i\in\{+,-\}$,
\[
  \E[|Z_t^i|^2]
  \le
  \liminf_{\tau\to\infty}\E[|Z_t^{\tau,i}|^2]
  \le C_T.
\]
Hence $\mathbf Z$ is a continuous square-integrable martingale on $[0,T]$.

Since $\mathbf Z^\tau$ satisfies P-UT,
\cite[Theorem~VI.6.26]{jacod2013limit} gives joint convergence of
$\mathbf Z^\tau$ and its optional quadratic covariations to $\mathbf Z$ and
its optional quadratic covariations. Combining the joint bracket convergence with
\eqref{eq:optional-bracket-count-close} and its
``$-$'' counterpart yields
\[
  [Z^+,Z^+]_t=Y_t^+.
\]
The same argument gives
\[
  [Z^-,Z^-]_t=Y_t^-.
\]

For the cross-variation,
\[
\left[Z^{\tau,+},Z^{\tau,-}\right]_{n/\tau}
  =
  \frac{1-a_\tau}{\tau^\alpha\mu}
  \sum_{s=1}^{n}
  \left(X_s^{\tau,+}-\lambda_s^{\tau,+}\right)\left(X_s^{\tau,-}-\lambda_s^{\tau,-}\right).
\]
Because $X_s^{\tau,+}$ and $X_s^{\tau,-}$ are conditionally independent given
$\mathcal F_{s-1}$, the process
$V_n^\tau:=[Z^{\tau,+},Z^{\tau,-}]_{n/\tau}$
is a martingale. Moreover,
\begin{align*}
  \E\!\left[\left(V_{\lfloor t\tau\rfloor}^\tau\right)^2\right]
  &=
  \left(\frac{1-a_\tau}{\tau^\alpha\mu}\right)^2
  \sum_{s=1}^{\lfloor t\tau\rfloor}
  \E\!\left[
    \lambda_s^{\tau,+}\lambda_s^{\tau,-}
  \right].
\end{align*}
Since $\lambda_s^{\tau,+}=\lambda_s^{\tau,-}$ and \eqref{eq:Elambda2-prop1} holds,
$\E\!\left[\left(V_{\lfloor t\tau\rfloor}^\tau\right)^2\right]
  \lesssim \tau^{-1}$.
Hence, by Doob's inequality, as $\tau\to\infty$,
\[
  \sup_{t\in[0,T]}|[Z^{\tau,+},Z^{\tau,-}]_t|
  \to0
  \qquad\text{in probability,}
\]
and therefore
$[Z^+,Z^-]_t=0$.
Thus, every limit point $\mathbf Z$ is a continuous martingale with
$[\mathbf Z,\mathbf Z]=\mathrm{diag}(\mathbf Y)$.
This completes the proof.

\subsection{Compensator--count approximation}\label{app:tightness}
We prove \eqref{eq:compensator-count-close}.  Fix $T>0$. It suffices to prove
\[
  \sup_{t \in [0,T]} \left|\Lambda_t^{\tau,+} - Y_t^{\tau,+}\right| \to 0
  \qquad\text{in probability as }\tau\to\infty,
\]
since the proof for the ``$-$'' component is identical.

Since
$Y_t^{\tau,+}-\Lambda_t^{\tau,+}
  =\frac{1-a_\tau}{\tau^\alpha\mu}\,M_{\lfloor t\tau\rfloor}^{\tau,+}$,
Doob's $L^2$ inequality yields
\begin{align*}
  \E\left[\sup_{t\in[0,T]}|Y_t^{\tau,+}-\Lambda_t^{\tau,+}|^2\right]
  &\le \left(\frac{1-a_\tau}{\tau^\alpha\mu}\right)^2
  \E\left[\sup_{0\le n\le \lfloor\tau T\rfloor}|M_n^{\tau,+}|^2\right]\\
  &\le 4\left(\frac{1-a_\tau}{\tau^\alpha\mu}\right)^2
  \E\left[|M_{\lfloor\tau T\rfloor}^{\tau,+}|^2\right].
\end{align*}
Because $M^{\tau,+}$ is a square-integrable martingale,
\[
  \E\left[\left|M_{\lfloor\tau T\rfloor}^{\tau,+}\right|^2\right]
  =\E\left[\left\langle M^{\tau,+}\right\rangle_{\lfloor\tau T\rfloor}\right]
  =\sum_{s=1}^{\lfloor\tau T\rfloor}\E[\lambda_s^{\tau,+}].
\]
Using the uniform estimate \eqref{eq:Elambda-prop1}, we have
$\sup_{1\le s\le \lfloor\tau T\rfloor}\E[\lambda_s^{\tau,+}]
  \lesssim \tau^{2\alpha-1}$,
and therefore
$\sum_{s=1}^{\lfloor\tau T\rfloor}\E[\lambda_s^{\tau,+}]
  \lesssim \tau^{2\alpha}$.
Hence, as $\tau\to\infty$,
\[
  \E\left[\sup_{t\in[0,T]}\left|Y_t^{\tau,+}-\Lambda_t^{\tau,+}\right|^2\right]
  \lesssim
  \left(\frac{1-a_\tau}{\tau^\alpha\mu}\right)^2\tau^{2\alpha}
  \lesssim \tau^{-2\alpha}\to0.
\]
Thus, as $\tau\to\infty$,
\[
  \sup_{t\in[0,T]}\left|Y_t^{\tau,+}-\Lambda_t^{\tau,+}\right|
  \to 0
  \qquad\text{in }L^2,
\]
and hence also in probability. This completes the proof.

\subsection{Identification of limit points}\label{app:prop2}

Let
$\left(Y^{\tau_k,+},Y^{\tau_k,-},Z^{\tau_k,+},Z^{\tau_k,-}\right)$
be a subsequence converging in law on $D([0,T],\mathbb R^4)$ to a limit point
$(Y^+,Y^-,Z^+,Z^-)$, where $T>0$ is fixed.
By Proposition~\ref{proposition1}, every limit point is continuous. Hence, by the
Skorokhod representation theorem, after possibly enlarging the probability space
and still denoting the subsequence by $\tau_k$, we may assume that
\begin{equation}\label{eq:prop2-skorokhod}
\sup_{t\in[0,T]}
\left(
\left|Y_t^{\tau_k,+}-Y_t^+\right|
+
\left|Y_t^{\tau_k,-}-Y_t^-\right|
+
\left|Z_t^{\tau_k,+}-Z_t^+\right|
+
\left|Z_t^{\tau_k,-}-Z_t^-\right|
\right)\to0
\qquad\text{a.s.}
\end{equation}
as $k\to\infty$.
On this represented space, define
\[
  \boldsymbol\Lambda^{\tau_k}
  :=
  \mathbf Y^{\tau_k}
  -
  \sqrt{\frac{1-a_{\tau_k}}{\tau_k^\alpha\mu}}\,
  \mathbf Z^{\tau_k}.
\]
This is the defining relation between
$(\boldsymbol\Lambda^{\tau_k},\mathbf Y^{\tau_k},\mathbf Z^{\tau_k})$ in the
original model, so the represented triple has the same law as the original
one and all exact identities used below remain valid. Moreover,
\begin{equation}\label{eq:prop2-coupled-compensator-close}
  \sup_{t\in[0,T]}
  \left\lVert
    \boldsymbol\Lambda_t^{\tau_k}-\mathbf Y_t^{\tau_k}
  \right\rVert_1
  \le
  \sqrt{\frac{1-a_{\tau_k}}{\tau_k^\alpha\mu}}
  \sup_{t\in[0,T]}\left\lVert\mathbf Z_t^{\tau_k}\right\rVert_1
  \longrightarrow0
  \qquad\text{a.s.}
\end{equation}
Indeed, the supremum on the right is almost surely bounded along the
convergent sequence, while the deterministic prefactor equals
$\sqrt{\gamma/\mu}\,\tau_k^{-\alpha}$.

\medskip
\noindent
\textit{Step 1: the two limit components of $\mathbf Y^\tau$ coincide.}

Recall that
\[
\Lambda_t^{\tau_k,+}=\Lambda_t^{\tau_k,-},
\qquad t\in[0,T].
\]
Hence, by \eqref{eq:prop2-coupled-compensator-close}, as $k\to\infty$,
\begin{align*}
\sup_{t\in[0,T]}\left|Y_t^{\tau_k,+}-Y_t^{\tau_k,-}\right|
  \le
  \sup_{t\in[0,T]}\left|Y_t^{\tau_k,+}-\Lambda_t^{\tau_k,+}\right|
  +
  \sup_{t\in[0,T]}\left|Y_t^{\tau_k,-}-\Lambda_t^{\tau_k,-}\right|
  \to0
  \quad\text{a.s.}
\end{align*}
Passing to the almost sure Skorokhod realization along the chosen subsequence,
we obtain
$Y_t^+=Y_t^-$
for all $t\in[0,T]$, a.s.
We therefore write
\[
  Y_t:=Y_t^+=Y_t^-,
  \qquad t\in[0,T].
\]

\medskip
\noindent
\textit{Step 2: an explicit decomposition for $\Lambda^{\tau,+}$.}

We first show that for every integer $n\ge1$,
\begin{equation}\label{eq:prop2-main-identity}
  \hat\mu_\tau(n)+\sum_{s=1}^{n-1}\psi_{n-s}^\tau\hat\mu_\tau(s)
  =
  \mu_\tau+\xi_0\mu_\tau\frac1{1-a_\tau}
  +\mu_\tau(1-\xi_0)\sum_{s=1}^{n-1}\psi_{n-s}^\tau.
\end{equation}
For a transparent verification, set
$C_n^\tau:=\sum_{r=1}^{n-1}\varphi_r^\tau$,
$D_n^\tau:=\sum_{r=1}^{n-1}\psi_{n-r}^\tau$.
The renewal identity
$\psi^\tau=\varphi^\tau+\psi^\tau*\varphi^\tau$ implies
\begin{equation}\label{eq:cumulative-renewal-prop2}
  C_n^\tau
  +
  \sum_{s=1}^{n-1}\psi_{n-s}^\tau C_s^\tau
  =D_n^\tau.
\end{equation}
Indeed, if $u_n=1$ for $n\ge1$ and $u_0=0$, then
$C^\tau=\varphi^\tau*u$, $D^\tau=\psi^\tau*u$, and
\eqref{eq:cumulative-renewal-prop2} follows by convolving the renewal identity
with $u$. Definition~\ref{def:params} can now be written as
\[
  \hat\mu_\tau(n)
  =
  \mu_\tau\left(1+\frac{\xi_0}{1-a_\tau}\right)
  -
  \xi_0\mu_\tau\left(1+\frac1{1-a_\tau}\right)C_n^\tau.
\]
Therefore, using \eqref{eq:cumulative-renewal-prop2},
\begin{align*}
  \hat\mu_\tau(n)
  +\sum_{s=1}^{n-1}\psi_{n-s}^\tau\hat\mu_\tau(s)
  &=
  \mu_\tau\left(1+\frac{\xi_0}{1-a_\tau}\right)(1+D_n^\tau)
  -
  \xi_0\mu_\tau\left(1+\frac1{1-a_\tau}\right)D_n^\tau
  \\
  &=
  \mu_\tau+\frac{\xi_0\mu_\tau}{1-a_\tau}
  +\mu_\tau(1-\xi_0)D_n^\tau,
\end{align*}
which proves \eqref{eq:prop2-main-identity}.
Now recall from \eqref{eq:lambda-decomp} that
\[
  \lambda_n^{\tau,+}
  =
  \hat\mu_\tau(n)
  +
  \sum_{s=1}^{n-1}\psi_{n-s}^\tau\hat\mu_\tau(s)
  +
  \frac1{1+\beta}\sum_{s=1}^{n-1}\psi_{n-s}^\tau
  \left(
    \Delta M_s^{\tau,+}
    +
    \beta \Delta M_s^{\tau,-}
  \right),
\]
where
\[
  \Delta M_s^{\tau,\pm}:=M_s^{\tau,\pm}-M_{s-1}^{\tau,\pm}.
\]
Using \eqref{eq:prop2-main-identity}, we obtain
\begin{align}
  \lambda_n^{\tau,+}
  &=
  \mu_\tau
  +
  \mu_\tau\sum_{s=1}^{n-1}\psi_{n-s}^\tau
  +
  \xi_0\mu_\tau
  \left(
    \frac1{1-a_\tau}
    -
    \sum_{s=1}^{n-1}\psi_{n-s}^\tau
  \right)
  \nonumber\\
  &\quad
  +
  \frac1{1+\beta}\sum_{s=1}^{n-1}\psi_{n-s}^\tau
  \left(
    \Delta M_s^{\tau,+}
    +
    \beta \Delta M_s^{\tau,-}
  \right).
  \label{eq:lambda-explicit-prop2}
\end{align}

We now sum \eqref{eq:lambda-explicit-prop2} over $n=1,\dots,m$.
For the deterministic term, Tonelli's theorem gives
\[
  \sum_{n=1}^{m}\sum_{s=1}^{n-1}\psi_{n-s}^\tau
  =
  \sum_{s=1}^{m-1}\sum_{j=1}^{m-s}\psi_j^\tau
  =
  \sum_{j=1}^{m-1}(m-j)\psi_j^\tau
  =
  \sum_{s=1}^{m-1}s\psi_{m-s}^\tau.
\]
For the martingale term, define
\[
  \mathcal H_s^\tau:=M_s^{\tau,+}+\beta M_s^{\tau,-},
  \qquad
  \Delta\mathcal H_s^\tau:=\mathcal H_s^\tau-\mathcal H_{s-1}^\tau.
\]
Then
\begin{align*}
  \sum_{n=1}^{m}\sum_{s=1}^{n-1}\psi_{n-s}^\tau \Delta\mathcal H_s^\tau
  &=
  \sum_{s=1}^{m-1}\left(\sum_{j=1}^{m-s}\psi_j^\tau\right)\Delta\mathcal H_s^\tau.
\end{align*}
Set
\[
  A_s^\tau:=\sum_{j=1}^{m-s}\psi_j^\tau,
  \qquad 1\le s\le m-1.
\]
By discrete summation by parts,
\[
  \sum_{s=1}^{m-1}A_s^\tau \Delta\mathcal H_s^\tau
  =
  A_{m-1}^\tau \mathcal H_{m-1}^\tau
  +
  \sum_{s=1}^{m-2}(A_s^\tau-A_{s+1}^\tau)\mathcal H_s^\tau.
\]
Since
\[
  A_s^\tau-A_{s+1}^\tau=\psi_{m-s}^\tau,
  \qquad
  A_{m-1}^\tau=\psi_1^\tau,
\]
this yields
\[
  \sum_{n=1}^{m}\sum_{s=1}^{n-1}\psi_{n-s}^\tau \Delta\mathcal H_s^\tau
  =
  \sum_{s=1}^{m-1}\psi_{m-s}^\tau \mathcal H_s^\tau.
\]
Consequently,
\begin{align}
  \sum_{n=1}^{m}\lambda_n^{\tau,+}
  &=
  \mu_\tau m
  +
  \mu_\tau\sum_{s=1}^{m-1}s\psi_{m-s}^\tau
  +
  \xi_0\mu_\tau
  \left(
    \frac{m}{1-a_\tau}
    -
    \sum_{s=1}^{m-1}s\psi_{m-s}^\tau
  \right)
  \nonumber\\
  &\quad
  +
  \frac1{1+\beta}
  \sum_{s=1}^{m-1}\psi_{m-s}^\tau
  \left(
    M_s^{\tau,+}+\beta M_s^{\tau,-}
  \right).
  \label{eq:sum-lambda-prop2}
\end{align}

Taking $m=\lfloor \tau t\rfloor$ and multiplying by
$\frac{1-a_\tau}{\tau^\alpha\mu}$, we obtain
\begin{equation}\label{eq:Lambda-decomp-prop2}
  \Lambda_t^{\tau,+}=\Lambda_t^{\tau,-}=T_1^\tau(t)+T_2^\tau(t)+T_3^\tau(t),
\end{equation}
where
\begin{align*}
  T_1^\tau(t)
  &:=
  (1-a_\tau)\frac{\lfloor \tau t\rfloor}{\tau},
  \\
  T_2^\tau(t)
  &:=
  \frac{1-a_\tau}{\tau}
  \sum_{s=1}^{\lfloor \tau t\rfloor-1}s\psi_{\lfloor \tau t\rfloor-s}^\tau
  +
  \xi_0
  \left(
    \frac{\lfloor \tau t\rfloor}{\tau}
    -
    \frac{1-a_\tau}{\tau}
    \sum_{s=1}^{\lfloor \tau t\rfloor-1}s\psi_{\lfloor \tau t\rfloor-s}^\tau
  \right),
  \\
  T_3^\tau(t)
  &:=
  \frac{1-a_\tau}{\tau^\alpha\mu(1+\beta)}
  \sum_{s=1}^{\lfloor \tau t\rfloor-1}
  \psi_{\lfloor \tau t\rfloor-s}^\tau
  \left(
    M_s^{\tau,+}+\beta M_s^{\tau,-}
  \right).
\end{align*}
Using
\[
  Z_{s/\tau}^{\tau,\pm}
  =
  \sqrt{\frac{1-a_\tau}{\tau^\alpha\mu}}\,M_s^{\tau,\pm},
\]
we can rewrite $T_3^\tau$ as
\begin{equation}\label{eq:T3-rewrite-prop2}
  T_3^\tau(t)
  =
  \frac{\tau(1-a_\tau)}{\sqrt{\gamma\mu(1+\beta)^2}}
  \sum_{s=1}^{\lfloor \tau t\rfloor-1}
  \frac1{\tau}\psi_{\lfloor \tau t\rfloor-s}^\tau
  \left(
    Z_{s/\tau}^{\tau,+}+\beta Z_{s/\tau}^{\tau,-}
  \right).
\end{equation}

\medskip
\noindent
\textit{Step 3: passage to the limit in the decomposition.}

For notational simplicity, throughout this step we write $\tau$ for
$\tau_k$; all limits are taken as $k\to\infty$, equivalently as
$\tau_k\to\infty$.

Recall the rescaled renewal density
\[
  \mathsf S^\tau(u):=
  \frac{1-a_\tau}{a_\tau}\tau\psi_{\lfloor \tau u\rfloor}^\tau,
  \qquad u\ge0,
\]
and Lemma~\ref{lem:kernel_convergence}, which states that
\[
  \mathsf S^\tau(u)\,du \Rightarrow f_{\alpha,\gamma}(u)\,du
\]
weakly on $[0,\infty)$, where
$f_{\alpha,\gamma}(u):=\gamma u^{\alpha-1}E_{\alpha,\alpha}(-\gamma u^\alpha)$
and 
$F_{\alpha,\gamma}(t):=\int_0^t f_{\alpha,\gamma}(u)\,du$.

Fix $t\in[0,T]$. We identify the limits of the three terms in
\eqref{eq:Lambda-decomp-prop2}.

First, as $k\to\infty$,
\[
  T_1^\tau(t)=(1-a_\tau)\frac{\lfloor \tau t\rfloor}{\tau}\to0,
\]
since $1-a_\tau=\gamma\tau^{-\alpha}\to0$.

Next, set $t_\tau:=\lfloor\tau t\rfloor/\tau$. Writing
$j=\lfloor \tau t\rfloor-s$, we have
\begin{align*}
  \frac{1-a_\tau}{\tau}
  \sum_{s=1}^{\lfloor \tau t\rfloor-1}s\psi_{\lfloor \tau t\rfloor-s}^\tau
  =
  \frac{1-a_\tau}{\tau}
  \sum_{j=1}^{\lfloor \tau t\rfloor-1}
  (\lfloor \tau t\rfloor-j)\psi_j^\tau
 =
  \frac{a_\tau}{\tau}
  \sum_{j=1}^{\lfloor \tau t\rfloor-1}
  \mathsf S^\tau(j/\tau)
  \left(\frac{\lfloor \tau t\rfloor}{\tau}-\frac{j}{\tau}\right).
\end{align*}
The step-function definition of $\mathsf S^\tau$ and
$\int_0^\infty\mathsf S^\tau(u)\,du=1$ show that the last Riemann sum
differs from
\[
  a_\tau\int_0^\infty
  \mathsf S^\tau(u)(t-u)^+\,du
\]
by at most $C(\tau^{-1}+|t_\tau-t|)$. Since
$u\mapsto(t-u)^+$ is bounded and continuous on $[0,\infty)$,
$a_\tau\to1$, and $\mathsf S^\tau(u)\,du$ converges weakly, it follows that,
as $k\to\infty$,
\[
  \frac{1-a_\tau}{\tau}
  \sum_{s=1}^{\lfloor \tau t\rfloor-1}s\psi_{\lfloor \tau t\rfloor-s}^\tau
  \to
  \int_0^t (t-u)f_{\alpha,\gamma}(u)\,du
  =
  \int_0^t F_{\alpha,\gamma}(t-s)\,ds.
\]
Also,
$\frac{\lfloor \tau t\rfloor}{\tau}\to t$.
Hence,
\begin{equation}\label{eq:T2-limit-prop2}
  T_2^\tau(t)
  \to
  \int_0^t F_{\alpha,\gamma}(t-s)\,ds
  +
  \xi_0\left(
    t-\int_0^t F_{\alpha,\gamma}(t-s)\,ds
  \right).
\end{equation}

We now treat $T_3^\tau(t)$. Using \eqref{eq:T3-rewrite-prop2}, write
\begin{align*}
  T_3^\tau(t)
  &=
  \frac{a_\tau}{\sqrt{\gamma\mu(1+\beta)^2}}
  \sum_{s=1}^{\lfloor \tau t\rfloor-1}
  \frac1{\tau}
  \mathsf S^\tau\!\left(\frac{\lfloor \tau t\rfloor-s}{\tau}\right)
  \left(
    Z_{s/\tau}^{\tau,+}+\beta Z_{s/\tau}^{\tau,-}
  \right).
\end{align*}
Set
\[
  \overline Z^\tau(s):=Z_s^{\tau,+}+\beta Z_s^{\tau,-},
  \qquad
  \overline Z(s):=Z_s^++\beta Z_s^-.
\]
Then
\begin{align*}
  &T_3^\tau(t)
  -
  \frac1{\sqrt{\gamma\mu(1+\beta)^2}}
  \int_0^t f_{\alpha,\gamma}(t-s)\overline Z(s)\,ds
  \\
  &=
  \frac{a_\tau}{\sqrt{\gamma\mu(1+\beta)^2}}
  \sum_{s=1}^{\lfloor \tau t\rfloor-1}
  \frac1{\tau}
  \mathsf S^\tau\!\left(\frac{\lfloor \tau t\rfloor-s}{\tau}\right)
  \left(\overline Z^\tau(s/\tau)-\overline Z(s/\tau)\right)
  \\
  &\quad
  +
  \frac{1}{\sqrt{\gamma\mu(1+\beta)^2}}
  \left[
    a_\tau
    \sum_{s=1}^{\lfloor \tau t\rfloor-1}
    \frac1{\tau}
    \mathsf S^\tau\!\left(\frac{\lfloor \tau t\rfloor-s}{\tau}\right)
    \overline Z(s/\tau)
    -
    \int_0^t f_{\alpha,\gamma}(t-s)\overline Z(s)\,ds
  \right]
  \\
  &=:I_1^\tau(t)+I_2^\tau(t).
\end{align*}
For the first term, since $a_\tau\le1$ and
\[
  \frac1{\tau}\sum_{s=1}^{\lfloor \tau t\rfloor-1}
  \mathsf S^\tau\!\left(\frac{\lfloor \tau t\rfloor-s}{\tau}\right)
  \le
  \int_0^\infty \mathsf S^\tau(u)\,du
  =1,
\]
we get, as $k\to\infty$,
\[
  |I_1^\tau(t)|
  \le
  \frac{1}{\sqrt{\gamma\mu(1+\beta)^2}}
  \sup_{0\le r\le T}|\overline Z^\tau(r)-\overline Z(r)|
  \to0
\]
almost surely by \eqref{eq:prop2-skorokhod}.

For the second term, note first that $\overline Z(0)=0$. Hence the zero extension
\[
  \widetilde Z_t(u):=\overline Z(t-u)\mathbf 1_{[0,t]}(u),
  \qquad u\ge0,
\]
is bounded and continuous. After the change of variables
$j=\lfloor\tau t\rfloor-s$, the discrete expression in $I_2^\tau(t)$ is
a Riemann sum against $\mathsf S^\tau(u)\,du$. Its difference from
$a_\tau\int_0^\infty\mathsf S^\tau(u)\widetilde Z_t(u)\,du$ is bounded
by the modulus of continuity of $\widetilde Z_t$ at
$\tau^{-1}+|t_\tau-t|$, because the measures have unit mass. This bound
tends to zero. Weak convergence and $a_\tau\to1$ therefore imply that
$I_2^\tau(t)\to0$.
Therefore, as $k\to\infty$,
\begin{equation}\label{eq:T3-limit-prop2}
  T_3^\tau(t)
  \to
  \frac1{\sqrt{\gamma\mu(1+\beta)^2}}
  \int_0^t
  f_{\alpha,\gamma}(t-s)\left(Z_s^++\beta Z_s^-\right)\,ds.
\end{equation}

Combining \eqref{eq:Lambda-decomp-prop2}, \eqref{eq:T2-limit-prop2},
\eqref{eq:T3-limit-prop2}, and the almost sure approximation
\eqref{eq:prop2-coupled-compensator-close},
we obtain the following identity first for every rational
$t\in[0,T]$ on a common probability-one event. Both sides are continuous
in $t$: for the convolution term this follows from
$f_{\alpha,\gamma}\in L^1([0,T])$ and the continuity of $Z^+$ and $Z^-$.
The identity therefore extends to every $t\in[0,T]$:
\begin{equation}\label{eq:Y-limit-eqn}
\begin{aligned}
  Y_t
  &=
  \int_0^t F_{\alpha,\gamma}(t-s)\,ds
  +
  \xi_0\left(
    t-\int_0^t F_{\alpha,\gamma}(t-s)\,ds
  \right)
  \\
  &\qquad
  +
  \frac1{\sqrt{\gamma\mu(1+\beta)^2}}
  \int_0^t
  f_{\alpha,\gamma}(t-s)\left(Z_s^++\beta Z_s^-\right)\,ds.
\end{aligned}
\end{equation}

From this point onward, we equip $(\mathbf Y,\mathbf Z)$ with its completed,
right-continuous natural filtration. Since the represented limit has the same
law as the original limit point, Proposition~\ref{proposition1} ensures that
$\mathbf Z$ is a continuous martingale with respect to this filtration and has
the quadratic covariations stated there.

\medskip
\noindent
\textit{Step 4: $Y$ is absolutely continuous and its density.}

Define
\[
  c_{\beta,\gamma,\mu}:=\frac1{\sqrt{\gamma\mu(1+\beta)^2}},
  \qquad
  \mathcal H_t:=Z_t^++\beta Z_t^-.
\]
By Proposition~\ref{proposition1},
\[
  [\mathcal H]_t=[Z^++\beta Z^-]_t=(1+\beta^2)Y_t.
\]
In particular, $Y$ is nondecreasing. Since $Y$ is also continuous and
adapted, each hitting time used below is a stopping time.
Then \eqref{eq:Y-limit-eqn} becomes
\[
  Y_t
  =
  \int_0^t F_{\alpha,\gamma}(t-s)\,ds
  +
  \xi_0\left(
    t-\int_0^t F_{\alpha,\gamma}(t-s)\,ds
  \right)
  +
  c_{\beta,\gamma,\mu}\int_0^t f_{\alpha,\gamma}(t-s)\mathcal H_s\,ds.
\]
Since $\mathcal H$ is a continuous martingale, $\mathcal H_0=0$, and
$F_{\alpha,\gamma}(0)=0$, integration by parts gives, for every $t\le T$,
\[
  \int_0^t f_{\alpha,\gamma}(t-s)\mathcal H_s\,ds
  =
  \int_0^t F_{\alpha,\gamma}(t-s)\,d\mathcal H_s.
\]
We apply the stochastic Fubini theorem
\cite[Theorem~IV.65]{protter2005stochastic} after localization. For $m\ge1$,
set
\[
  \sigma_m:=\inf\{t\in[0,T]:Y_t\ge m\}\wedge T,
  \qquad
  \mathcal H_t^{(m)}:=\mathcal H_{t\wedge\sigma_m}.
\]
Since $f_{\alpha,\gamma}\in L^2([0,T])$ because $\alpha>1/2$, we have
\begin{align*}
  &\E\left[
    \int_0^T\int_0^t
      f_{\alpha,\gamma}(t-s)^2\,d[\mathcal H^{(m)}]_s\,dt
  \right] \\
  &\qquad\le
  \lVert f_{\alpha,\gamma}\rVert_{L^2([0,T])}^2
  \E\left[[\mathcal H^{(m)}]_T\right]
  \le
  (1+\beta^2)m
  \lVert f_{\alpha,\gamma}\rVert_{L^2([0,T])}^2<\infty.
\end{align*}
The preceding estimate verifies the integrability condition in that theorem.
Hence, for each rational $t\in[0,T]$, stochastic Fubini theorem gives the following
identity. Taking a countable intersection, it holds for all such $t$ on a
common probability-one event:
\[
  \int_0^t F_{\alpha,\gamma}(t-s)\,d\mathcal H_s^{(m)}
  =
  \int_0^t J_r^{(m)}\,dr,
  \qquad
  J_r^{(m)}:=\int_0^r
    f_{\alpha,\gamma}(r-s)\,d\mathcal H_s^{(m)}.
\]
The left-hand side equals the ordinary convolution
$\int_0^t f_{\alpha,\gamma}(t-s)\mathcal H_s^{(m)}\,ds$ by integration by parts and
is continuous in $t$ because $f_{\alpha,\gamma}\in L^1([0,T])$ and
$\mathcal H^{(m)}$ is continuous.  The right-hand side is absolutely continuous in
$t$.  The equality therefore extends simultaneously to every $t\in[0,T]$.
The processes $J^{(m)}$ are consistent before the stopping times
$\sigma_m$.  Patching them and using that, for every sample path,
$\sigma_m=T$ for all sufficiently large $m$, yields a jointly measurable
process $J$ such that, on a common probability-one event,
\[
  \int_0^t F_{\alpha,\gamma}(t-s)\,d\mathcal H_s
  =
  \int_0^t J_s\,ds,
  \qquad 0\le t\le T.
\]
Let
\[
  b(t):=\xi_0\left(1-F_{\alpha,\gamma}(t)\right)
        +F_{\alpha,\gamma}(t).
\]
Equation~\eqref{eq:Y-limit-eqn} now implies
\[
  Y_t=\int_0^t
  \left(b(s)+c_{\beta,\gamma,\mu}J_s\right)\,ds,
  \qquad 0\le t\le T.
\]
In particular, $Y$ is absolutely continuous. Since $Y$ is continuous,
adapted, and nondecreasing, its predictable random measure $dY$ admits a
nonnegative predictable Radon--Nikodym density.  Denote one such density by
$v^0$:
\begin{equation}\label{eq:v-rn-density}
  dY_t=v_t^0\,dt,
  \qquad
  v_t^0\ge0
  \quad dt\otimes d\mathbb P\text{-a.e.}
\end{equation}
By uniqueness of the Radon--Nikodym density,
\begin{equation}\label{eq:v-density-ae}
  v_t^0
  =
  b(t)+c_{\beta,\gamma,\mu}
  \int_0^t f_{\alpha,\gamma}(t-s)\,d\mathcal H_s,
  \qquad dt\otimes d\mathbb P\text{-a.e.}
\end{equation}

\medskip
\noindent
\textit{Step 5: Brownian representation of $Z^+$ and $Z^-$.}

By Proposition~\ref{proposition1} and \eqref{eq:v-rn-density},
$d[\mathbf Z,\mathbf Z]_t=v_t^0 I_2\,dt$.
After a possible extension of the filtered probability space, the vector
representation theorem provides a standard Brownian motion
$B=(B^1,B^2)$ in $\mathbb R^2$; see
\cite[Thm.~V.3.9]{revuz2013continuous}. It may be chosen so that
\[
  Z_t^+=\int_0^t\sqrt{v_s^0}\,dB_s^1,
  \qquad
  Z_t^-=\int_0^t\sqrt{v_s^0}\,dB_s^2.
\]
On the predictable set
$\{(\omega,t):v_t^0(\omega)=0\}$, the missing Brownian noise may be supplied
independently in the two coordinates. Define
\[
  W_t^V:=\frac{B_t^1+\beta B_t^2}{\sqrt{1+\beta^2}},
  \qquad
  c_\beta:=
  \sqrt{\frac{1+\beta^2}{\gamma\mu(1+\beta)^2}}.
\]
Then $W^V$ is a standard Brownian motion and
\[
  d\mathcal H_t=\sqrt{1+\beta^2}\sqrt{v_t^0}\,dW_t^V.
\]
Consequently, \eqref{eq:v-density-ae} becomes
\begin{equation}\label{eq:v-resolvent-ae}
  v_t^0
  =
  \xi_0\left(1-F_{\alpha,\gamma}(t)\right)
  +F_{\alpha,\gamma}(t)
  +c_\beta\int_0^t
    f_{\alpha,\gamma}(t-s)\sqrt{v_s^0}\,dW_s^V,
  \qquad dt\otimes d\mathbb P\text{-a.e.}
\end{equation}

We now apply the differentiability and stochastic-convolution regularity
argument of \cite[Theorem~3.2 and its proof]{jaisson2016rough}, in the
nonzero-initial-value form established in
\cite[Proposition~6.2, Eq.~(28), and its proof]{el2019characteristic}.
To make the identification explicit, set
$\widehat{\mathcal H}:=\mathcal H/\sqrt{1+\beta^2}$.  Then
$[\widehat{\mathcal H}]=Y$, and, since
\[
  c_{\beta,\gamma,\mu}\sqrt{1+\beta^2}=c_\beta,
  \qquad
  \int_0^t F_{\alpha,\gamma}(t-s)\,ds
  =\int_0^t F_{\alpha,\gamma}(s)\,ds,
\]
equation~\eqref{eq:Y-limit-eqn} is equivalently
\begin{equation}\label{eq:Y-limit-eqn-normalized}
  Y_t
  =\int_0^t b(s)\,ds
  +c_\beta\int_0^t
    f_{\alpha,\gamma}(t-s)\widehat{\mathcal H}_s\,ds,
  \qquad [\widehat{\mathcal H}]_t=Y_t.
\end{equation}
Thus $Y$ here is the process denoted by $X$ in the cited proofs, while its
derivative below is the process denoted there by $Y$.  The
Dambis--Dubins--Schwarz theorem gives
$\widehat{\mathcal H}_t=\overline B_{Y_t}$ for a Brownian motion $\overline B$
(up to the usual extension).  Therefore
\eqref{eq:Y-limit-eqn-normalized} is precisely the integrated equation to
which those proofs apply: Proposition~6.2 supplies the extra deterministic
term $b(t)=\xi_0(1-F_{\alpha,\gamma}(t))+F_{\alpha,\gamma}(t)$ absent from the
zero-initial-value statement of Theorem~3.2.  The remaining hypotheses are
the Mittag--Leffler kernel estimates, $\alpha>1/2$, and the localized
integrability established above.  It follows that $Y$ is continuously
differentiable and that
its derivative has a version $v$ such that, for every
$\epsilon\in(0,\alpha-1/2)$,
\[
  v\in C^{\alpha-1/2-\epsilon}([0,T])
  \qquad\text{a.s.}
\]
Moreover, $v=v^0$ $dt\otimes d\mathbb P$-a.e.  Since $v^0\ge0$ almost
everywhere and $v$ is continuous, it follows that $v_t\ge0$ for every
$t\in[0,T]$, almost surely.  The same regularity argument upgrades
\eqref{eq:v-resolvent-ae} to
\begin{equation}\label{eq:v-resolvent-continuous}
  v_t
  =
  \xi_0\left(1-F_{\alpha,\gamma}(t)\right)
  +F_{\alpha,\gamma}(t)
  +c_\beta\int_0^t
    f_{\alpha,\gamma}(t-s)\sqrt{v_s}\,dW_s^V,
  \qquad 0\le t\le T,
\end{equation}
simultaneously on a probability-one event.  Replacing $v^0$ by $v$ changes
neither the Lebesgue integrals nor the stochastic integrals. Hence,
\[
  Y_t=\int_0^t v_s\,ds,
  \qquad
  Z_t^+=\int_0^t\sqrt{v_s}\,dB_s^1,
  \qquad
  Z_t^-=\int_0^t\sqrt{v_s}\,dB_s^2.
\]

To recover the fractional form asserted in the theorem, set
$K_\alpha(t):=\frac{t^{\alpha-1}}{\Gamma(\alpha)}$.
The Mittag--Leffler resolvent identity is
\[
  f_{\alpha,\gamma}
  +\gamma K_\alpha*f_{\alpha,\gamma}
  =\gamma K_\alpha.
\]
To apply the operator $I+\gamma K_\alpha*$ to
\eqref{eq:v-resolvent-continuous}, it remains to justify interchanging
convolution by $K_\alpha$ with the stochastic integral. This interchange
follows, after localization, from the stochastic Fubini theorem
\cite[Theorem~IV.65]{protter2005stochastic}. The version $v$ above is
continuous and adapted, so
\[
  \eta_m:=\inf\{t\in[0,T]:v_t\ge m\}\wedge T
\]
is a stopping time. For the stopped integrand
$\mathbf 1_{\{s\le\eta_m\}}\sqrt{v_s}$, the required integrability condition
follows from its boundedness and Young's inequality:
\[
  \lVert K_\alpha*f_{\alpha,\gamma}\rVert_{L^2([0,T])}
  \le
  \lVert K_\alpha\rVert_{L^2([0,T])}
  \lVert f_{\alpha,\gamma}\rVert_{L^1([0,T])}<\infty.
\]
Thus the resolvent identity below holds up to $\eta_m$.  Since every
continuous sample path of $v$ is bounded on $[0,T]$, one has $\eta_m=T$ for
all sufficiently large $m$.  Using the resolvent identity and patching the
stopped equalities therefore gives
\[
  v
  =
  \xi_0+\gamma K_\alpha*(1-v)
  +\gamma c_\beta K_\alpha*
    \left(\sqrt v\,dW^V\right),
\]
which is exactly \eqref{eq:vt}.  This identifies every limit point and proves
the asserted H\"older regularity.

\subsection{Proof of Theorem~\ref{theorem1}}\label{app:proof-theorem1}

From Proposition~\ref{proposition1}, we have obtained $C$-tightness of
$(\boldsymbol\Lambda^{\tau},\mathbf Y^{\tau},\mathbf Z^{\tau})
_{\tau\ge\tau_{\min}}$, so every subsequence admits a further subsequence
converging in law to some limit point
$(\boldsymbol\Lambda,\mathbf Y,\mathbf Z)$. The same proposition also
identifies the limiting quadratic covariation:
\[
  [\mathbf Z,\mathbf Z]=\mathrm{diag}(\mathbf Y),
  \qquad
  [Z^+,Z^-]_t=0.
\]

Appendix~\ref{app:prop2} shows that every such limit point admits a Brownian
representation with independent components and that its rate process $v$,
defined by $Y_t=\int_0^t v_s\,ds$, satisfies \eqref{eq:vt}. To identify the
joint law of the full limit, define
\[
  A:=Z^+-Z^-,
  \qquad
  \mathcal H:=Z^++\beta Z^-,
  \qquad
  X:=\frac{A}{\sqrt{2}}-\frac{1}{2}Y.
\]
The limiting covariations give
\[
  [A]_t=2Y_t,\qquad
  [\mathcal H]_t=(1+\beta^2)Y_t,\qquad
  [A,\mathcal H]_t=(1-\beta)Y_t.
\]
Consequently, $(X,v)$ is precisely a Volterra--Heston system with
correlation
$\rho=\frac{1-\beta}{\sqrt{2(1+\beta^2)}}$.
The fractional kernel
$K_\alpha(t)=t^{\alpha-1}/\Gamma(\alpha)$, with
$\alpha\in(1/2,1)$, satisfies condition~(2.5) of
\cite{abi2019affine} by its Example~2.3(ii); it is completely monotone, so
its shifted kernels satisfy condition~(3.4) by Example~3.7 of the same
reference. Hence that paper's Theorem~7.1 applies.
The joint law of $(X,v)$ is unique by
\cite[Theorem~7.1]{abi2019affine}; the corresponding integrated-variance law
is characterized by
\cite[Theorems~2.2 and~2.5]{AbiJaber2021WeakExistence}. Moreover,
$\mathcal H$ is
recovered from $v$ by applying $I^{1-\alpha}$ to \eqref{eq:vt}:
\[
  \frac{\sqrt{\gamma}}{\sqrt{\mu}(1+\beta)}\mathcal H_t
  =
  I^{1-\alpha}(v-\xi_0)(t)
  +\gamma\int_0^t(v_s-1)\,ds.
\]
Thus the joint law of $(Y,A,\mathcal H)$ is fixed. Finally,
\[
  Z^-=\frac{\mathcal H-A}{1+\beta},
  \qquad
  Z^+=\frac{\beta A+\mathcal H}{1+\beta},
\]
so the law of $(\boldsymbol\Lambda,\mathbf Y,\mathbf Z)$ is unique.

Hence all subsequential limit points coincide, and the whole sequence converges in law to the unique limit described in Theorem~\ref{theorem1}. This completes the proof of Theorem~\ref{theorem1}.


\subsection{Proof of Corollary~\ref{maintheorem}}\label{sec:proofoftheorem2}
By Theorem~\ref{theorem1},
$(Z^{\tau,+},Z^{\tau,-},Y^{\tau,+})$ converges in the Skorokhod topology to
$(Z^+,Z^-,Y^+)$.  The limit processes are continuous by
Proposition~\ref{proposition1}.  Define the fixed-coefficient approximation
\[
  \overline P_t^\tau
  :=
  \sqrt{\frac{\theta}{2}}\left(Z_t^{\tau,+}-Z_t^{\tau,-}\right)
  -\frac{\theta}{2}Y_t^{\tau,+}
\]
and the compensation remainder
\[
  \mathcal E_t^\tau
  :=P_t^\tau-\overline P_t^\tau
  =-\frac{\theta}{2}
    \left(\frac{d_\tau}{c_\tau^2}-1\right)Y_t^{\tau,+}.
\]
By \eqref{eq:dtau-expansion} and
$c_\tau=\sqrt{\theta\gamma/(2\mu)}\,\tau^{-\alpha}$,
\begin{equation}\label{eq:martingale-compensation-ratio}
  \left|\frac{d_\tau}{c_\tau^2}-1\right|
  \le C\tau^{-\alpha}
\end{equation}
for all sufficiently large $\tau$.  Moreover, $Y^{\tau,+}$ is nondecreasing
and \eqref{eq:Elambda-prop1} gives
\[
  \sup_{\tau\ge\tau_{\min}}\E\left[Y_T^{\tau,+}\right]
  =\sup_{\tau\ge\tau_{\min}}
    \frac{1-a_\tau}{\mu\tau^\alpha}
    \sum_{s=1}^{\lfloor\tau T\rfloor}\E\left[\lambda_s^{\tau,+}\right]
  <\infty.
\]
Consequently,
\[
  \E\!\left[\sup_{t\in[0,T]}|\mathcal E_t^\tau|\right]
  \le C\tau^{-\alpha}\E\left[Y_T^{\tau,+}\right]
  \longrightarrow0.
\]
Thus $\|P^\tau-\overline P^\tau\|_{\infty,T}\to0$ in probability, and
\[
  d_{J_1}\left(P^\tau,\overline P^\tau\right)
  \le \left\|P^\tau-\overline P^\tau\right\|_{\infty,T}
  \xrightarrow{\mathbb P}0.
\]
Since
the fixed linear map defining $\overline P^\tau$ is continuous in the
Skorokhod $J_1$ topology at continuous limits, the continuous-mapping theorem
and Slutsky's theorem give
\[
  P_t
  =
  \sqrt{\frac{\theta}{2}}
  \int_0^t\sqrt{v_s}\,\left(dB_s^1-dB_s^2\right)
  -\frac{\theta}{2}\int_0^t v_s\,ds .
\]
Setting
\[
  V_t:=\theta v_t,
  \qquad
  W_t^X:=\frac{B_t^1-B_t^2}{\sqrt2},
\]
therefore yields
\[
  P_t=\int_0^t\sqrt{V_s}\,dW_s^X-\frac12\int_0^tV_s\,ds.
\]
Multiplying \eqref{eq:vt} by $\theta$ and using the parameter
identifications in the statement gives \eqref{eq:V}.  Finally, with
$W^V=(B^1+\beta B^2)/\sqrt{1+\beta^2}$ and the independence of $B^1,B^2$,
$d\left\langle W^X,W^V\right\rangle_t
  =
  \frac{1-\beta}{\sqrt{2(1+\beta^2)}}\,dt
  =
  \rho\,dt$
by \eqref{eq:rho-def}.  This proves the corollary.


\section{Technical Results for the Weak-Error Analysis}
\label{app:weak-error-proofs}

For $\alpha\in(0,1)$ and $f\in C^\alpha([0,T])$, we define
\begin{equation}\label{eq:holder-seminorm}
  [f]_{C^\alpha([0,T])}
  :=
  \sup_{\substack{x,y\in[0,T]\\x\neq y}}
  \frac{|f(x)-f(y)|}{|x-y|^\alpha},
\end{equation}
and
\begin{equation}\label{eq:holder-norm}
  \|f\|_{C^\alpha([0,T])}
  :=
  \|f\|_\infty+[f]_{C^\alpha([0,T])}.
\end{equation}

The results below follow the proof chain used in Section~\ref{sec:weak-euro}.
We first derive the exact transform recursion and its stopped second-order
reduction, then pass to equivalent discrete and continuous resolvent forms.
Regularity and quadrature estimates lead to the cumulative-resolvent, operator,
and Riccati comparisons.  Finally, the independent weighted-baseline estimate
is combined with the Riccati comparison to obtain the transform and full
call-price convergence results.

\subsection{Proof of Proposition~\ref{prop:exact-transform-weak}}\label{app:weak-euro-exact}

\begin{proof}
We use the Poisson branching--immigration representation of INAR($\infty$).

Because every reproduction lag is at least one, a cluster observed over a
finite horizon of $n$ time steps has at most $n$ generations. Starting from
\eqref{eq:M0-exact}, induction using the Poisson generating-function identity,
applied first to absolute moments, shows that the cluster transforms below are
finite for every fixed $n$ and $z\in\C$. The recursion expresses them through
finite compositions of entire functions; the same is therefore true of
$\Phi^\tau(\cdot,T)$.

Fix $z\in\C$. Consider first a single ``$+$''-immigrant born at the current time. Its direct contribution to $zP_T^\tau$ is $\Theta_+^\tau(z)$, hence the factor $e^{\Theta_+^\tau(z)}$.

Now fix a lag $k\in\{1,\dots,n\}$. Because $q_k^{\tau,+}$ is the lag-$k$ coefficient associated with a ``$+$''-parent, the numbers of ``$+$''-children and ``$-$''-children produced at lag $k$ are independent Poisson random variables, both with mean $q_k^{\tau,+}$. Each such child starts an independent subcluster with remaining horizon $n-k$. Therefore, by the Poisson generating-function identity
$\E\left[s^{\mathrm{Poisson}(\lambda)}\right]=e^{\lambda(s-1)}$,
the lag-$k$ continuation factor equals
\[
  \exp\!\left(q_k^{\tau,+}(\mathcal M_{n-k}^{\tau,+}(z)-1)\right)\,
  \exp\!\left(q_k^{\tau,+}(\mathcal M_{n-k}^{\tau,-}(z)-1)\right)
  =
  \exp\!\left(q_k^{\tau,+}G_{n-k}^\tau(z)\right).
\]
Multiplying these independent continuation factors over $k=1,\dots,n$, and including the ancestor contribution $e^{\Theta_+^\tau(z)}$, we obtain
\[
  \mathcal M_n^{\tau,+}(z)
  =
  e^{\Theta_+^\tau(z)}
  \prod_{k=1}^n
  \exp\!\left(q_k^{\tau,+}G_{n-k}^\tau(z)\right)
  =
  \exp\!\left(
    \Theta_+^\tau(z)
    +\sum_{k=1}^n q_k^{\tau,+}G_{n-k}^\tau(z)
  \right),
\]
which is exactly \eqref{eq:Mplus-exact}.

The proof of \eqref{eq:Mminus-exact} is identical. For a single ``$-$''-immigrant, the lag-$k$ numbers of ``$+$''- and ``$-$''-children are independent Poisson random variables with common mean $q_k^{\tau,-}$, so the lag-$k$ continuation factor is
\[
  \exp\!\left(q_k^{\tau,-}(\mathcal M_{n-k}^{\tau,+}(z)-1)\right)\,
  \exp\!\left(q_k^{\tau,-}(\mathcal M_{n-k}^{\tau,-}(z)-1)\right)
  =
  \exp\!\left(q_k^{\tau,-}G_{n-k}^\tau(z)\right),
\]
which gives \eqref{eq:Mminus-exact}.

Now define $G_n^\tau(z):=\mathcal M_n^{\tau,+}(z)+\mathcal M_n^{\tau,-}(z)-2$. Summing \eqref{eq:Mplus-exact} and \eqref{eq:Mminus-exact} and subtracting $2$ yields the closed recursion \eqref{eq:G-rec-exact}.

Finally, at each time $m\in\{1,\dots,m_\tau\}$, the numbers of ``$+$''- and ``$-$''-immigrants are independent Poisson random variables with common mean $\hat\mu_\tau(m)$, and clusters generated by distinct immigrants are independent. Hence, again by the Poisson generating-function identity,
\[
  \Phi^\tau(z,T)
  =
  \prod_{m=1}^{m_\tau}
  \exp\!\left(
    \hat\mu_\tau(m)\left(\mathcal M_{m_\tau-m}^{\tau,+}(z)-1\right)
  \right)
  \exp\!\left(
    \hat\mu_\tau(m)\left(\mathcal M_{m_\tau-m}^{\tau,-}(z)-1\right)
  \right).
\]
Combining the two exponential factors and using the definition of $G_n^\tau(z)$, we obtain
\[
  \Phi^\tau(z,T)
  =
  \exp\!\left(
    \sum_{m=1}^{m_\tau}\hat\mu_\tau(m)\,
    \left(\mathcal M_{m_\tau-m}^{\tau,+}(z)+\mathcal M_{m_\tau-m}^{\tau,-}(z)-2\right)
  \right),
\]
which is exactly \eqref{eq:Phi-exact}.

It remains to verify the martingale assertion. At $z=1$, the initial
conditions and the identity
$e^{c_\tau-d_\tau}+e^{-c_\tau}=2$ give
\[
  G_0^\tau(1)
  =e^{\Theta_+^\tau(1)}+e^{\Theta_-^\tau(1)}-2
  =0.
\]
Suppose that $G_j^\tau(1)=0$ for $j=0,\dots,n-1$. Then
\eqref{eq:G-rec-exact} yields
\[
  G_n^\tau(1)
  =e^{\Theta_+^\tau(1)}+e^{\Theta_-^\tau(1)}-2
  =0.
\]
Thus $G_n^\tau(1)=0$ for every $n\ge0$ by induction. Substitution into
\eqref{eq:Phi-exact} gives
$\Phi^\tau(1,T)=1$,
which completes the proof.
\end{proof}


\subsection{Stopped strip estimates and continuation setup}\label{app:weak-euro-second-order}

We prove the smallness estimate required for the Taylor expansion by a
continuation argument. Fix $\eta,R$ and recall $L$ and the first exit
index $\sigma_\tau$ introduced above. Equivalently,
\begin{equation}\label{eq:exit-index-new33}
  \sigma_\tau
  :=\inf\left\{0\le n\le m_\tau:
       \sup_{z\in\Gamma_{\eta,R}}|H_n^\tau(z)|>L\right\},
\end{equation}
with the convention $\inf\varnothing=m_\tau+1$.

\begin{lemma}[Stopped uniform strip smallness]
\label{prop:smallness-new33}
Under Assumptions~\ref{ass:strip} and~\ref{ass:riccati-bound}, there are
$\tau_0\ge1$ and $C_{\eta,R,T}>0$ such that, for $\tau\ge\tau_0$ and
$0\le n\le m_\tau\wedge\sigma_\tau$,
\begin{equation}\label{eq:stopped-smallness-new33}
 \sup_{z\in\Gamma_{\eta,R}}
 \left(
 |G_n^\tau(z)|+|A_n^\tau(z)|+|B_n^\tau(z)|
 \right)
 \le C_{\eta,R,T}\varepsilon_\tau.
\end{equation}
In addition,
\begin{equation}\label{eq:stopped-H-bound-new33}
  \sup_{0\le n\le m_\tau\wedge\sigma_\tau}
  \sup_{z\in\Gamma_{\eta,R}}|H_n^\tau(z)|
  \le C_{\eta,R,T}.
\end{equation}
\end{lemma}

\begin{proof}
The cancellation of the linear $c_\tau$ terms at $n=0$ gives
$G_0^\tau=\mathcal O(\varepsilon_\tau^2)$, uniformly on
$\Gamma_{\eta,R}$.  Hence $H_0^\tau=\mathcal O(\varepsilon_\tau)$, so
$\sigma_\tau\ne0$ for large $\tau$.

For $n<\sigma_\tau$, the definition of the exit index and
$G_j^\tau=(\theta/\mu)\varepsilon_\tau H_j^\tau$ give
\[
  \sup_{0\le j\le n}\sup_{z\in\Gamma_{\eta,R}}|G_j^\tau(z)|
  \le \frac{\theta}{\mu}L\varepsilon_\tau.
\]
If $n=\sigma_\tau\le m_\tau$, the quantities $A_n^\tau,B_n^\tau$ only
involve $G_j^\tau$ with $j<n$.  Thus, both for $n<\sigma_\tau$ and at the
possible exit index,
\begin{align*}
 |A_n^\tau(z)|+|B_n^\tau(z)|
 &\le |z|(2c_\tau+d_\tau)
   +\sum_{k=1}^n\varphi_k^\tau |G_{n-k}^\tau(z)|
 \le C_{\eta,R,T}\varepsilon_\tau.
\end{align*}
Since $|e^w-1|\le e^{|w|}|w|$, the exact identity
$G_n^\tau=e^{A_n^\tau}+e^{B_n^\tau}-2$ then gives
$|G_n^\tau|\le C_{\eta,R,T}\varepsilon_\tau$, including at the possible
exit index. Rescaling gives \eqref{eq:stopped-H-bound-new33} and proves the
stopped bounds.
\end{proof}


\subsection{Proof of Proposition~\ref{prop:reduced-recursion-new33}}\label{app:proof-reduced-recursion-new33}
Set
$\overline G_n^\tau(z):=\sum_{k=1}^n \varphi_k^\tau\,G_{n-k}^\tau(z)$.
Then
\[
  A_n^\tau(z)=z(c_\tau-d_\tau)+\frac{1}{1+\beta}\overline G_n^\tau(z),
\qquad
  B_n^\tau(z)=-zc_\tau+\frac{\beta}{1+\beta}\overline G_n^\tau(z).
\]

By Lemma~\ref{prop:smallness-new33}, the quantities $A_n^\tau(z)$ and $B_n^\tau(z)$
are uniformly $\mathcal O(\varepsilon_\tau)$ on the stopped grid
$\{0,\dots,m_\tau\wedge\sigma_\tau\}\times\Gamma_{\eta,R}$.
Hence, for $\tau$ large enough, they belong to a fixed complex neighborhood of the origin on
which the third-order Taylor formula for the exponential map is uniform:
$e^x = 1+x+\frac{x^2}{2}+r_{\exp}(x)$,
with $|r_{\exp}(x)|\le C|x|^3$.
Applying this with $x=A_n^\tau(z)$ and $x=B_n^\tau(z)$, and subtracting $2$, yields
\[
  G_n^\tau(z)
  =A_n^\tau(z)+B_n^\tau(z)
   +\frac12\left((A_n^\tau(z))^2+(B_n^\tau(z))^2\right)
   +r_n^\tau(z),
\]
where
\[
  |r_n^\tau(z)|\le C\left(|A_n^\tau(z)|^3+|B_n^\tau(z)|^3\right)\le C\varepsilon_\tau^3
\]
uniformly for $0\le n\le m_\tau\wedge\sigma_\tau$ and
$z\in\Gamma_{\eta,R}$.

The linear term is
\[
  A_n^\tau(z)+B_n^\tau(z)=\overline G_n^\tau(z)-zd_\tau.
\]
For the quadratic term, a direct computation gives
\begin{align*}
  \frac12\left((A_n^\tau(z))^2+(B_n^\tau(z))^2\right)
  &=
    \frac12\left((z(c_\tau-d_\tau))^2+z^2c_\tau^2\right) \\
  &\quad
    + \frac{z}{1+\beta}\left((c_\tau-d_\tau)-\beta c_\tau\right)\overline G_n^\tau(z)
    + \frac{1+\beta^2}{2(1+\beta)^2}\left(\overline G_n^\tau(z)\right)^2.
\end{align*}
Since $d_\tau=\mathcal O(\varepsilon_\tau^2)$ and $c_\tau=\mathcal O(\varepsilon_\tau)$, we have
\[
  \frac12\left((z(c_\tau-d_\tau))^2+z^2c_\tau^2\right)
  = c_\tau^2 z^2 + \mathcal O(\varepsilon_\tau^3),
\]
uniformly on $\Gamma_{\eta,R}$, and
\[
  \frac{z}{1+\beta}\left((c_\tau-d_\tau)-\beta c_\tau\right)
  = \frac{1-\beta}{1+\beta}\,c_\tau z + \mathcal O(\varepsilon_\tau^2)
\]
uniformly on $\Gamma_{\eta,R}$.
Moreover, by Lemma~\ref{prop:smallness-new33},
$\overline G_n^\tau(z)=\mathcal O(\varepsilon_\tau)$
uniformly on the stopped grid, and thus the $\mathcal O(\varepsilon_\tau^2)$
coefficient error in the cross term contributes only $\mathcal O(\varepsilon_\tau^3)$.
Finally, \eqref{eq:dtau-expansion} and
$c_\tau=\mathcal O(\varepsilon_\tau)$ give
\[
  -zd_\tau+c_\tau^2z^2
  =c_\tau^2(z^2-z)+z(c_\tau^2-d_\tau)
  =c_\tau^2(z^2-z)+\mathcal O(\varepsilon_\tau^3)
\]
uniformly on $\Gamma_{\eta,R}$.  The last term is included in
$\mathcal R_n^\tau(z)$.
Collecting the terms proves
\eqref{eq:G-reduced-new33}--\eqref{eq:R-reduced-new33} on the stopped grid.
The proof of Proposition~\ref{prop:Volterra-rate-fixed} starts from this
stopped version, establishes $\sigma_\tau=m_\tau+1$ for all sufficiently large
$\tau$, and thereby yields the corresponding full-grid estimates.

\subsection{Equivalent resolvent reformulation}\label{app:proof-barH}

The first two lemmas in this subsection give algebraic reformulations of the
discrete recursion. Lemma~\ref{lem:barH} replaces the averaged state variable
by the current state up to a lower-order error, and
Lemma~\ref{lem:renewal-inversion} rewrites the resulting recursion in resolvent
form. Lemma~\ref{lem:continuous-resolvent} derives the analogous continuous
resolvent representation. Together they reduce the comparison to matching the
two resolvent kernels.

\begin{lemma}
\label{lem:barH}
With the convention that the empty sum is zero when $n=0$, uniformly for
$0\le n\le m_\tau\wedge\sigma_\tau$ and $z\in\Gamma_{\eta,R}$,
\begin{equation}\label{eq:H-barH-close}
  \left|
    H_n^\tau(z)-\sum_{k=1}^n\varphi_k H_{n-k}^\tau(z)
  \right|
  \le C_{\eta,R,T}\tau^{-\alpha}.
\end{equation}
Consequently,
\begin{equation}\label{eq:Htau-rec-local}
  H_n^\tau(z)
  =
  \sum_{k=1}^n\varphi_k^\tau H_{n-k}^\tau(z)
  +\tau^{-\alpha}F_z(H_n^\tau(z))
  +\widehat E_n^\tau(z),
\end{equation}
where
\begin{equation}\label{eq:Ehat-bound}
  \sup_{0\le n\le m_\tau\wedge\sigma_\tau}\sup_{z\in\Gamma_{\eta,R}}
  |\widehat E_n^\tau(z)|
  \le C_{\eta,R,T}\tau^{-2\alpha}.
\end{equation}
\end{lemma}

\begin{proof}

We first treat the case $n=0$. Since
\[
  H_0^\tau(z)=\frac{\mu}{\theta\varepsilon_\tau}G_0^\tau(z),
  \qquad
  G_0^\tau(z)=e^{\Theta_+^\tau(z)}+e^{\Theta_-^\tau(z)}-2,
\]
with $\Theta_+^\tau(z)=z(c_\tau-d_\tau)$ and $\Theta_-^\tau(z)=-zc_\tau$,
and since $c_\tau=\mathcal O(\varepsilon_\tau)$, $d_\tau=\mathcal O(\varepsilon_\tau^2)$,
the order-$c_\tau$ terms cancel and therefore
$\sup_{z\in\Gamma_{\eta,R}}|G_0^\tau(z)|\le C_{\eta,R,T}\varepsilon_\tau^2$.
Hence,
\[
\sup_{z\in\Gamma_{\eta,R}}|H_0^\tau(z)|
  \le C_{\eta,R,T}\varepsilon_\tau
  \le C_{\eta,R,T}\tau^{-\alpha}.
\]
Thus \eqref{eq:H-barH-close} holds for $n=0$.

Now let $1\le n\le m_\tau\wedge\sigma_\tau$. By
Lemma~\ref{prop:smallness-new33} and the rescaling
\eqref{eq:H-rescaled-def},
\[
  \sup_{0\le n\le m_\tau\wedge\sigma_\tau}\sup_{z\in\Gamma_{\eta,R}}
  |H_n^\tau(z)|\le C_{\eta,R,T}.
\]
Hence,
\[
  \sup_{0\le n\le m_\tau\wedge\sigma_\tau}\sup_{z\in\Gamma_{\eta,R}}
  \left|\sum_{k=1}^n\varphi_k H_{n-k}^\tau(z)\right|
  \le C_{\eta,R,T},
\]
since $\sum_{k\ge1}\varphi_k=1$.

Subtracting $\sum_{k=1}^n\varphi_kH_{n-k}^\tau(z)$ from both sides of
\eqref{eq:Htau-rec-main-again} gives
\begin{align*}
  &H_n^\tau(z)-\sum_{k=1}^n\varphi_k H_{n-k}^\tau(z)\\
  &=
  -(1-a_\tau)\sum_{k=1}^n\varphi_k H_{n-k}^\tau(z)
  +\tau^{-\alpha}
   F_z\!\left(\sum_{k=1}^n\varphi_k H_{n-k}^\tau(z)\right)
  +E_n^\tau(z).
\end{align*}
Using $1-a_\tau=\gamma\tau^{-\alpha}$, the boundedness of the convolution term,
the boundedness of $F_z$ on the relevant range, and \eqref{eq:En-bound},
we obtain \eqref{eq:H-barH-close}.

Write $\overline H_n^\tau:=\sum_{k=1}^n\varphi_kH_{n-k}^\tau$.  Since $F_z$
is quadratic and the relevant arguments remain uniformly bounded on
$\Gamma_{\eta,R}$, it is uniformly Lipschitz on this range.  Hence, by
\eqref{eq:H-barH-close},
\[
  \left|F_z\left(\overline H_n^\tau(z)\right)-F_z(H_n^\tau(z))\right|
  \le C_{\eta,R,T}\left|\overline H_n^\tau(z)-H_n^\tau(z)\right|
  \le C_{\eta,R,T}\tau^{-\alpha}.
\]
Multiplying by $\tau^{-\alpha}$ and using
$a_\tau\sum_{k=1}^n\varphi_k H_{n-k}^\tau(z)
  =\sum_{k=1}^n\varphi_k^\tau H_{n-k}^\tau(z)$,
we obtain \eqref{eq:Htau-rec-local} with
\[
  \widehat E_n^\tau(z)
  :=
  E_n^\tau(z)
  +\tau^{-\alpha}
  \left[
    F_z\!\left(\sum_{k=1}^n\varphi_k H_{n-k}^\tau(z)\right)-F_z(H_n^\tau(z))
  \right].
\]
The bound \eqref{eq:Ehat-bound} follows immediately.
\end{proof}


\begin{lemma}[Renewal inversion]
\label{lem:renewal-inversion}
Let
$\psi^\tau:=\sum_{m\ge1}(\varphi^\tau)^{*m}$
with $\varphi^\tau=a_\tau\varphi$
be the renewal kernel associated with $\varphi^\tau$, and define
\begin{equation}\label{eq:S-discrete-def}
  \mathsf S_k^\tau
  :=
  \frac{\tau(1-a_\tau)}{a_\tau}\psi_k^\tau,
  \qquad k\ge1.
\end{equation}
Then \eqref{eq:Htau-rec-local} is equivalent to
\begin{equation}\label{eq:Htau-resolvent-discrete}
\begin{aligned}
  H_n^\tau(z)
  &=
  \tau^{-\alpha}F_z(H_n^\tau(z))
  +\widehat E_n^\tau(z)
\\
  &\quad
  +
  \frac{a_\tau}{1-a_\tau}\tau^{-\alpha-1}
  \sum_{j=0}^{n-1}\mathsf S_{n-j}^\tau F_z(H_j^\tau(z))
  +
  \sum_{j=0}^{n-1}\psi_{n-j}^\tau\widehat E_j^\tau(z),
\end{aligned}
\end{equation}
for every $1\le n\le m_\tau\wedge\sigma_\tau$.
\end{lemma}

\begin{proof}
Equation \eqref{eq:Htau-rec-local} is a discrete renewal equation of the form
\[
  H^\tau = U^\tau + \varphi^\tau * H^\tau,
\]
where
\[
  U_n^\tau(z):=\tau^{-\alpha}F_z(H_n^\tau(z))+\widehat E_n^\tau(z).
\]
Fix $1\le n\le m_\tau\wedge\sigma_\tau$. Since $\varphi^\tau$ is supported
on the positive integers, coefficientwise iteration terminates after at most
$n$ steps and gives
\[
  H_n^\tau(z)
  =U_n^\tau(z)
   +\sum_{m=1}^{n}\left((\varphi^\tau)^{*m}*U^\tau\right)_n(z).
\]
Thus no summability assumption on the sequences $H^\tau$ and $U^\tau$ is
needed. Since $\|\varphi^\tau\|_1=a_\tau<1$, the renewal kernel
$\psi^\tau:=\sum_{m\ge1}(\varphi^\tau)^{*m}$ is also well defined in
$\ell^1$, and the preceding finite coefficientwise identity becomes
\[
  H_n^\tau(z)
  =
  U_n^\tau(z)
  +\sum_{j=0}^{n-1}\psi_{n-j}^\tau U_j^\tau(z),
  \qquad 1\le n\le m_\tau\wedge\sigma_\tau.
\]
Substituting the definition of $U_j^\tau(z)$ yields
\[
  H_n^\tau(z)
  =
  \tau^{-\alpha}F_z(H_n^\tau(z))
  +\widehat E_n^\tau(z)
  +\sum_{j=0}^{n-1}\psi_{n-j}^\tau
   \left(
     \tau^{-\alpha}F_z(H_j^\tau(z))+\widehat E_j^\tau(z)
   \right).
\]
Using \eqref{eq:S-discrete-def},
\[
  \psi_{n-j}^\tau
  =
  \frac{a_\tau}{1-a_\tau}\frac1\tau \mathsf{S}_{n-j}^\tau,
\]
we obtain \eqref{eq:Htau-resolvent-discrete}.

Conversely, the renewal identity
$\psi^\tau
  =\varphi^\tau+\varphi^\tau*\psi^\tau$
and coefficientwise associativity of convolution show that
\[
  \varphi^\tau*H^\tau
  =\varphi^\tau*U^\tau
   +\varphi^\tau*\psi^\tau*U^\tau
  =\psi^\tau*U^\tau.
\]
Hence $H^\tau=U^\tau+\psi^\tau*U^\tau$ implies
$H^\tau=U^\tau+\varphi^\tau*H^\tau$ on the same finite grid. This proves the
claimed equivalence.
\end{proof}


\begin{lemma}[Continuous resolvent form]
\label{lem:continuous-resolvent}
Fix $\eta\in\R$ and $R>0$ with
$\Gamma_{\eta,R}\subset\mathcal S_T$, and suppose that
Assumption~\ref{ass:riccati-bound} holds.  For each
$z\in\Gamma_{\eta,R}$, the solution of the continuous fractional Riccati
equation admits the representation
\begin{equation}\label{eq:continuous-resolvent}
  h(t,z)
  =
  \frac1\gamma\int_0^t f_{\alpha,\gamma}(t-s)F_z(h(s,z))\,ds,
  \qquad 0\le t\le T.
\end{equation}
\end{lemma}

\begin{proof}
Write
\[
  K_\alpha(t):=\frac{t^{\alpha-1}}{\Gamma(\alpha)},
  \qquad
  g_z(t):=F_z(h(t,z)),
  \qquad 0<t\le T,
\]
and let $*$ denote Volterra convolution on $[0,T]$.  By
Assumption~\ref{ass:riccati-bound} and the fact that $F_z$ is a polynomial,
both $h(\cdot,z)$ and $g_z$ are bounded on $[0,T]$.  Equation
\eqref{eq:riccati-volterra-weak} can be written, on this finite interval, as
\begin{equation}\label{eq:finite-horizon-resolvent-start}
  h=K_\alpha*g_z-\gamma K_\alpha*h.
\end{equation}
Iterating \eqref{eq:finite-horizon-resolvent-start} $N+1$ times gives, for
$N\ge0$,
\begin{equation}\label{eq:finite-horizon-resolvent-iterate}
  h
  =\sum_{n=0}^{N}(-\gamma)^n
       K_\alpha^{*(n+1)}*g_z
   +(-\gamma)^{N+1}K_\alpha^{*(N+1)}*h.
\end{equation}
The beta-function identity yields
\[
  K_\alpha^{*m}(t)
  =\frac{t^{m\alpha-1}}{\Gamma(m\alpha)},
  \qquad m\ge1,
\]
and therefore
\[
  \sup_{0\le t\le T}
  \left|
    \gamma^{N+1}\left(K_\alpha^{*(N+1)}*h\right)(t)
  \right|
  \le
  \lVert h(\cdot,z)\rVert_{\infty,[0,T]}
  \frac{\gamma^{N+1}T^{(N+1)\alpha}}
       {\Gamma((N+1)\alpha+1)}
  \longrightarrow0.
\]
Moreover,
\[
  \sum_{n=0}^{\infty}
  \gamma^n\lVert K_\alpha^{*(n+1)}\rVert_{L^1(0,T)}
  =
  \sum_{n=0}^{\infty}
  \frac{\gamma^nT^{(n+1)\alpha}}
       {\Gamma((n+1)\alpha+1)}
  =\frac{E_\alpha(\gamma T^\alpha)-1}{\gamma}<\infty.
\]
Hence the kernel series converges absolutely in $L^1(0,T)$.  Define
\[
  r_N:=\sum_{n=0}^N(-\gamma)^nK_\alpha^{*(n+1)},
  \qquad
  r_{\alpha,\gamma}:=
  \sum_{n=0}^{\infty}(-\gamma)^nK_\alpha^{*(n+1)},
\]
where the second series is understood in $L^1(0,T)$.  Since
\[
  \lVert(r_{\alpha,\gamma}-r_N)*g_z\rVert_{\infty,[0,T]}
  \le
  \lVert g_z\rVert_{\infty,[0,T]}
  \lVert r_{\alpha,\gamma}-r_N\rVert_{L^1(0,T)},
\]
passage to the limit in \eqref{eq:finite-horizon-resolvent-iterate} gives
\[
  h=r_{\alpha,\gamma}*g_z.
\]
For $t>0$, the defining series of the two-parameter Mittag--Leffler function
identifies
\begin{align*}
  r_{\alpha,\gamma}(t)
  =\sum_{n=0}^{\infty}
    \frac{(-\gamma)^n t^{(n+1)\alpha-1}}
         {\Gamma((n+1)\alpha)}
         =t^{\alpha-1}E_{\alpha,\alpha}(-\gamma t^\alpha)
   =\frac1\gamma f_{\alpha,\gamma}(t).
\end{align*}
Consequently, for $0\le t\le T$,
\[
  h(t,z)
  =\frac1\gamma\int_0^t
    f_{\alpha,\gamma}(t-s)F_z(h(s,z))\,ds,
\]
which is \eqref{eq:continuous-resolvent}.
\end{proof}


\subsection{Regularity estimates for the continuous fractional Riccati equation}

\begin{proposition}[Bounded-strip regularity of the continuous fractional Riccati solution]
\label{prop:riccati-regularity}
Fix $\eta\in\R$ and $R>0$ such that
\[
  \Gamma_{\eta,R}:=\{\eta+\ii\xi:\ |\xi|\le R\}\subset \mathcal S_T.
\]
Under Assumptions~\ref{ass:strip} and~\ref{ass:riccati-bound}, the map
$(t,z)\mapsto h(t,z)$ is
jointly continuous on $[0,T]\times\Gamma_{\eta,R}$, each map
$t\mapsto h(t,z)$ belongs to $C^\alpha([0,T])\cap C^1((0,T])$, and there exists a
constant $C_{\eta,R,T}>0$ such that, uniformly for $z\in\Gamma_{\eta,R}$,
\begin{align}
  |h(t,z)| &\le C_{\eta,R,T}t^\alpha,
  &&0\le t\le T,\label{eq:h-small-t-prop}\\
  |h(t,z)-h(s,z)| &\le C_{\eta,R,T}|t-s|^\alpha,
  &&0\le s,t\le T,\label{eq:h-holder-t-prop}\\
  |\partial_t h(t,z)| &\le C_{\eta,R,T}t^{\alpha-1},
  &&0<t\le T.\label{eq:h-derivative-t-prop}
\end{align}
Moreover, the map $z\mapsto h(\cdot,z)$ is locally Lipschitz from
$\Gamma_{\eta,R}$ into $C([0,T])$.
Define
\[
  g_z(t):=F_z(h(t,z))
  =\frac12(z^2-z)+\gamma\rho\nu z\,h(t,z)
    +\frac{(\gamma\nu)^2}{2}(h(t,z))^2.
\]
Then $g_z$ is uniformly bounded and absolutely continuous on $[0,T]$, and,
uniformly for $z\in\Gamma_{\eta,R}$,
\begin{align}
  |g_z(t)-g_z(0)| &\le C_{\eta,R,T}t^\alpha,
  &&0\le t\le T,\label{eq:g-small-t-new}\\
  |g_z(t)-g_z(s)| &\le C_{\eta,R,T}|t-s|^\alpha,
  &&0\le s,t\le T,\label{eq:g-holder-t-new}\\
  |\partial_t g_z(t)| &\le C_{\eta,R,T}t^{\alpha-1},
  &&0<t\le T.\label{eq:g-derivative-t-new}
\end{align}
\end{proposition}

\begin{proof}
Set
\[
  \widetilde F_z(x):=\frac12(z^2-z)+\gamma(\rho\nu z-1)x
           +\frac{(\gamma\nu)^2}{2}x^2
  =F_z(x)-\gamma x.
\]
The Volterra equation is
$h(\cdot,z)=I^\alpha \widetilde F_z(h(\cdot,z))$.
Assumption~\ref{ass:riccati-bound} and compactness of
$\Gamma_{\eta,R}$ imply that $\widetilde F_z$ and its first derivatives are uniformly
bounded on the range of $h$.  The Volterra representation therefore gives
\[
  |h(t,z)|\le Ct^\alpha,\qquad
  |h(t,z)-h(s,z)|\le C|t-s|^\alpha,
\]
where the second estimate follows by splitting the two fractional integrals
at $s$ and integrating the kernel difference.  For $z_1,z_2\in
\Gamma_{\eta,R}$, subtracting the two Volterra equations and using the uniform
Lipschitz bounds for $\widetilde F_z$ gives, with
$u(t):=|h(t,z_1)-h(t,z_2)|$,
\[
  u(t)
  \le C|z_1-z_2|t^\alpha
     +C(K_\alpha*u)(t),
  \qquad
  K_\alpha(t):=\frac{t^{\alpha-1}}{\Gamma(\alpha)}.
\]
Iterating this inequality, letting the number of iterations tend to infinity,
and using
$K_\alpha^{*m}(t)=t^{m\alpha-1}/\Gamma(m\alpha)$ for $m\ge1$ yields
\[
  u(t)
  \le
  C|z_1-z_2|\Gamma(\alpha+1)
  \sum_{m=0}^\infty
  \frac{C^m t^{(m+1)\alpha}}
       {\Gamma((m+1)\alpha+1)}
  \le C_T|z_1-z_2|,
  \qquad 0\le t\le T.
\]
This also proves joint continuity.

For $p_z(t):=\widetilde F_z(h(t,z))$, the preceding estimates give
$p_z\in C^\alpha([0,T])$ uniformly on the strip.  To differentiate its
fractional integral, write
\[
  I^\alpha p_z(t)
  =\frac{p_z(t)t^\alpha}{\Gamma(\alpha+1)}
   +\frac1{\Gamma(\alpha)}\int_0^t
     (t-s)^{\alpha-1}\left(p_z(s)-p_z(t)\right)\,ds.
\]
Taking difference quotients in this identity, the terms containing the
increment of $p_z(t)$ cancel.  The remaining integrands are dominated near
$s=t$ by $C(t-s)^{2\alpha-2}$, which is integrable because $\alpha>1/2$.
Dominated convergence therefore yields, for $t>0$,
\[
  \partial_t h(t,z)
  =
  \frac{p_z(t)}{\Gamma(\alpha)}t^{\alpha-1}
  +\frac{\alpha-1}{\Gamma(\alpha)}
   \int_0^t(t-s)^{\alpha-2}\left(p_z(s)-p_z(t)\right)\,ds .
\]
The right-hand side is continuous on every interval $[\delta,T]$ with
$\delta>0$.  Indeed, after the change of variables $r=t-s$, split the
integral into $r\le\varepsilon$ and $r>\varepsilon$.  The uniform
$\alpha$-H\"older bound for $p_z$ controls the first part by
$C\varepsilon^{2\alpha-1}$, while continuity and dominated convergence apply
to the second part.  Hence $h(\cdot,z)\in C^1((0,T])$.

The same domination gives
$|\partial_t h(t,z)|\le C(t^{\alpha-1}+t^{2\alpha-1})
\le C_Tt^{\alpha-1}$.  Since this bound is integrable at zero, the
fundamental theorem of calculus on $[\delta,t]$, followed by
$\delta\downarrow0$ and $h(\delta,z)\to h(0,z)=0$, gives
\[
  h(t,z)-h(0,z)=\int_0^t\partial_s h(s,z)\,ds.
\]
Thus $h(\cdot,z)$ is absolutely continuous on $[0,T]$, and all three estimates
for $h$ follow.

Finally, on the bounded strip, $F_z$ and its derivative in $x$ are uniformly
bounded on the range of $h$.  The boundedness, small-time, and H\"older
estimates for $g_z$ therefore follow from the corresponding estimates for $h$
and the mean-value theorem.  Absolute continuity and the chain rule give
\[
  \partial_t g_z(t)
  =
  \left(\gamma\rho\nu z+(\gamma\nu)^2 h(t,z)\right)\partial_t h(t,z),
  \qquad 0<t\le T.
\]
The derivative estimate for $h$ proves \eqref{eq:g-derivative-t-new}; its
right-hand side is integrable at zero.  This proves the remaining assertions.
\end{proof}


\subsection{Weakly singular quadrature defect for the cumulative resolvent}

\begin{lemma}[Weakly singular quadrature defect for $\mathcal G'$]
\label{lem:quadrature-defect-Gprime}
Fix $\alpha\in(1/2,1)$ and $\gamma,T>0$. Let
\[
  \mathcal G(t):=\frac1\gamma\left(1-E_\alpha(-\gamma t^\alpha)\right),
  \qquad t\ge0.
\]
For $1\le n\le m_\tau$, define
\begin{equation}\label{eq:Dn-def}
\begin{aligned}
  \mathcal D_n^\tau
  &:=
  \int_0^1\left(1-\frac{v^{-\alpha}}{\Gamma(1-\alpha)}\right)
  \frac1\tau \mathcal G'(t_n-v/\tau)\,dv
\\
  &\qquad
  +
  \frac1{\Gamma(1-\alpha)}
  \sum_{j=1}^{n-1}\int_j^{j+1}(j^{-\alpha}-v^{-\alpha})
  \frac1\tau \mathcal G'(t_n-v/\tau)\,dv .
\end{aligned}
\end{equation}
Then
\begin{equation}\label{eq:quadrature-defect-main}
  \mathcal D_n^\tau
  =
  c_1(\alpha)\frac1\tau \mathcal G'(t_n)+R_n^\tau,
\end{equation}
where $c_1(\alpha):=1+\frac{\zeta(\alpha)}{\Gamma(1-\alpha)}$ and
\begin{equation}\label{eq:Rn-bound}
  |R_n^\tau|\le C_T\tau^{-\alpha}n^{-\alpha},
  \qquad 1\le n\le m_\tau.
\end{equation}
Here $C_T$ may depend on $(\alpha,\gamma,T)$ but is independent of
$(n,\tau)$. Integrals whose right endpoint corresponds to the singular value
$\mathcal G'(0)$ are understood as improper integrals.
\end{lemma}

\begin{proof}
Write
\[
  w_n(v):=\frac1\tau \mathcal G'(t_n-v/\tau),
  \qquad 0\le v<n.
\]
Then \eqref{eq:Dn-def} becomes
\begin{align*}
  \mathcal D_n^\tau
  =
  \int_0^1\left(1-\frac{v^{-\alpha}}{\Gamma(1-\alpha)}\right)w_n(v)\,dv
  +
  \frac1{\Gamma(1-\alpha)}
  \sum_{j=1}^{n-1}\int_j^{j+1}(j^{-\alpha}-v^{-\alpha})w_n(v)\,dv.
\end{align*}
We split
\[
  w_n(v)=w_n(0)+\left(w_n(v)-w_n(0)\right).
\]

The contribution of the constant part $w_n(0)$ equals
\begin{align*}
  w_n(0)
  &\left[
    \int_0^1\left(1-\frac{v^{-\alpha}}{\Gamma(1-\alpha)}\right)\,dv
    +
    \frac1{\Gamma(1-\alpha)}
    \sum_{j=1}^{n-1}\int_j^{j+1}(j^{-\alpha}-v^{-\alpha})\,dv
  \right].
\end{align*}
Using
\[
  \int_j^{j+1}(j^{-\alpha}-v^{-\alpha})\,dv
  =j^{-\alpha}
   -\frac{(j+1)^{1-\alpha}-j^{1-\alpha}}{1-\alpha},
\]
a direct telescoping computation gives
\begin{align*}
  &\int_0^1\left(1-\frac{v^{-\alpha}}{\Gamma(1-\alpha)}\right)\,dv
  +
  \frac1{\Gamma(1-\alpha)}
  \sum_{j=1}^{n-1}\int_j^{j+1}(j^{-\alpha}-v^{-\alpha})\,dv\\
  =&
  1+\frac1{\Gamma(1-\alpha)}
  \left(
    \sum_{j=1}^{n-1}j^{-\alpha}
    -\frac{n^{1-\alpha}}{1-\alpha}
  \right).
\end{align*}
The Euler expansion
$\sum_{j=1}^{n-1}j^{-\alpha}
  =\frac{n^{1-\alpha}}{1-\alpha}
   +\zeta(\alpha)+\mathcal O(n^{-\alpha})$
therefore shows that the preceding expression equals
$c_1(\alpha)+\mathcal O(n^{-\alpha})$, uniformly for $n\ge1$ after
enlarging the implied constant. Consequently, the constant part contributes
$c_1(\alpha)\,w_n(0)+\mathcal O\!\left(n^{-\alpha}|w_n(0)|\right)$.
Since
$\mathcal G'(t)=\frac1\gamma f_{\alpha,\gamma}(t)$,
and
$|f_{\alpha,\gamma}(t)|\le C_T t^{\alpha-1}$,
$0<t\le T$,
we obtain
\[
  |w_n(0)|
  =
  \frac1\tau |\mathcal G'(t_n)|
  \le
  C_T\tau^{-1}t_n^{\alpha-1}
  =
  C_T\tau^{-\alpha}n^{\alpha-1}.
\]
Therefore,
\[
  n^{-\alpha}|w_n(0)|
  \le
  C_T\tau^{-\alpha}n^{-1}
  \le
  C_T\tau^{-\alpha}n^{-\alpha},
\]
since $n\ge1$ and $\alpha\in(0,1)$.

It remains to estimate the oscillatory part
\begin{align*}
  \widetilde R_n^\tau
  &:=
  \int_0^1\left(1-\frac{v^{-\alpha}}{\Gamma(1-\alpha)}\right)
  \left(w_n(v)-w_n(0)\right)\,dv
  \\
  &\quad
  +
  \frac1{\Gamma(1-\alpha)}
  \sum_{j=1}^{n-1}\int_j^{j+1}(j^{-\alpha}-v^{-\alpha})
  \left(w_n(v)-w_n(0)\right)\,dv .
\end{align*}

We first record a derivative bound for $\mathcal G'$. Since
\[
  \mathcal G'(t)=\frac1\gamma f_{\alpha,\gamma}(t)
  =t^{\alpha-1}E_{\alpha,\alpha}(-\gamma t^\alpha),
\]
we may write
\[
  \mathcal G'(t)
  =
  \sum_{m=0}^\infty
  \frac{(-\gamma)^m}{\Gamma(\alpha m+\alpha)}\,t^{\alpha m+\alpha-1},
  \qquad t>0.
\]
For every $\delta>0$, this series and its termwise derivative converge uniformly
on $[\delta,T]$, so differentiation term by term is justified there and yields
\[
  (\mathcal G')'(t)
  =
  \sum_{m=0}^\infty
  \frac{(-\gamma)^m(\alpha m+\alpha-1)}{\Gamma(\alpha m+\alpha)}
  t^{\alpha m+\alpha-2}.
\]
The $m=0$ term is of order $t^{\alpha-2}$, while the sum over $m\ge1$ is bounded
by $C_T t^{2\alpha-2}$. Since $2\alpha-2>\alpha-2$ and $t\le T$, we obtain
\[
  |(\mathcal G')'(t)|
  \le C_T t^{\alpha-2},
  \qquad 0<t\le T.
\]

For $n\in\{1,2\}$, use the crude bound
\[
  |w_n(v)-w_n(0)|
  \le |w_n(v)|+|w_n(0)|
  \le C_T\tau^{-\alpha}
  \left((n-v)^{\alpha-1}+n^{\alpha-1}\right),
  \qquad 0\le v<n.
\]
The weights in the definition of $\widetilde R_n^\tau$ are integrable on
these fixed intervals: their only endpoint singularities have exponents
$-\alpha$ and $\alpha-1$, both strictly larger than $-1$, and occur at
distinct endpoints. Thus, after enlarging the constant to cover these two
values of $n$,
\[
  |\widetilde R_n^\tau|
  \le C_T\tau^{-\alpha}n^{-\alpha},
  \qquad n\in\{1,2\}.
\]
Hence, it remains to consider $n\ge3$.

For the integral over $[0,1]$, the mean value theorem and the derivative bound above give
\[
  |w_n(v)-w_n(0)|
  \le
  C_T\tau^{-2}\,v\,(t_n-v/\tau)^{\alpha-2},
  \qquad 0\le v\le1.
\]
Since $t_n-v/\tau\ge (n-1)/\tau\asymp n/\tau$ for $n\ge3$, it follows that
\[
  |w_n(v)-w_n(0)|
  \le
  C_T\tau^{-\alpha}n^{\alpha-2}v.
\]
Because
$1-\frac{v^{-\alpha}}{\Gamma(1-\alpha)}\in L^1(0,1)$,
we obtain
\[
  \left|
    \int_0^1\left(1-\frac{v^{-\alpha}}{\Gamma(1-\alpha)}\right)
    \left(w_n(v)-w_n(0)\right)\,dv
  \right|
  \le
  C_T\tau^{-\alpha}n^{\alpha-2}
  \le
  C_T\tau^{-\alpha}n^{-\alpha},
\]
because $\alpha-2\le -\alpha$ for $\alpha\in(0,1)$.

For the discrete part, write
\[
  I_{j,n}
  :=
  \int_j^{j+1}(j^{-\alpha}-v^{-\alpha})\left(w_n(v)-w_n(0)\right)\,dv,
  \qquad 1\le j\le n-1.
\]
We split the sum into the regions $1\le j\le \lfloor n/2\rfloor$ and
$\lfloor n/2\rfloor+1\le j\le n-1$.

If $1\le j\le \lfloor n/2\rfloor$, then for $v\in[j,j+1]$ we have
$t_n-v/\tau\asymp n/\tau$ and $v\asymp j$. Hence
\[
  |w_n(v)-w_n(0)|
  \le
  C_T\tau^{-2}\,v\,(t_n-v/\tau)^{\alpha-2}
  \le
  C_T\tau^{-\alpha}j\,n^{\alpha-2}.
\]
Using
$\int_j^{j+1}(j^{-\alpha}-v^{-\alpha})\,dv
  \le C_\alpha j^{-\alpha-1}$,
we conclude that
\[
  |I_{j,n}|
  \le
  C_T\tau^{-\alpha}n^{\alpha-2}j^{-\alpha}.
\]
Therefore,
\[
  \sum_{j=1}^{\lfloor n/2\rfloor}|I_{j,n}|
  \le
  C_T\tau^{-\alpha}n^{\alpha-2}
  \sum_{j=1}^{\lfloor n/2\rfloor}j^{-\alpha}
  \le
  C_T\tau^{-\alpha}n^{-1}
  \le
  C_T\tau^{-\alpha}n^{-\alpha}.
\]

If $\lfloor n/2\rfloor+1\le j\le n-1$, then for $v\in[j,j+1]$ we use the crude bound
\[
  |w_n(v)-w_n(0)|
  \le
  |w_n(v)|+|w_n(0)|
  \le
  C_T\tau^{-\alpha}\left((n-v)^{\alpha-1}+n^{\alpha-1}\right).
\]
Moreover, $j\asymp n$ on this range, and the mean-value theorem applied to
$x\mapsto x^{-\alpha}$ gives the pointwise estimate
\[
  0\le j^{-\alpha}-v^{-\alpha}
  \le \alpha j^{-\alpha-1}(v-j)
  \le C_\alpha n^{-\alpha-1},
  \qquad v\in[j,j+1].
\]
Hence
\begin{align*}
  \sum_{j=\lfloor n/2\rfloor+1}^{n-1}|I_{j,n}|
  &\le
  C_T\tau^{-\alpha}n^{-\alpha-1}
  \sum_{j=\lfloor n/2\rfloor+1}^{n-1}
  \int_j^{j+1}\left((n-v)^{\alpha-1}+n^{\alpha-1}\right)\,dv
  \\
  &\le
  C_T\tau^{-\alpha}n^{-\alpha-1}
  \left(
    \int_{n/2}^{n}(n-v)^{\alpha-1}\,dv
    +n^{\alpha-1}\sum_{j=\lfloor n/2\rfloor+1}^{n-1}1
  \right)
  \\
  &\le
  C_T\tau^{-\alpha}n^{-1}
  \le
  C_T\tau^{-\alpha}n^{-\alpha}.
\end{align*}

Combining the bounds for the integral part and the two discrete ranges, we obtain
\[
  |\widetilde R_n^\tau|
  \le
  C_T\tau^{-\alpha}n^{-\alpha},
  \qquad 1\le n\le m_\tau.
\]

Finally, combining the constant contribution and the oscillatory contribution,
we get
\[
  \mathcal D_n^\tau
  =
  c_1(\alpha)\,w_n(0)+R_n^\tau
  =
  c_1(\alpha)\frac1\tau \mathcal G'(t_n)+R_n^\tau,
\]
where
\[
  |R_n^\tau|
  \le
  C_T\tau^{-\alpha}n^{-\alpha}.
\]
This proves \eqref{eq:quadrature-defect-main}--\eqref{eq:Rn-bound}.
\end{proof}


\subsection{Proof of Proposition~\ref{prop:cum-resolvent-uniform-comparison}}\label{app:proof-cum-resolvent-uniform-comparison}

\begin{proof}
We split the proof into five steps.

\medskip
\noindent
\textit{Step 1: error equation.}
Set $\Sigma_n^\tau:=\sum_{j=1}^n\psi_j^\tau$. The renewal identity
$\psi^\tau=\varphi^\tau+\varphi^\tau*\psi^\tau$, with
$\varphi^\tau=a_\tau\varphi$, gives
\[
  \Sigma_n^\tau
  =a_\tau\mathsf C_n
   +a_\tau\sum_{k=1}^{n-1}\varphi_k \Sigma_{n-k}^\tau,
  \qquad
  \mathsf C_n:=\sum_{k=1}^n\varphi_k.
\]
Multiplying by $(1-a_\tau)/(\gamma a_\tau)$ and using
$\mathcal G_0^\tau=0$ and $(1-a_\tau)/\gamma=\tau^{-\alpha}$, we obtain,
for every $n\ge1$,
\begin{equation}\label{eq:Gtau-renewal-proof}
  \mathcal G_n^\tau
  =
  a_\tau\sum_{k=1}^n\varphi_k\,\mathcal G_{n-k}^\tau
  +\tau^{-\alpha}\mathsf C_n.
\end{equation}
On the continuous side, since
$\widehat{\mathcal G}(z)=\frac{1}{z(\gamma+z^\alpha)}$,
we have
$z^\alpha\widehat{\mathcal G}(z)=\frac1z-\gamma\widehat{\mathcal G}(z)$,
and hence $\mathcal G$ solves
\begin{equation}\label{eq:G-cont-fde-proof}
  D^\alpha\mathcal G(t)=1-\gamma\mathcal G(t),
  \qquad
  \mathcal G(0)=0.
\end{equation}
Define
\[
  e_n^\tau:=\mathcal G_n^\tau-\mathcal G(t_n),
  \qquad 0\le n\le m_\tau.
\]
Then
\begin{equation}\label{eq:error-renewal-proof}
  e_n^\tau
  =
  a_\tau\sum_{k=1}^n\varphi_k\,e_{n-k}^\tau
  -L_n^\tau,
\end{equation}
where
\begin{equation}\label{eq:local-defect-def-proof}
  L_n^\tau
  :=
  \mathcal G(t_n)
  -
  a_\tau\sum_{k=1}^n\varphi_k\,\mathcal G(t_{n-k})
  -
  \tau^{-\alpha}\mathsf C_n.
\end{equation}

\medskip
\noindent
\textit{Step 2: defect decomposition.}
Introduce the exact tail sequence
$\bar\varphi_j:=\sum_{m=j+1}^\infty \varphi_m$,
$j\ge0$.
For the kernel from Definition~\ref{def:params},
\begin{equation}\label{eq:exact-tail}
  \bar\varphi_0=1,
  \qquad
  \bar\varphi_j=\frac{j^{-\alpha}}{\Gamma(1-\alpha)},
  \qquad j\ge1.
\end{equation}
Since $\varphi_k=\bar\varphi_{k-1}-\bar\varphi_k$ and $\mathcal G(0)=0$,
summation by parts yields
\begin{equation}\label{eq:abel-G-proof}
  \mathcal G(t_n)-\sum_{k=1}^n\varphi_k\,\mathcal G(t_{n-k})
  =
  \sum_{j=0}^{n-1}\bar\varphi_j
  \left(\mathcal G(t_{n-j})-\mathcal G(t_{n-j-1})\right).
\end{equation}
By \eqref{eq:abel-G-proof} and \eqref{eq:exact-tail},
\begin{align}
  &\mathcal G(t_n)-\sum_{k=1}^n\varphi_k\,\mathcal G(t_{n-k})
  \nonumber\\
  &=
  \int_0^1 \frac1\tau \mathcal G'(t_n-v/\tau)\,dv
  +
  \frac1{\Gamma(1-\alpha)}
  \sum_{j=1}^{n-1}
  j^{-\alpha}\int_j^{j+1}\frac1\tau\mathcal G'(t_n-v/\tau)\,dv.
  \label{eq:discrete-caputo-proof}
\end{align}
On the other hand, from \eqref{eq:G-cont-fde-proof},
\begin{equation}\label{eq:continuous-caputo-proof}
  \tau^{-\alpha}\left(1-\gamma\mathcal G(t_n)\right)
  =
  \frac1{\Gamma(1-\alpha)}
  \int_0^n v^{-\alpha}\frac1\tau\mathcal G'(t_n-v/\tau)\,dv.
\end{equation}
Using $a_\tau=1-\gamma\tau^{-\alpha}$ and
$\mathsf C_n=1-\bar\varphi_n$, we may
therefore rewrite \eqref{eq:local-defect-def-proof} as
\begin{equation}\label{eq:Ln-split-proof}
  L_n^\tau
  =
  \mathcal D_n^\tau
  +
  \gamma\tau^{-\alpha}
  \left(
    \sum_{k=1}^n\varphi_k\,\mathcal G(t_{n-k})-\mathcal G(t_n)
  \right)
  +\tau^{-\alpha}\bar\varphi_n,
\end{equation}
where $\mathcal D_n^\tau$ is the quadrature defect from
Lemma~\ref{lem:quadrature-defect-Gprime}.

\medskip
\noindent
\textit{Step 3: bound the main defect.}
By Lemma~\ref{lem:quadrature-defect-Gprime},
$\mathcal D_n^\tau
  =
  c_1(\alpha)\,\frac1\tau \mathcal G'(t_n)
  +R_n^\tau$,
with
$|R_n^\tau|\le C_T\tau^{-\alpha}n^{-\alpha}$.
Since
$|\mathcal G'(t)|=\frac1\gamma|f_{\alpha,\gamma}(t)|
\le C_T t^{\alpha-1}$,
we obtain
\begin{equation}\label{eq:Ln-main-bound}
  |\mathcal D_n^\tau|
  \le
  C_T|c_1(\alpha)|\,\tau^{-\alpha}n^{\alpha-1}
  +C_T\tau^{-\alpha}n^{-\alpha}.
\end{equation}

\medskip
\noindent
\textit{Step 4: bound the lower-order terms.}
By \eqref{eq:abel-G-proof}, \eqref{eq:exact-tail}, and the bound
\[
  |\mathcal G(t_{m})-\mathcal G(t_{m-1})|
  \le
  \int_{t_{m-1}}^{t_m}|\mathcal G'(s)|\,ds
  \le C_T\tau^{-\alpha}m^{\alpha-1},
  \qquad 1\le m\le m_\tau,
\]
we obtain
\begin{align*}
  \left|
    \mathcal G(t_n)-\sum_{k=1}^n\varphi_k\,\mathcal G(t_{n-k})
  \right|
  &\le
  C_T\tau^{-\alpha}
  \left(
    n^{\alpha-1}
    +
    \sum_{j=1}^{n-1}j^{-\alpha}(n-j)^{\alpha-1}
  \right).
\end{align*}
Splitting the discrete convolution at $j=\lfloor n/2\rfloor$ gives
\[
  \sum_{j=1}^{n-1}j^{-\alpha}(n-j)^{\alpha-1}
  \le
  Cn^{\alpha-1}\sum_{j\le n/2}j^{-\alpha}
  +
  Cn^{-\alpha}\sum_{r\le n/2}r^{\alpha-1}
  \le C,
\]
and hence,
\[
  \left|
    \mathcal G(t_n)-\sum_{k=1}^n\varphi_k\,\mathcal G(t_{n-k})
  \right|
  \le C_T\tau^{-\alpha}.
\]
Since $\bar\varphi_n=n^{-\alpha}/\Gamma(1-\alpha)$ for $n\ge1$, it follows that
\begin{equation}\label{eq:Ln-lower-order}
  \left|
    \gamma\tau^{-\alpha}
    \left(
      \sum_{k=1}^n\varphi_k\,\mathcal G(t_{n-k})-\mathcal G(t_n)
    \right)
    +\tau^{-\alpha}\bar\varphi_n
  \right|
  \le
  C_T\left(\tau^{-2\alpha}+\tau^{-\alpha}n^{-\alpha}\right).
\end{equation}
Combining \eqref{eq:Ln-split-proof}, \eqref{eq:Ln-main-bound}, and
\eqref{eq:Ln-lower-order}, and observing that $n\le m_\tau\le T\tau$
implies
$\tau^{-2\alpha}
  \le T^\alpha\tau^{-\alpha}n^{-\alpha}$,
we conclude that
\begin{equation}\label{eq:Ln-final-bound}
  |L_n^\tau|
  \le
  C_T\tau^{-\alpha}n^{-\alpha}
  +
  C_T|c_1(\alpha)|\,\tau^{-\alpha}n^{\alpha-1},
  \qquad 1\le n\le m_\tau.
\end{equation}

\medskip
\noindent
\textit{Step 5: renewal amplification.}
From \eqref{eq:error-renewal-proof}, inversion by the renewal kernel gives
\begin{equation}\label{eq:error-inversion-proof}
  e_n^\tau
  =
  -L_n^\tau-\sum_{j=1}^{n-1}\psi_{n-j}^\tau L_j^\tau.
\end{equation}
By Lemma~\ref{lem:psi-pointwise},
$\psi_m^\tau\le C_T m^{\alpha-1}$
for any $1\le m\le m_\tau$.
Hence, using \eqref{eq:Ln-final-bound},
\begin{align*}
  \sum_{j=1}^{n-1}\psi_{n-j}^\tau |L_j^\tau|
  \le
  C_T\tau^{-\alpha}
  \sum_{j=1}^{n-1}(n-j)^{\alpha-1}j^{-\alpha}
  +
  C_T|c_1(\alpha)|\,\tau^{-\alpha}
  \sum_{j=1}^{n-1}(n-j)^{\alpha-1}j^{\alpha-1}.
\end{align*}
The first sum is uniformly bounded in $n$, while the second satisfies
\[
  \sum_{j=1}^{n-1}(n-j)^{\alpha-1}j^{\alpha-1}
  \le C n^{2\alpha-1}.
\]
Since $n\le m_\tau\le T\tau$,
$\tau^{-\alpha}n^{2\alpha-1}\le C_T\tau^{-(1-\alpha)}$.
Therefore,
\[
  \sum_{j=1}^{n-1}\psi_{n-j}^\tau |L_j^\tau|
  \le
  C_T\tau^{-\alpha}
  +
  C_T|c_1(\alpha)|\,\tau^{-(1-\alpha)}.
\]
For the nonconvolved term, \eqref{eq:Ln-final-bound},
$n^{\alpha-1}\le1$, and $\tau^{-\alpha}\le\tau^{-(1-\alpha)}$
for $\alpha>1/2$ and $\tau\ge1$ give
\[
  |L_n^\tau|
  \le
  C_T\tau^{-\alpha}
  +C_T|c_1(\alpha)|\tau^{-(1-\alpha)}.
\]
Combining this with \eqref{eq:error-inversion-proof}, and recalling that
$e_0^\tau=0$, we obtain
\[
  |e_n^\tau|
  \le
  C_T\tau^{-\alpha}
  +
  C_{\mathrm{res}}(\alpha)\tau^{-(1-\alpha)},
  \qquad 0\le n\le m_\tau,
\]
where one may take $C_{\mathrm{res}}(\alpha)=C_T|c_1(\alpha)|$.
Taking the supremum over $0\le n\le m_\tau$ proves
\eqref{eq:cum-resolvent-uniform-comparison}.
\end{proof}


\subsection{Proof of Proposition~\ref{prop:resolvent-transfer-from-cumulative}}\label{app:proof-resolvent-transfer-from-cumulative}

\begin{proof}
Define
\[
  (\mathcal Q_\tau g)_n
  :=
  \frac{1}{\gamma\tau}\sum_{j=1}^{n-1}\mathsf{S}_{n-j}^\tau g(t_j),
  \qquad
  (\mathcal Q g)(t_n)
  :=
  \frac1\gamma\int_0^{t_n}f_{\alpha,\gamma}(t_n-s)g(s)\,ds.
\]
Recall from the definition of $\mathcal G_n^\tau$ that
$\mathcal G_n^\tau-\mathcal G_{n-1}^\tau
  =
  \frac{1}{\gamma\tau}\mathsf{S}_n^\tau$,
$n\ge1$,
and
$\mathcal G'(t)=\frac1\gamma f_{\alpha,\gamma}(t)$, $t>0$.
In particular,
$|\mathcal G'(t)|\le C_T t^{\alpha-1}$,
$0<t\le T$,
and therefore,
\[
  |\mathcal G(t)-\mathcal G(s)|\le C_T|t-s|^\alpha,
  \qquad 0\le s,t\le T.
\]

Using summation by parts, for $1\le n\le m_\tau$,
\begin{equation}\label{eq:Qtau-abel}
  (\mathcal Q_\tau g)_n
  =
  g(0)\mathcal G_{n-1}^\tau
  +\sum_{j=0}^{n-2}\mathcal G_{n-j-1}^\tau
   \int_{t_j}^{t_{j+1}} g'(s)\,ds.
\end{equation}
On the continuous side, integration by parts yields
\begin{equation}\label{eq:Q-abel}
  (\mathcal Q g)(t_n)
  =
  g(0)\mathcal G(t_n)
  +\int_0^{t_n}\mathcal G(t_n-s)g'(s)\,ds.
\end{equation}
Subtracting \eqref{eq:Q-abel} from \eqref{eq:Qtau-abel}, we obtain
\begin{align*}
  &(\mathcal Q_\tau g)_n-(\mathcal Q g)(t_n)
    =
  g(0)\left(\mathcal G_{n-1}^\tau-\mathcal G(t_n)\right)
  \\
  &
  \qquad\qquad+
  \sum_{j=0}^{n-2}
  \int_{t_j}^{t_{j+1}}
  \left(\mathcal G_{n-j-1}^\tau-\mathcal G(t_n-s)\right)g'(s)\,ds
  -
  \int_{t_{n-1}}^{t_n}\mathcal G(t_n-s)g'(s)\,ds.
\end{align*}
Hence,
\begin{align}
  &\left|(\mathcal Q_\tau g)_n-(\mathcal Q g)(t_n)\right|
  \le
  |g(0)|\,|\mathcal G_{n-1}^\tau-\mathcal G(t_n)|
  \nonumber
  \\
  &\quad\quad
  +
  \sum_{j=0}^{n-2}
  \int_{t_j}^{t_{j+1}}
  \left|\mathcal G_{n-j-1}^\tau-\mathcal G(t_n-s)\right|
  |g'(s)|\,ds
  +
  \int_{t_{n-1}}^{t_n}|\mathcal G(t_n-s)|\,|g'(s)|\,ds.\label{3:terms:2}
\end{align}
For convenience, write
\[
  \Delta_\tau
  :=C_T\tau^{-\alpha}
    +C_{\mathrm{res}}(\alpha)\tau^{-(1-\alpha)},
\]
where $C_T$ may be enlarged from line to line below.
For the first term on the right hand side in \eqref{3:terms:2}, we insert and subtract $\mathcal G(t_{n-1})$:
\[
  |\mathcal G_{n-1}^\tau-\mathcal G(t_n)|
  \le
  |\mathcal G_{n-1}^\tau-\mathcal G(t_{n-1})|
  +
  |\mathcal G(t_{n-1})-\mathcal G(t_n)|.
\]
By Proposition~\ref{prop:cum-resolvent-uniform-comparison} and the
$\alpha$-H\"older continuity of $\mathcal G$, this gives
$|\mathcal G_{n-1}^\tau-\mathcal G(t_n)|
  \le \Delta_\tau$.
We split
\[
  \left|\mathcal G_{n-j-1}^\tau-\mathcal G(t_n-s)\right|
  \le
  \left|\mathcal G_{n-j-1}^\tau-\mathcal G(t_{n-j-1})\right|
  +
  \left|\mathcal G(t_{n-j-1})-\mathcal G(t_n-s)\right|.
\]
By Proposition~\ref{prop:cum-resolvent-uniform-comparison}, the first term on
the right-hand side of this decomposition is bounded by
$C_T\tau^{-\alpha}
  +C_{\mathrm{res}}(\alpha)\tau^{-(1-\alpha)}$.
Since $\mathcal G$ is $\alpha$-H\"older continuous on $[0,T]$, the second term
on the right-hand side of the decomposition is bounded by
$C_T\tau^{-\alpha}$ uniformly for $s\in[t_j,t_{j+1}]$. Therefore,
\[
  \left|\mathcal G_{n-j-1}^\tau-\mathcal G(t_n-s)\right|
  \le \Delta_\tau.
\]
Consequently, the sum of cell integrals in \eqref{3:terms:2} satisfies
\begin{align*}
  \sum_{j=0}^{n-2}\int_{t_j}^{t_{j+1}}
  \left|\mathcal G_{n-j-1}^\tau-\mathcal G(t_n-s)\right|
  |g'(s)|\,ds
\le
  \Delta_\tau\int_0^{t_{n-1}}|g'(s)|\,ds
  \le
  \frac{C_gT^\alpha}{\alpha}\,\Delta_\tau,
\end{align*}
where the last inequality follows from \eqref{eq:g-assump-2}.
Also $\mathcal G$ is bounded on $[0,T]$, and thus the last integral in \eqref{3:terms:2} is bounded by
$C_T\int_{t_{n-1}}^{t_n}|g'(s)|\,ds$.
Using \eqref{eq:g-assump-2}, we obtain
\[
  \int_{t_{n-1}}^{t_n}|g'(s)|\,ds
  \le
  C_g\int_{t_{n-1}}^{t_n}s^{\alpha-1}\,ds.
\]
If $n=1$, then
\[
  \int_{0}^{1/\tau}s^{\alpha-1}\,ds
  =\frac1\alpha\tau^{-\alpha}.
\]
If $n\ge2$, then $s\ge t_1=1/\tau$ on $[t_{n-1},t_n]$, so
\[
  \int_{t_{n-1}}^{t_n}s^{\alpha-1}\,ds
  \le
  \tau^{-1}(t_{n-1})^{\alpha-1}
  \le \tau^{-1}(\tau^{-1})^{\alpha-1}
  = \tau^{-\alpha},
\]
since $\alpha\in(0,1)$. Therefore, uniformly for $1\le n\le m_\tau$,
\[
  \int_{t_{n-1}}^{t_n}|g'(s)|\,ds
  \le C_T C_g\,\tau^{-\alpha}.
\]
Combining the preceding bounds gives
\[
  \left|(\mathcal Q_\tau g)_n-(\mathcal Q g)(t_n)\right|
  \le
  C_T' \left(|g(0)|+C_g\right)\tau^{-\alpha}
  +
  C_{\mathrm{op}}(\alpha)\left(|g(0)|+C_g\right)\tau^{-(1-\alpha)},
\]
uniformly in $1\le n\le m_\tau$.
For example, after enlarging $C_T'$, one may take
$C_{\mathrm{op}}(\alpha)
  :=\left(1+\frac{T^\alpha}{\alpha}\right)
    C_{\mathrm{res}}(\alpha)$.
This proves \eqref{eq:resolvent-operator-compare-main}.
\end{proof}

\subsection{Proof of Proposition~\ref{prop:Volterra-rate-fixed}}\label{app:proof-Volterra-rate-fixed}

We first record the discrete Gr\"onwall estimate used below.

\begin{lemma}[Discrete weakly singular Gr\"onwall on the $\tau$-grid]
\label{lem:discrete-weakly-singular-gronwall}
Fix $\alpha\in(0,1)$, $T>0$, and $B>0$. Suppose that a nonnegative sequence
$(a_n)_{n=0}^{m_\tau}$ satisfies
$a_0\le A$ and
\begin{equation}\label{eq:discrete-weakly-singular-gronwall-assump}
  a_n\le A+B\tau^{-\alpha}\sum_{j=1}^{n-1}(n-j)^{\alpha-1}a_j,
  \qquad 1\le n\le m_\tau,
\end{equation}
for some $A\ge0$. Then
\begin{equation}\label{eq:discrete-weakly-singular-gronwall-concl}
  a_n\le A E_\alpha(B\Gamma(\alpha)t_n^\alpha),
  \qquad 0\le n\le m_\tau.
\end{equation}
In particular,
$\sup_{0\le n\le m_\tau}a_n\le C_{\alpha,B,T}A$.
\end{lemma}

\begin{proof}
We prove the claim directly by induction on $n$.  The case $n=0$ follows
from $a_0\le A$.  Suppose the claimed Mittag--Leffler bound holds for all
$j<n$.  Since
$\tau^{-\alpha}(n-j)^{\alpha-1}
  =
  \tau^{-1}(t_n-t_j)^{\alpha-1}$,
both factors on the left below are no larger than their values at
$s\in[t_j,t_{j+1}]$.  Consequently, for every integer $r\ge0$,
\begin{align*}
  \tau^{-\alpha}\sum_{j=1}^{n-1}(n-j)^{\alpha-1}t_j^{r\alpha}
  &=
  \tau^{-1}\sum_{j=1}^{n-1}(t_n-t_j)^{\alpha-1}t_j^{r\alpha}
  \\
  &\le
  \sum_{j=1}^{n-1}\int_{t_j}^{t_{j+1}}(t_n-s)^{\alpha-1}s^{r\alpha}\,ds
  \\
  &\le
  \int_0^{t_n}(t_n-s)^{\alpha-1}s^{r\alpha}\,ds
 =
  \frac{\Gamma(\alpha)\Gamma(r\alpha+1)}{\Gamma((r+1)\alpha+1)}
  t_n^{(r+1)\alpha}.
\end{align*}
Using the induction hypothesis in
\eqref{eq:discrete-weakly-singular-gronwall-assump}, Tonelli's theorem for
the nonnegative series, and the preceding beta-integral estimate, we obtain
\begin{align*}
 a_n
 &\le A+B\tau^{-\alpha}\sum_{j=1}^{n-1}(n-j)^{\alpha-1}
  A\sum_{r=0}^{\infty}
  \frac{(B\Gamma(\alpha)t_j^\alpha)^r}{\Gamma(r\alpha+1)}
 \\
 &\le A+A\sum_{r=0}^{\infty}
  \frac{(B\Gamma(\alpha)t_n^\alpha)^{r+1}}
       {\Gamma((r+1)\alpha+1)}
 =A E_\alpha(B\Gamma(\alpha)t_n^\alpha).
\end{align*}
This closes the induction.  The uniform bound follows from $t_n\le T$.
\end{proof}

\begin{proof}[Proof of Proposition~\ref{prop:Volterra-rate-fixed}]
Set
$N_\tau:=m_\tau\wedge\sigma_\tau$.
All estimates below are first proved on the stopped grid
$0\le n\le N_\tau$. Recall from
Lemma~\ref{lem:renewal-inversion} that, for
$1\le n\le N_\tau$,
\begin{equation}\label{eq:Htau-resolvent-discrete-newproof}
\begin{aligned}
  H_n^\tau(z)
  &=
  \tau^{-\alpha}F_z(H_n^\tau(z))
  +\widehat E_n^\tau(z)
\\
  &\quad
  +
  \frac{a_\tau}{1-a_\tau}\tau^{-\alpha-1}
  \sum_{j=0}^{n-1}\mathsf{S}_{n-j}^\tau F_z(H_j^\tau(z))
  +
  \sum_{j=0}^{n-1}\psi_{n-j}^\tau\widehat E_j^\tau(z),
\end{aligned}
\end{equation}
where
\[
  \sup_{0\le n\le m_\tau\wedge\sigma_\tau}
  \sup_{z\in\Gamma_{\eta,R}}
  |\widehat E_n^\tau(z)|
  \le C_{\eta,R,T}\tau^{-2\alpha}.
\]

On the continuous side, Lemma~\ref{lem:continuous-resolvent} gives
\[
  h(t_n,z)
  =
  \frac1\gamma\int_0^{t_n}
    f_{\alpha,\gamma}(t_n-s)F_z(h(s,z))\,ds.
\]

Set
\[
  e_n^\tau(z):=H_n^\tau(z)-h(t_n,z).
\]
The two $j=0$ terms omitted from the discrete comparison operator are
collected in
\[
  \mathcal B_n^\tau(z)
  :=\tau^{-\alpha}\psi_n^\tau F_z(H_0^\tau(z))
    +\psi_n^\tau\widehat E_0^\tau(z),
\]
where we used the definition of $\mathsf S_n^\tau$.  The bounds
$H_0^\tau=\mathcal O(\tau^{-\alpha})$,
$\widehat E_0^\tau=\mathcal O(\tau^{-2\alpha})$, and
$\psi_n^\tau\le C n^{\alpha-1}\le C$ imply
\begin{equation}\label{eq:j0-boundary-volterra}
  \sup_{1\le n\le m_\tau}\sup_{z\in\Gamma_{\eta,R}}
  |\mathcal B_n^\tau(z)|\le C_{\eta,R,T}\tau^{-\alpha}.
\end{equation}
Subtracting \eqref{eq:continuous-resolvent} from
\eqref{eq:Htau-resolvent-discrete-newproof}, isolating this boundary term,
and inserting the intermediate discrete resolvent operator applied to
$F_z(h(\cdot,z))$, we obtain
\begin{align*}
  |e_n^\tau(z)|
  &\le
  \tau^{-\alpha}|F_z(H_n^\tau(z))|
  +|\widehat E_n^\tau(z)|
  +\sum_{j=1}^{n-1}\psi_{n-j}^\tau |\widehat E_j^\tau(z)|
  +|\mathcal B_n^\tau(z)|
  \\
  &\quad
  +
  \left|
    \frac{a_\tau}{1-a_\tau}\tau^{-\alpha-1}
    \sum_{j=1}^{n-1}\mathsf{S}_{n-j}^\tau F_z(H_j^\tau(z))
    -
    \frac{1}{\gamma\tau}\sum_{j=1}^{n-1}\mathsf{S}_{n-j}^\tau F_z(h(t_j,z))
  \right|
  \\
  &\quad\quad
  +
  \left|
    \frac{1}{\gamma\tau}\sum_{j=1}^{n-1}\mathsf{S}_{n-j}^\tau F_z(h(t_j,z))
    -
    \frac1\gamma\int_0^{t_n}
      f_{\alpha,\gamma}(t_n-s)F_z(h(s,z))\,ds
  \right|.
\end{align*}

By Proposition~\ref{prop:riccati-regularity}, the functions
$g_z(t):=F_z(h(t,z))$ satisfy, uniformly for
$z\in\Gamma_{\eta,R}$,
$|\partial_t g_z(t)|\le C_{\eta,R,T}t^{\alpha-1}$,
while $g_z(0)=\frac12(z^2-z)$ is uniformly bounded.  Applying
Proposition~\ref{prop:resolvent-transfer-from-cumulative} with $g=g_z$ shows
that the last term is bounded by
\[
  C_{\eta,R,T}
  \left(\tau^{-\alpha}
  +C_{\mathrm{op}}(\alpha)\tau^{-(1-\alpha)}\right)
\]
with $C_{\mathrm{op}}(\alpha)<\infty$. The explicit local term
$\tau^{-\alpha}|F_z(H_n^\tau(z))|$ is of order $\tau^{-\alpha}$ because
$H_n^\tau(z)$ remains uniformly bounded on the stopped grid by
Lemma~\ref{prop:smallness-new33}. The same stopped boundedness implies
that $F_z$ is uniformly Lipschitz on the relevant range:
\[
  |F_z(x)-F_z(y)|\le C_{\eta,R,T}|x-y|.
\]

Using the exact identity
$\frac{a_\tau}{1-a_\tau}\tau^{-\alpha-1}
  =\frac{a_\tau}{\gamma\tau}
  =\frac{1}{\gamma\tau}-\tau^{-1-\alpha}$
and
$\sum_{m\ge1}\mathsf{S}_m^\tau=\tau$,
shows that the coefficient mismatch is bounded by
$C_{\eta,R,T}\tau^{-\alpha}$. Hence,
\[
  |e_n^\tau(z)|
  \le
  C_{\eta,R,T}
  \left(\tau^{-\alpha}
  +C_{\mathrm{op}}(\alpha)\tau^{-(1-\alpha)}\right)
  +
  \frac{C_{\eta,R,T}}{\tau}
  \sum_{j=1}^{n-1}\mathsf{S}_{n-j}^\tau |e_j^\tau(z)|
  +
  \sum_{j=1}^{n-1}\psi_{n-j}^\tau |\widehat E_j^\tau(z)|.
\]

Since
\[
  \sup_{1\leq j\leq n-1} |\widehat E_j^\tau(z)|\le C_{\eta,R,T}\tau^{-2\alpha},
  \qquad
  \sum_{m\ge1}\psi_m^\tau=\frac{a_\tau}{1-a_\tau}\lesssim \tau^\alpha,
\]
we have
\[
  \sum_{j=1}^{n-1}\psi_{n-j}^\tau |\widehat E_j^\tau(z)|
  \le
  C_{\eta,R,T}\tau^{-2\alpha}\sum_{m\ge1}\psi_m^\tau
  \le
  C_{\eta,R,T}\tau^{-\alpha}.
\]
Hence
\[
  |e_n^\tau(z)|
  \le
  C_{\eta,R,T}
  \left(\tau^{-\alpha}
  +C_{\mathrm{op}}(\alpha)\tau^{-(1-\alpha)}\right)
  +
  \frac{C_{\eta,R,T}}{\tau}
  \sum_{j=1}^{n-1}\mathsf{S}_{n-j}^\tau |e_j^\tau(z)|.
\]

Finally, using the pointwise bound from Lemma~\ref{lem:psi-pointwise},
\[
  \frac{1}{\tau}\mathsf{S}_{n-j}^\tau
  =
  \frac{1-a_\tau}{a_\tau}\psi_{n-j}^\tau
  \le
  C_T\tau^{-\alpha}(n-j)^{\alpha-1}
  =
  \frac{C_T}{\tau^\alpha (n-j)^{1-\alpha}},
\]
and setting
$b_n:=\sup_{z\in\Gamma_{\eta,R}}|e_n^\tau(z)|$,
we obtain from the preceding inequality that
\[
  b_n
  \le
  C_{\eta,R,T}\tau^{-\alpha}
  +
  C_{\eta,R,T}C_{\mathrm{op}}(\alpha)\tau^{-(1-\alpha)}
  +
  C_T\tau^{-\alpha}\sum_{j=1}^{n-1}(n-j)^{\alpha-1}b_j,
  \qquad 1\le n\le N_\tau.
\]
Moreover, $b_0\le C_{\eta,R,T}\tau^{-\alpha}$. The proof of
Lemma~\ref{lem:discrete-weakly-singular-gronwall} applies verbatim with terminal
index $N_\tau$. Applying it with
$A=C_{\eta,R,T}\tau^{-\alpha}
  +
  C_{\eta,R,T}C_{\mathrm{op}}(\alpha)\tau^{-(1-\alpha)}$
and $B=C_T$, we obtain
\[
  \sup_{0\le n\le m_\tau\wedge\sigma_\tau}
  \sup_{z\in\Gamma_{\eta,R}}|e_n^\tau(z)|
  \le
  C_{\eta,R,T}\tau^{-\alpha}
  +
  C_{\mathrm{Ric}}(\alpha)\tau^{-(1-\alpha)}.
\]
where one may take
$C_{\mathrm{Ric}}(\alpha)
  :=C_{\alpha,\eta,R,T}C_{\mathrm{op}}(\alpha)$, 
after enlarging fixed constants if necessary.  
For fixed $\alpha$, the right-hand side tends to zero.  If
$\sigma_\tau\le m_\tau$, then, for all sufficiently large $\tau$, the preceding
estimate gives
\[
 \sup_{z\in\Gamma_{\eta,R}}|H_{\sigma_\tau}^\tau(z)|
 \le
 \sup_{z\in\Gamma_{\eta,R}}|h(t_{\sigma_\tau},z)|
 +\sup_{z\in\Gamma_{\eta,R}}|e_{\sigma_\tau}^\tau(z)|
 \le M_{\eta,R,T}+1<L,
\]
contradicting the definition of $\sigma_\tau$.  Thus
$\sigma_\tau=m_\tau+1$ for all sufficiently large $\tau$.  The stopped
estimate is therefore global and proves \eqref{eq:Volterra-rate-fixed}; it
also completes the continuation steps used in
Lemma~\ref{prop:smallness-new33} and
Proposition~\ref{prop:reduced-recursion-new33}.
\end{proof}

\subsection{Proof of Proposition~\ref{prop:baseline-riemann}}\label{app:baseline-riemann}

\begin{proof}
We use the explicit form of $\hat\mu_\tau(m)$ from Definition~\ref{def:params}. Write
$\overline\varphi(n):=\sum_{k=n+1}^{\infty}\varphi_k$, $n\ge0$.
For the kernel chosen in Definition~\ref{def:params}, a direct telescoping computation gives
\begin{equation}\label{eq:tail-phi-explicit-final}
  \overline\varphi(0)=1,
  \qquad
  \overline\varphi(n)=\frac{1}{\Gamma(1-\alpha)}n^{-\alpha},
  \qquad n\ge1.
\end{equation}
Moreover,
\[
  1-\sum_{s=1}^{m-1}\varphi_s^\tau
  =1-a_\tau\sum_{s=1}^{m-1}\varphi_s
  =\varepsilon_\tau+a_\tau\overline\varphi(m-1).
\]
Substituting this into the definition of $\hat\mu_\tau(m)$ yields
\begin{align}
  \hat\mu_\tau(m)
  &=\mu_\tau+\xi_0\mu_\tau\left(
      \frac{\varepsilon_\tau+a_\tau\overline\varphi(m-1)}{\varepsilon_\tau}
      -a_\tau\sum_{s=1}^{m-1}\varphi_s
    \right)
  \nonumber\\
  &=\mu_\tau
    +\xi_0\mu_\tau\varepsilon_\tau
    +\xi_0\mu_\tau a_\tau\left(\varepsilon_\tau^{-1}+1\right)\overline\varphi(m-1).
  \label{eq:muhat-decomp-final}
\end{align}
Multiplying by $\varepsilon_\tau/\mu$ and using $\mu_\tau=\mu\tau^{\alpha-1}$ and $\varepsilon_\tau=\gamma\tau^{-\alpha}$, we obtain
\begin{equation}\label{eq:muhat-scaled-decomp-final}
  \frac{\varepsilon_\tau}{\mu}\hat\mu_\tau(m)
  = \frac{\gamma}{\tau}
    + \frac{\xi_0\gamma^2}{\tau^{\alpha+1}}
    + \xi_0 a_\tau\tau^{\alpha-1}\overline\varphi(m-1)
    + \frac{\xi_0\gamma a_\tau}{\tau}\overline\varphi(m-1).
\end{equation}
Hence,
\begin{align}
  &\frac{\varepsilon_\tau}{\mu}
    \sum_{m=1}^{m_\tau}\hat\mu_\tau(m)
    f\!\left(T-\frac{m}{\tau}\right)
  \nonumber\\
  &=
    \frac{\gamma}{\tau}
    \sum_{m=1}^{m_\tau}f\!\left(T-\frac{m}{\tau}\right)
    +
    \xi_0 a_\tau\tau^{\alpha-1}
    \sum_{m=1}^{m_\tau}\overline\varphi(m-1)
    f\!\left(T-\frac{m}{\tau}\right)
    +R_\tau(f),
  \label{eq:baseline-split-final}
\end{align}
where, by \eqref{eq:tail-phi-explicit-final},
\begin{align}
  |R_\tau(f)|
  &\le C_T\|f\|_{\infty}
  \left(
    \tau^{-\alpha}
    +
    \frac1\tau\sum_{m=2}^{m_\tau}(m-1)^{-\alpha}
  \right)
  \nonumber\\
  &\le C_T\|f\|_{\infty}
  \left(
    \tau^{-\alpha}
    +\tau^{-1}m_\tau^{1-\alpha}
  \right)
  \le C_T\|f\|_{\infty}\tau^{-\alpha}.
  \label{eq:Rtau-bound-final}
\end{align}

We now treat the two main terms in \eqref{eq:baseline-split-final}.

\medskip
\noindent
\emph{Step 1: regular Riemann sum.}
Set
$x_m:=T-\frac{m}{\tau}$, $m=1,\dots,m_\tau$.
Then
\[
\frac{1}{\tau}\sum_{m=1}^{m_\tau}f\!\left(T-\frac{m}{\tau}\right)
=
\sum_{m=1}^{m_\tau}\int_{x_m}^{x_m+1/\tau} f(x_m)\,ds.
\]
Therefore,
\begin{align*}
&\left|
\frac{1}{\tau}\sum_{m=1}^{m_\tau}f\!\left(T-\frac{m}{\tau}\right)
-\int_0^T f(s)\,ds
\right|
\\
&\le
\sum_{m=1}^{m_\tau}
\int_{x_m}^{x_m+1/\tau}|f(x_m)-f(s)|\,ds
+
\left|\int_{0}^{T-m_\tau/\tau}f(s)\,ds\right|.
\end{align*}
Since $f\in C^\alpha([0,T])$, by \eqref{eq:holder-seminorm}, for
$s\in[x_m,x_m+1/\tau]$,
\[
  |f(x_m)-f(s)|
  \le [f]_{C^\alpha([0,T])}|s-x_m|^\alpha
  \le [f]_{C^\alpha([0,T])}\tau^{-\alpha},
\]
and hence,
\[
\sum_{m=1}^{m_\tau}
\int_{x_m}^{x_m+1/\tau}|f(x_m)-f(s)|\,ds
\le
m_\tau\cdot \tau^{-1-\alpha}[f]_{C^\alpha([0,T])}
\le C_T\|f\|_{C^\alpha([0,T])}\tau^{-\alpha}.
\]
Also, since $0\le T-m_\tau/\tau\le 1/\tau$,
\[
\left|\int_{0}^{T-m_\tau/\tau}f(s)\,ds\right|
\le
\frac{1}{\tau}\|f\|_\infty
\le
C_T\|f\|_{C^\alpha([0,T])}\tau^{-1}
\le
C_T\|f\|_{C^\alpha([0,T])}\tau^{-\alpha},
\]
because $\alpha\in(0,1)$. Thus
\begin{equation}\label{eq:regular-riemann-final}
  \left|
    \frac{\gamma}{\tau}\sum_{m=1}^{m_\tau}f\!\left(T-\frac{m}{\tau}\right)
    -\gamma\int_0^T f(s)\,ds
  \right|
  \le C_T\|f\|_{C^{\alpha}([0,T])}\tau^{-\alpha}.
\end{equation}

\medskip
\noindent
\emph{Step 2: weakly singular Riemann sum.}
Set
$g(u):=f(T-u)$, $u\in[0,T]$.
Then $g\in C^{\alpha}([0,T])$ and
\[
  \|g\|_{C^{\alpha}([0,T])}\le C\|f\|_{C^{\alpha}([0,T])}.
\]
By \eqref{eq:tail-phi-explicit-final}, the term $m=1$ must be treated separately, since
$\overline\varphi(0)=1$,
$\overline\varphi(j)=\frac{1}{\Gamma(1-\alpha)}j^{-\alpha}$,
$j\ge1$.
Hence,
\begin{equation}\label{eq:singular-sum-decomp-final}
\begin{aligned}
  \tau^{\alpha-1}\sum_{m=1}^{m_\tau}\overline\varphi(m-1)
  f\!\left(T-\frac{m}{\tau}\right)
  =
  \tau^{\alpha-1}g\!\left(\frac1\tau\right)
  +
  \frac{1}{\Gamma(1-\alpha)}\frac1\tau
  \sum_{j=1}^{m_\tau-1}\left(\frac{j}{\tau}\right)^{-\alpha}
  g\!\left(\frac{j+1}{\tau}\right).
\end{aligned}
\end{equation}
We compare this with
\[
  I^{1-\alpha}f(T)
  =
  \frac{1}{\Gamma(1-\alpha)}\int_0^T u^{-\alpha}g(u)\,du.
\]

The leading error is most transparent after writing
\[
  g(u)=g(0)+r(u),
  \qquad
  |r(u)|\le [g]_{C^\alpha([0,T])}u^\alpha.
\]
For the constant part, set
\begin{align*}
  A_\tau(T)
  &:=
  \tau^{\alpha-1}
  \left(
    1+\frac{1}{\Gamma(1-\alpha)}
      \sum_{j=1}^{m_\tau-1}j^{-\alpha}
  \right)
  -
  \frac{T^{1-\alpha}}{\Gamma(2-\alpha)}.
\end{align*}
The Euler expansion
$\sum_{j=1}^{m-1}j^{-\alpha}
  =
  \frac{m^{1-\alpha}}{1-\alpha}
  +
  \zeta(\alpha)
  +
  \mathcal O(m^{-\alpha})$
where $\zeta$ denotes the analytically continued Riemann zeta function,
and $|m_\tau/\tau-T|\le\tau^{-1}$ imply
\begin{equation}\label{eq:baseline-constant-part}
  A_\tau(T)
  =
  c_1(\alpha)\tau^{-(1-\alpha)}
  +
  \mathcal O_T(\tau^{-\alpha}),
  \qquad
  c_1(\alpha):=
  1+\frac{\zeta(\alpha)}{\Gamma(1-\alpha)}.
\end{equation}
Here the terms of order $\tau^{-1}$ have been absorbed into
$\mathcal O_T(\tau^{-\alpha})$.

It remains to estimate the part containing $r$. On the first cell,
\[
  \left|
    \tau^{\alpha-1}r\!\left(\frac1\tau\right)
    -
    \frac{1}{\Gamma(1-\alpha)}
    \int_0^{1/\tau}u^{-\alpha}r(u)\,du
  \right|
  \le C_\alpha[g]_{C^\alpha([0,T])}\tau^{-1}.
\]
For $1\le j\le m_\tau-1$, let
\[
  E_j^r:=
  \left|
    \frac1\tau\left(\frac{j}{\tau}\right)^{-\alpha}
      r\!\left(\frac{j+1}{\tau}\right)
    -
    \int_{j/\tau}^{(j+1)/\tau}u^{-\alpha}r(u)\,du
  \right|.
\]
The kernel rectangle error and the H\"older variation of $r$ give,
respectively,
\begin{align*}
  &\left|
    \frac1\tau\left(\frac{j}{\tau}\right)^{-\alpha}
    -
    \int_{j/\tau}^{(j+1)/\tau}u^{-\alpha}\,du
  \right|
  \left|r\!\left(\frac{j+1}{\tau}\right)\right|
  \le
  C_\alpha[g]_{C^\alpha([0,T])}\tau^{-1}j^{-1},\\
  &\int_{j/\tau}^{(j+1)/\tau}
    u^{-\alpha}
    \left|r\!\left(\frac{j+1}{\tau}\right)-r(u)\right|\,du
  \le
  C_\alpha[g]_{C^\alpha([0,T])}\tau^{-1}j^{-\alpha}.
\end{align*}
Consequently,
\[
  \sum_{j=1}^{m_\tau-1}E_j^r
  \le
  C_{\alpha,T}\|g\|_{C^\alpha([0,T])}
  \left(\tau^{-1}\log(1+\tau)+\tau^{-\alpha}\right)
  \le
  C_{\alpha,T}\|g\|_{C^\alpha([0,T])}\tau^{-\alpha}.
\]
Finally, because $u^{-\alpha}|r(u)|
\le[g]_{C^\alpha([0,T])}$, the omitted terminal interval contributes at
most $C\|g\|_{C^\alpha([0,T])}\tau^{-1}$. Combining these estimates with
\eqref{eq:baseline-constant-part} gives
\begin{equation}\label{eq:singular-riemann-refined}
\begin{aligned}
  &\left|
    \tau^{\alpha-1}\sum_{m=1}^{m_\tau}\overline\varphi(m-1)
    f\!\left(T-\frac{m}{\tau}\right)
    -
    I^{1-\alpha}f(T)
  \right|
  \\
  &\le
  C_T\|f\|_{C^{\alpha}([0,T])}\tau^{-\alpha}
  +
  |c_1(\alpha)|
  \|f\|_{C^{\alpha}([0,T])}\tau^{-(1-\alpha)}.
\end{aligned}
\end{equation}

It remains to reinsert the prefactor $a_\tau$. First,
\begin{align*}
\tau^{\alpha-1}\sum_{m=1}^{m_\tau}\overline\varphi(m-1)
&=
\tau^{\alpha-1}\left(
1+\frac{1}{\Gamma(1-\alpha)}\sum_{j=1}^{m_\tau-1}j^{-\alpha}
\right)
\\
&\le
C_T\left(\tau^{\alpha-1}+\tau^{\alpha-1}m_\tau^{1-\alpha}\right)
\le C_T.
\end{align*}
Therefore,
\begin{align*}
  &\left|
    \xi_0 a_\tau\tau^{\alpha-1}
    \sum_{m=1}^{m_\tau}\overline\varphi(m-1)
    f\!\left(T-\frac{m}{\tau}\right)
    -
    \xi_0\tau^{\alpha-1}
    \sum_{m=1}^{m_\tau}\overline\varphi(m-1)
    f\!\left(T-\frac{m}{\tau}\right)
  \right|
  \\
  &\le
  |\xi_0|\,|a_\tau-1|\,\tau^{\alpha-1}
  \sum_{m=1}^{m_\tau}\overline\varphi(m-1)
  \left|f\!\left(T-\frac{m}{\tau}\right)\right|
\le  C_T\|f\|_{\infty}\tau^{-\alpha}.
\end{align*}
Together with \eqref{eq:singular-riemann-refined}, this yields
\begin{equation}\label{eq:singular-riemann-final}
\begin{aligned}
  &\left|
    \xi_0 a_\tau\tau^{\alpha-1}
    \sum_{m=1}^{m_\tau}\overline\varphi(m-1)
    f\!\left(T-\frac{m}{\tau}\right)
    -
    \xi_0 I^{1-\alpha}f(T)
  \right|
  \\
  &\le
  C_T\|f\|_{C^{\alpha}([0,T])}\tau^{-\alpha}
  +
  |\xi_0|\,|c_1(\alpha)|
  \|f\|_{C^{\alpha}([0,T])}\tau^{-(1-\alpha)}.
\end{aligned}
\end{equation}

Finally, combining \eqref{eq:baseline-split-final}, \eqref{eq:Rtau-bound-final},
\eqref{eq:regular-riemann-final}, and \eqref{eq:singular-riemann-final}, we obtain
\begin{align*}
  &\left|
    \frac{\varepsilon_\tau}{\mu}
    \sum_{m=1}^{m_\tau}\hat\mu_\tau(m)
    f\!\left(T-\frac{m}{\tau}\right)
    -
    \left[
      \gamma\int_0^T f(s)\,ds
      +
      \xi_0 I^{1-\alpha}f(T)
    \right]
  \right|\\
  \le&
  C_T\|f\|_{C^{\alpha}([0,T])}\tau^{-\alpha}
  +
  C_{\mathrm{base}}(\alpha)\|f\|_{C^{\alpha}([0,T])}\tau^{-(1-\alpha)},
\end{align*}
with
$C_{\mathrm{base}}(\alpha)
  :=
  |\xi_0|
  \left|1+\frac{\zeta(\alpha)}{\Gamma(1-\alpha)}\right|$.
This proves \eqref{eq:baseline-riemann}.
\end{proof}


\subsection{Proof of Proposition~\ref{prop:transform-bound-fixed}}\label{app:proof-transform-bound-fixed}

\begin{proof}
\noindent
\emph{Step 1: cumulant-exponent estimate.}
By Proposition~\ref{prop:Volterra-rate-fixed},
\[
  \sup_{0\le n\le m_\tau}\sup_{z\in\Gamma_{\eta,R}}
  \left|H_n^\tau(z)-h(t_n,z)\right|
  \le
  C_{\eta,R,T}\tau^{-\alpha}+{C_{\mathrm{Ric}}(\alpha)}\,\tau^{-(1-\alpha)},
\]
where $t_n:=n/\tau$.
Using
\[
  \Psi^\tau(z,T)
  =
  \frac{\theta}{\mu}(1-a_\tau)
  \sum_{m=1}^{m_\tau}\hat\mu_\tau(m)\,H_{m_\tau-m}^\tau(z),
\]
we write
\begin{align}
  \left|
    \Psi^\tau(z,T)-\Psi(z,T)
  \right|
  &\le
  \frac{\theta}{\mu}(1-a_\tau)
  \sum_{m=1}^{m_\tau}\hat\mu_\tau(m)
  \left|
    H_{m_\tau-m}^\tau(z)-h\!\left(\frac{m_\tau-m}{\tau},z\right)
  \right|
  \nonumber
  \\
  &\qquad
  +
  \frac{\theta}{\mu}(1-a_\tau)
  \sum_{m=1}^{m_\tau}\hat\mu_\tau(m)
  \left|
    h\!\left(\frac{m_\tau-m}{\tau},z\right)-h\!\left(T-\frac{m}{\tau},z\right)
  \right|
  \nonumber
  \\
  &\qquad\qquad
  +
  \left|
    \frac{\theta}{\mu}(1-a_\tau)
    \sum_{m=1}^{m_\tau}\hat\mu_\tau(m)
    h\!\left(T-\frac{m}{\tau},z\right)
    -
    \Psi(z,T)
  \right|.\label{eq:exponent-error-split}
\end{align}

For the first term in \eqref{eq:exponent-error-split},
Proposition~\ref{prop:Volterra-rate-fixed} gives
\[
  \left|
    H_{m_\tau-m}^\tau(z)-h\!\left(\frac{m_\tau-m}{\tau},z\right)
  \right|
  \le
  C_{\eta,R,T}\tau^{-\alpha}+{C_{\mathrm{Ric}}(\alpha)}\,\tau^{-(1-\alpha)}
\]
uniformly in $m$ and $z$. Moreover,
Proposition~\ref{prop:baseline-riemann}, applied to the constant function
$f\equiv1$, shows that, for some $\tau_0$,
\begin{equation}\label{eq:baseline-mass-bound}
  M_{\mathrm{base}}
  :=\sup_{\tau\ge\tau_0: \tau\text{ admissible}}
  \frac{1-a_\tau}{\mu}\sum_{m=1}^{m_\tau}\hat\mu_\tau(m)
  <\infty.
\end{equation}
Hence the first term in \eqref{eq:exponent-error-split} is bounded by
\[
  C_{\eta,R,T}\tau^{-\alpha}
  +\theta M_{\mathrm{base}}C_{\mathrm{Ric}}(\alpha)
   \tau^{-(1-\alpha)}.
\]

For the second term in \eqref{eq:exponent-error-split}, note that
\[
  \left|
    \frac{m_\tau-m}{\tau}-\left(T-\frac{m}{\tau}\right)
  \right|
  =
  \left|\frac{m_\tau}{\tau}-T\right|
  \le \frac1\tau.
\]
Since $h(\cdot,z)$ is uniformly $\alpha$-H\"older continuous on $[0,T]$ by
Proposition~\ref{prop:riccati-regularity}, we obtain
\[
  \left|
    h\!\left(\frac{m_\tau-m}{\tau},z\right)-h\!\left(T-\frac{m}{\tau},z\right)
  \right|
  \le C_{\eta,R,T}\tau^{-\alpha}
\]
uniformly in $m$ and $z$. Using again \eqref{eq:baseline-mass-bound},
it follows that the second term in \eqref{eq:exponent-error-split} is bounded by
$C_{\eta,R,T}\tau^{-\alpha}$.

For the third term in \eqref{eq:exponent-error-split}, we apply
Proposition~\ref{prop:baseline-riemann} with
$f(s)=h(s,z)$. By Proposition~\ref{prop:riccati-regularity},
\[
  M_{h,\alpha}
  :=\sup_{z\in\Gamma_{\eta,R}}
    \|h(\cdot,z)\|_{C^\alpha([0,T])}
  <\infty.
\]
Using $V_0=\theta\xi_0$, the limiting expression in
Proposition~\ref{prop:baseline-riemann} is
\[
  \theta\left(
    \gamma\int_0^T h(s,z)\,ds
    +\xi_0 I^{1-\alpha}h(\cdot,z)(T)
  \right)
  =\Psi(z,T).
\]
Consequently,
\begin{align*}
&\left|
    \frac{\theta}{\mu}(1-a_\tau)
    \sum_{m=1}^{m_\tau}\hat\mu_\tau(m)
    h\left(T-\frac{m}{\tau},z\right)
    -
    \Psi(z,T)
  \right|
\\
&\le
  C_{\eta,R,T}\tau^{-\alpha}
  +
  \theta M_{h,\alpha}C_{\mathrm{base}}(\alpha)
  \tau^{-(1-\alpha)},
\end{align*}
uniformly on $\Gamma_{\eta,R}$.

Combining the bounds on the three terms in
\eqref{eq:exponent-error-split}, and enlarging
$C_{\eta,R,T}$ if necessary, we obtain
\[
  \sup_{z\in\Gamma_{\eta,R}}
  \left|\Psi^\tau(z,T)-\Psi(z,T)\right|
  \le
  C_{\eta,R,T}\tau^{-\alpha}
  +
  C_{\mathrm{exp}}(\alpha)\tau^{-(1-\alpha)},
\]
where
$C_{\mathrm{exp}}(\alpha):=
\theta M_{\mathrm{base}}C_{\mathrm{Ric}}(\alpha)
+\theta M_{h,\alpha}C_{\mathrm{base}}(\alpha)$.
As throughout this section, the dependence on $\alpha$ is written explicitly,
whereas dependence on the fixed strip, maturity, and model parameters is
suppressed. The same convention is used for $C_\Phi(\alpha)$ below.
This proves \eqref{eq:log-transform-bound}.

\medskip
\noindent
\emph{Step 2: transform estimate.}
Recall that $\Phi^\tau(z,T)=e^{\Psi^\tau(z,T)}$ and
$\phi(z,T)=e^{\Psi(z,T)}$. Set $x:=\Psi^\tau(z,T)$ and
$y:=\Psi(z,T)$. By \eqref{eq:log-transform-bound}, uniformly on
$\Gamma_{\eta,R}$,
\[
 |x-y|\le C_{\eta,R,T}\tau^{-\alpha}
 +C_{\mathrm{exp}}(\alpha)\tau^{-(1-\alpha)}.
\]
Assumption~\ref{ass:riccati-bound} and the integral formula for $\Psi$ give
$B_\Psi:=\sup_{z\in\Gamma_{\eta,R}}|\Psi(z,T)|<\infty$. Set
$M_\Psi:=e^{1+B_\Psi}$.
Since the right-hand side of \eqref{eq:log-transform-bound} tends to zero,
$|x-y|\le1$ uniformly on $\Gamma_{\eta,R}$ for all sufficiently large $\tau$.
Thus $|e^{tx+(1-t)y}|\le M_\Psi$ for $t\in[0,1]$. The identity
$e^x-e^y=(x-y)\int_0^1e^{tx+(1-t)y}\,dt$ therefore gives
$|\Phi^\tau(z,T)-\phi(z,T)|\le M_\Psi|x-y|$. Consequently,
\[
  |\Phi^\tau(z,T)-\phi(z,T)|
  \le
  C_{\eta,R,T}\tau^{-\alpha}
  +
  C_{\Phi}(\alpha)\tau^{-(1-\alpha)},
\]
uniformly on $\Gamma_{\eta,R}$, where one may take
$C_{\Phi}(\alpha):=M_\Psi C_{\mathrm{exp}}(\alpha)$.
This proves \eqref{eq:transform-bound-fixed}.
\end{proof}

\subsection{Proof of Theorem~\ref{thm:full-call-convergence}}
\label{app:proof-full-call-price}

\begin{proof}
Set
$c_X(k):=\E\left[\left(e^X-e^k\right)^+\right]$
and
$g_X(k):=e^{(\eta-1)k}c_X(k)$.
For every $x\in\R$, direct integration gives
\begin{align*}
\int_{\R}e^{(\eta-1)k}(e^x-e^k)^+\,dk
 =\int_{-\infty}^{x}e^{(\eta-1)k}(e^x-e^k)\,dk
 =e^{\eta x}\left(\frac{1}{\eta-1}-\frac{1}{\eta}\right)
 =\frac{e^{\eta x}}{\eta(\eta-1)}.
\end{align*}
Applying Tonelli's theorem to the nonnegative integrand then yields
\[
 \|g_X\|_{L^1(\R)}
 =\frac{\E[e^{\eta X}]}{\eta(\eta-1)}<\infty.
\]
The map $k\mapsto c_X(k)$ is continuous by dominated convergence, since
$0\le(e^X-e^k)^+\le e^X$ and $\E[e^X]<\infty$ follows from the assumed
$\eta$-moment when $\eta>1$.
Hence $g_X$ is continuous.

Use the Fourier-transform convention
$\widehat g_X(\xi)=\int_\R e^{\ii\xi k}g_X(k)\,dk$.
Fubini's theorem is justified by the preceding $L^1$ calculation, and it yields
\begin{align*}
 \widehat g_X(\xi)
 =\E\!\left[
   \int_{-\infty}^{X}e^{(\eta-1+\ii\xi)k}(e^X-e^k)\,dk
 \right]
=\frac{\E[e^{(\eta+\ii\xi)X}]}
 { (\eta+\ii\xi)(\eta-1+\ii\xi)}.
\end{align*}
Moreover,
\[
 |\widehat g_X(\xi)|
 \le
 \frac{\E[e^{\eta X}]}
 {|\eta+\ii\xi|\,|\eta-1+\ii\xi|},
\]
and the right-hand side is integrable on $\R$: it is bounded near the origin
and is $\mathcal O(|\xi|^{-2})$ at infinity.  Thus both $g_X$ and
$\widehat g_X$ belong to $L^1(\R)$. Fourier inversion at every continuity
point gives
\[
 g_X(k)=\frac1{2\pi}\int_\R e^{-\ii\xi k}\widehat g_X(\xi)\,d\xi.
\]
Multiplying by $e^{-(\eta-1)k}$ and using
$\E[(S_0e^X-K)^+]/S_0=c_X(k)$ gives
\begin{equation}\label{eq:full-carr-madan-general}
 \frac{\E[(S_0e^X-K)^+]}{S_0}
 =
 \frac{e^{-(\eta-1)k}}{2\pi}
 \int_{\R}e^{-\ii\xi k}
 \frac{\E[e^{(\eta+\ii\xi)X}]}
      {(\eta+\ii\xi)(\eta-1+\ii\xi)}\,d\xi.
\end{equation}
Formula~\eqref{eq:full-carr-madan-general} applies to arbitrary laws with a
finite $\eta$-moment, including the atomic finite-$\tau$ model. It is the
direct-inversion form of the damped call transform introduced in
\cite{carrmadan1999fft}.

Proposition~\ref{prop:transform-bound-fixed}, applied on
$\Gamma_{\eta,1}$ at $z=\eta$, gives
$\Phi^\tau(\eta,T)\longrightarrow\phi(\eta,T)$.
Thus there are $\tau_\eta\ge1$ and $M_{\eta,T}<\infty$ such that
$\phi(\eta,T)\le M_{\eta,T}$,
$\sup_{\tau\ge\tau_\eta}\Phi^\tau(\eta,T)\le M_{\eta,T}$.
Apply \eqref{eq:full-carr-madan-general} with $X=P_T^\tau$ and $X=P_T$.
For every fixed $R\ge1$ and $\tau\ge\tau_\eta$,
\begin{equation}\label{two:integrals:to:bound}
\begin{aligned}
 \frac{|C^\tau(T,K)-C(T,K)|}{S_0}
 \le{}&
 \frac{e^{-(\eta-1)k}}{2\pi}
 \int_{|\xi|\le R}
 \frac{|\Phi^\tau(\eta+\ii\xi,T)-\phi(\eta+\ii\xi,T)|}
 {|\eta+\ii\xi|\,|\eta-1+\ii\xi|}\,d\xi
 \\
 &+
 \frac{e^{-(\eta-1)k}}{2\pi}
 \int_{|\xi|>R}
 \frac{|\Phi^\tau(\eta+\ii\xi,T)|+|\phi(\eta+\ii\xi,T)|}
 {|\eta+\ii\xi|\,|\eta-1+\ii\xi|}\,d\xi.
\end{aligned}
\end{equation}
For fixed $R$, set
$A_{\eta,R}
  :=\int_{|\xi|\le R}
  \frac{d\xi}{|\eta+\ii\xi|\,|\eta-1+\ii\xi|}
  <\infty$.
The first integral in \eqref{two:integrals:to:bound} is bounded by
\[
  A_{\eta,R}
  \sup_{|\xi|\le R}
  |\Phi^\tau(\eta+\ii\xi,T)-\phi(\eta+\ii\xi,T)|,
\]
apart from the fixed prefactor, and therefore tends to zero by
Proposition~\ref{prop:transform-bound-fixed}. For the second integral in \eqref{two:integrals:to:bound},
\[
 |\Phi^\tau(\eta+\ii\xi,T)|\le\Phi^\tau(\eta,T),
 \qquad
 |\phi(\eta+\ii\xi,T)|\le\phi(\eta,T),
\]
and
$|\eta+\ii\xi|\,|\eta-1+\ii\xi|\ge|\xi|^2$.  Its contribution is therefore
bounded by
$\frac{2e^{-(\eta-1)k}M_{\eta,T}}{\pi R}$.
First choose $R$ large and then let $\tau\to\infty$ with this window fixed.
This proves \eqref{eq:full-call-convergence}.

Finally, if $k$ lies in a compact set $\mathcal K$, then
$e^{-(\eta-1)k}$ is uniformly bounded.  Both the fixed-window term and the
tail bound are therefore uniform on $\mathcal K$, proving locally uniform
convergence in log-moneyness.
\end{proof}


\end{document}